\documentclass[reqno]{amsart}
\usepackage[margin = 1.1in]{geometry}
\usepackage{amsmath, amssymb, amsthm, fancyhdr, verbatim, graphicx, relsize, stmaryrd}
\usepackage{enumerate}
\usepackage[all]{xy}
\usepackage[usenames,dvipsnames]{xcolor}
\usepackage{mathrsfs}
\usepackage{tikz}
\usepackage{tikz-cd}
\usetikzlibrary{calc}
\usepackage[OT2, T1]{fontenc}
\usepackage{framed, hyperref}
\usepackage[titletoc]{appendix}
\usepackage{bbm}
\usepackage{xspace}

\usepackage{adjustbox}

\numberwithin{equation}{subsection}

\DeclareSymbolFont{cyrletters}{OT2}{wncyr}{m}{n}
\DeclareMathSymbol{\Sha}{\mathalpha}{cyrletters}{"58}

\newcommand{\F}{\mathbf{F}}

\newcommand{\CC}{\mathbf{C}}
\newcommand{\G}{\mathbf{G}}
\newcommand{\tr}[0]{\operatorname{tr}}
\newcommand{\wt}[1]{\widetilde{#1}}

\newcommand{\Z}{\mathbf{Z}}

\newcommand{\mf}[1]{\mathfrak{#1}}

\newcommand{\Gal}{\operatorname{Gal}}

\newcommand{\ul}[1]{\underline{#1}}
\newcommand{\ol}[1]{\overline{#1}}
\newcommand{\wh}[1]{\widehat{#1}}
\newcommand{\mbb}[1]{\mathbb{#1}}

\newcommand{\Cal}[1]{\mathcal{#1}}

\DeclareMathOperator{\et}{\text{\'{e}t}}

\newcommand{\co}{\colon}
\newcommand{\mrm}[1]{\mathrm{#1}}
\newcommand{\msf}[1]{\mathsf{#1}}
\newcommand{\bs}{\backslash}
\newcommand{\TT}{\mathbb{T}}
\newcommand{\dotimes}{\stackrel{\mrm{L}}\otimes}
\newcommand{\ld}{{}^L}
\newcommand{\lc}{{}^c}

\newcommand{\bbm}[1]{\mathbbm{#1}}
\newcommand{\bu}{\bullet}
\newcommand{\der}{\mathrm{der}}

\newcommand{\surj}{\twoheadrightarrow}
\newcommand{\inj}{\hookrightarrow}

\newcommand{\colim}{\varinjlim}
\newcommand{\limit}{\varprojlim}

\newcommand{\OO}{\mbb{O}}

\newcommand{\cHom}{\Cal{H}om}

\newcommand{\cRHom}{\Cal{R}\Cal{H}om}
\newcommand{\Rlim}{\mrm{R}\hspace{-0.1cm}\varprojlim}
\newcommand{\dimg}{\mrm{dim.}\,\mrm{trg}\,}

\DeclareMathOperator{\bd}{bd}

\newcommand{\ellt}{^{(\ell)}}
\newcommand{\iellt}{^{(\ell^{-1})}}

\newcommand{\PP}{\mathbb{P}}

\newcommand\cB{\mathcal{B}}
\newcommand\cC{\mathcal{C}}
\newcommand\cD{\mathcal{D}}
\newcommand\cE{\mathcal{E}}
\newcommand\cF{\mathcal{F}}
\newcommand\cG{\mathcal{G}}

\newcommand\cK{\mathcal{K}}
\newcommand\cL{\mathcal{L}}

\newcommand\cO{\mathcal{O}}
\newcommand\cP{\mathcal{P}}

\newcommand\cS{\mathcal{S}}
\newcommand\cT{\mathcal{T}}

\newcommand\cY{\mathcal{Y}}
\newcommand\cZ{\mathcal{Z}}

\newcommand\sH{\mathscr{H}}

\DeclareMathOperator{\GL}{GL}

\DeclareMathOperator{\Frob}{Frob}

\DeclareMathOperator{\Hom}{Hom}

\DeclareMathOperator{\Ima}{Im\,}

\DeclareMathOperator{\ord}{ord}
\DeclareMathOperator{\Aut}{Aut}
\DeclareMathOperator{\Rep}{Rep}
\DeclareMathOperator{\Nm}{Nm}
\DeclareMathOperator{\Spec}{Spec\,}

\DeclareMathOperator{\End}{End}
\DeclareMathOperator{\Isom}{Isom}

\DeclareMathOperator{\ad}{ad}

\DeclareMathOperator{\Br}{Br}

\DeclareMathOperator{\Res}{Res}

\DeclareMathOperator{\Div}{Div}

\DeclareMathOperator{\Stab}{Stab}
\DeclareMathOperator{\Bun}{Bun}

\DeclareMathOperator{\ch}{ch\,}
\DeclareMathOperator{\Id}{Id}
\DeclareMathOperator{\Ad}{Ad}
\DeclareMathOperator{\Gr}{Gr}

\DeclareMathOperator{\pt}{pt}

\DeclareMathOperator{\Hck}{Hck}
\DeclareMathOperator{\Fix}{Fix}

\DeclareMathOperator{\Sht}{Sht}

\DeclareMathOperator{\Sat}{Sat}

\DeclareMathOperator{\Perf}{Perf}

\DeclareMathOperator{\Psm}{Psm}

\DeclareMathOperator{\Exc}{Exc}

\DeclareMathOperator{\Shv}{Shv}

\DeclareMathOperator{\alg}{alg}

\DeclareMathOperator{\pr}{pr}
\DeclareMathOperator{\dR}{dR}
\DeclareMathOperator{\supp}{supp}
\DeclareMathOperator{\Iw}{Iw}

\DeclareMathOperator{\cofib}{cofib}
\DeclareMathOperator{\Real}{Real}

\newcommand{\cHck}{\mathcal{H}\mrm{ck}}
\DeclareMathOperator{\ULA}{ULA}

\DeclareMathOperator{\Spd}{Spd}

\DeclareMathOperator{\Spa}{Spa}

\DeclareMathOperator{\Tot}{Tot}
\DeclareMathOperator{\fib}{fib}

\DeclareMathOperator{\conv}{conv}
\DeclareMathOperator{\twi}{tw}
\DeclareMathOperator{\sm}{sm}
\DeclareMathOperator{\FS}{FS}
\DeclareMathOperator{\cInd}{c-Ind}

\DeclareMathOperator{\Flat}{Flat}

\DeclareMathOperator{\CT}{CT}

\DeclareMathOperator{\red}{red}
\DeclareMathOperator{\GM}{GM}

\DeclareMathOperator{\TV}{TV}

\DeclareMathOperator{\Fun}{Fun}
\DeclareMathOperator{\Loc}{Loc}

\DeclareMathOperator{\cd}{cd}

\newcommand{\RGamma}{\mathrm{R}\Gamma}

\newcommand{\bdrp}{\mathrm{B}^+_{\mathrm{dR}}}
\newcommand{\tw}[1]{\langle #1 \rangle}

\newcommand{\GG}{\mathbb{G}}
\newcommand{\UU}{\mathbb{U}}
\DeclareMathOperator{\Fr}{Fr}

\newcommand{\rE}{\ensuremath{\mathrm{E}}\xspace}

\newcommand{\rH}{\ensuremath{\mathrm{H}}\xspace}

\newcommand{\rR}{\ensuremath{\mathrm{R}}\xspace}

\newcommand{\rT}{\ensuremath{\mathrm{T}}\xspace}

\newcommand{\bT}{\mathbf{T}}
\newcommand{\bW}{\mathbf{W}}

\newcommand{\brE}{\breve{E}}
\newcommand{\br}{\mathrm{br}}

\newcommand{\cbr}{br}

\newcommand{\wcbr}{\wt{\cbr}}

\newcommand{\chG}{\check{G}}
\newcommand{\chH}{\check{H}}
\newcommand{\chT}{\check{T}}

\newcommand{\chphi}{\check{\phi}}
\newcommand{\chpsi}{\check{\psi}}

\newcommand{\Dula}[3]{D_{(L^+{#1})}^{\ULA}(\Gr_{{#1}, S/\Div^{#3}_X}; #2)^{\bd}}

\newcommand{\DulacHck}[3]{D_{\et}^{\ULA}(\cHck_{{#1}, S/\Div^{#3}_X}; #2)^{\bd}}

\newtheorem{thm}{Theorem}[subsection]
\newtheorem{lemma}[thm]{Lemma}
\newtheorem{prop}[thm]{Proposition}
\newtheorem{cor}[thm]{Corollary}

\newtheorem{conj}[thm]{Conjecture}

\theoremstyle{remark}
\newtheorem{remark}[thm]{Remark} 
 
\newtheorem{defn}[thm]{Definition}

\newtheorem{example}[thm]{Example}

\makeatletter
\def\th@remark{%
  \thm@headfont{\bfseries}%
  \normalfont 
  \thm@preskip \thm@preskip 
  \thm@postskip\thm@preskip
}
\def\imod#1{\allowbreak\mkern5mu({\operator@font mod}\,\,#1)}
\makeatother

\title[Modular functoriality in the Local Langlands Correspondence]{Modular functoriality in the Local Langlands Correspondence}

\author{Tony Feng}

\begin{document}

\begin{abstract}
We develop a theory of Smith--Treumann localization in the context of Fargues-Scholze's Geometrization of the Local Langlands Correspondence. We then apply this theory to prove some conjectures of Treumann-Venkatesh concerning mod-$\ell$ Local Langlands functoriality between a reductive group $G$ and its fixed subgroup under an order $\ell$ automorphism. As another application, we explicitly calculate the Fargues-Scholze parameters of certain mod-$\ell$ toral representations. 
\end{abstract}
\maketitle

\tableofcontents

	\section{Introduction}

\subsection{Modular Langlands functoriality} Let $E$ be a local field of residue characteristic $p$ and $G$ be a reductive group over $E$. For simplicity, we assume for now that $G$ is split; beyond the Introduction, the text always treats general $G$. The Local Langlands Correspondence predicts, vaguely speaking, that the smooth representations of $G(E)$ over a field $F$ are controlled by the \emph{Langlands dual group} $\chG$ over $F$. 

For example, letting $W_E \subset \Gal(E^s/E)$ be the Weil group of $E$, the Local Langlands Correspondence predicts that there should be a natural parametrization of irreducible smooth representations of $G(E)$ over $F$ by ``$L$-parameters'', which are homomorphisms $W_E \rightarrow \chG(F)$ up to conjugacy. The existence of such a parametrization has surprising implications for representation theory: for example, a homomorphism $\chpsi \co \chH \rightarrow \chG$ of dual groups induces an obvious map 
\[
\{\text{$L$-parameters for $\chH$}\} \xrightarrow{\chpsi_*} \{\text{$L$-parameters for $\chG$}\}
\]
and therefore suggests some \emph{Langlands functoriality} operation from (packets of) irreducible representations of $H(E)$ to (packets of) irreducible representations of $G(E)$. In practice, it is usually difficult to construct such operations directly, or to describe them explicitly in representation-theoretic terms. 

In this paper we investigate Langlands functoriality for a specific class of dual homomorphisms $\chpsi$ that was identified by Treumann-Venkatesh in \cite{TV}. It considers the situation where the reductive group $H$ arises as the fixed points of an order $\ell$ automorphism $\sigma$ of another reductive group $G$, where $\ell$ is a prime different from $p$. Furthermore, we take $k$ to be an algebraically closed field of characteristic $\ell$. In this situation (and under some hypotheses), Treumann-Venkatesh constructed a dual homomorphism $\chpsi \co \chH \rightarrow \chG$ over $k$, with a specific property in terms of the Satake isomorphism that we will discuss later. We refer to this situation as \emph{modular functoriality}, because its construction depends on special features of modular arithmetic (and this paper depends in turn on special features of modular representation theory). By contrast, Langlands' original conjectures were focused on the case where the coefficient field has characteristic zero. We note that in many examples the map $\chpsi$ does not lift to characteristic zero, and the resulting functoriality is truly specific to modular arithmetic. 

The story extends to non-split groups. In that case, the dual group $\chG$ should be augmented to the $L$-group $\ld G \cong \wh{G} \rtimes W_E$, where the action of $W_E$ on $\wh{G}$ reflects the twisting of $G$ over $E$ relative to its split form. This generalization is necessary to treat some of the most interesting examples. 

\begin{example}
A familiar example is cyclic base change, where $H$ is a reductive group over $E$ and $G = \Res_{E'/E}(H_{E'})$ for a cyclic $\ell$-extension $E'/E$. Let $\sigma$ be a generator of $\Gal(E'/E)$ acting on $G$ in the natural way, so that $G^{\sigma} = H$. This particular example lifts to characteristic $0$, but many do not; several examples are tabulated in the arXiv version of Treumann-Venkatesh's paper \cite{TVarxiv}. We also consider here some interesting examples that are ruled out by the hypotheses of \cite{TVarxiv, TV}: a useful one is where $\sigma$ is conjugation by a strongly regular element of $G(E)$, in which case $H$ is a (not necessarily split!) maximal torus. Note that the work of Treumann-Venkatesh only treats examples where $G$ is simply connected and $H$ is semisimple.
\end{example}

 \subsection{Conjectures of Treumann-Venkatesh} We describe some conjectures, essentially due to Treumann-Venkatesh, concerning modular Langlands functoriality. We again restrict our attention to the case where $G$ and $H$ are split, for simplicity. We let $\sH(G,K)$ be the spherical Hecke algebra of $G$ (with respect to a hyperspecial maximal compact subgroup $K$ stable under $\sigma$) with coefficients in $k$. Then the Satake isomorphism supplies a $k$-algebra isomorphism 
\[
\sH(G,K) \cong \cO(\chG \sslash \chG)
\]
Under some technical assumptions, Treumann-Venkatesh construct a $k$-algebra homomorphism $\br \co \sH(G, K) \rightarrow \sH(H, K^\sigma)$ which they call the \emph{(normalized) Brauer homomorphism}. (The construction uses in an essential way the assumption that char $k = \ell = \ord(\sigma)$.) This implies the existence of a commutative diagram 
\[
\begin{tikzcd}
\sH(G,K) \ar[r, "\sim","\text{Satake}"'] \ar[d, "\br"]  & \cO(\chG \sslash \chG) \ar[d, dashed] \\
\sH(H, K^\sigma) \ar[r, "\sim", "\text{Satake}"']   & \cO(\chH \sslash \chH)
\end{tikzcd}
\]
It is then natural to ask if the dashed map is induced by restriction along a homomorphism $\chpsi \co \chH \rightarrow \chG$. If so, then $\chpsi$ will be called a ``$\sigma$-dual homomorphism''. One of the main theorems of Treumann-Venkatesh \cite[\S 1.3]{TV} is that if $G$ is simply connected and $H$ is semisimple, then a $\sigma$-dual homomorphism exists, with three possible exceptions if $G$ has type $\rE_6$. In \cite[\S 1.4(iii)]{TVarxiv}, Treumann-Venkatesh pose the open problem of constructing a $\sigma$-dual homomorphism in full generality. Our first theorem addresses this question in full generality.

\begin{thm}\label{thm: intro sigma-dual} A $\sigma$-dual homomorphism $\chpsi \co \chH \rightarrow \chG$ exists. 
\end{thm} 

\begin{example} 
For applications, there are important cases beyond the hypotheses of Treumann--Venkatesh. For example, in \S \ref{sec: toral supercuspidals} we study a situation where $H$ is a torus, while Treumann--Venkatesh assume that $H$ is semisimple. 
\end{example}

The construction of the $\sigma$-dual homomorphism in \cite{TV} is a tour de force: it invokes classification theorems to tabulate all examples of modular functoriality, and analyzes them case-by-case. By contrast, the proof of Theorem \ref{thm: intro sigma-dual} is completely uniform across all cases, and makes no use of classification theorems. The idea is quite natural: instead of contemplating the dual groups, we try to \emph{categorify} the Brauer homomorphism $\br$ to a \emph{Brauer functor} $\wcbr$ from $\Sat(\Gr_G;k)$ to $\Sat(\Gr_H;k)$. After equipping this Brauer functor with a Tannakian structure, we obtain the $\sigma$-dual homomorphism for free from the Geometric Satake equivalence.\footnote{A subtle but important point found by Treumann-Venkatesh is that, when taking into account the full $L$-group, a $\sigma$-dual homomorphism may not exist with the ``usual'' notion of $L$-group due to Langlands. Treumann-Venkatesh suggest in \cite[\S 7.8]{TVarxiv} that this problem might be repaired by instead using the ``$c$-group'', which is the variant of the $L$-group that naturally comes out of the geometric Satake equivalence \cite[Remark 5.5.11]{Zhu17}. For non-archimedean local fields there is actually an isomorphism between the $L$-group and the $c$-group (possibly depending on a choice of square root of $q$ in $k$), but for the most canonical statements one needs to use the $c$-group.} 
\[
\begin{tikzcd}
\Sat(\Gr_G; k)  \ar[rr, "\sim","\text{Geom. Satake}"'] \ar[d, "\wcbr"]  & & \Rep_k(\chG) \ar[d, dashed] \\
\Sat(\Gr_H; k)  \ar[rr, "\sim", "\text{Geom. Satake}"']  & & \Rep_k( \chH)
\end{tikzcd}
\]
Although this idea is simple, its implementation is quite involved and will be elaborated upon later. 

We next turn to describe conjectures of Treumann-Venkatesh pertaining to the Local Langlands Correspondence. We write $\Z[\sigma]$ for the group ring of $\tw{\sigma} \cong \Z/\ell\Z$ and let $N := 1 + \sigma + \ldots + \sigma^{\ell-1} \in \Z[\sigma]$. If $\Pi$ is a representation of $G(E) \rtimes \tw{\sigma}$, then its \emph{Tate cohomology} groups $\rT^j(\Pi)$, for $j \in \Z/2\Z$, are defined as 
\begin{equation}\label{eq: Tate cohomology}
\rT^0 (\Pi) := \frac{\ker(1-\sigma \mid \Pi)}{\Ima(N \mid \Pi)} \hspace{1cm}
\rT^1 (\Pi) := \frac{\ker(N \mid \Pi)}{\Ima(1-\sigma \mid \Pi)}
\end{equation}
and the $G(E)$-action on $\Pi$ induces an $H(E)$-action on each $\rT^{j}(\Pi)$.
 
 \begin{conj}[{\cite[Conjecture 6.3]{TV}}]\label{conj: TV} 
 Let $\Pi$ be an irreducible smooth representation of $G(E)$ over an algebraically closed field $k$ of characteristic $\ell$, whose isomorphism class is fixed by $\sigma$, so that the $G(E)$-action on $\Pi$ uniquely extends to a $G(E) \rtimes \tw{\sigma}$-action by \cite[Proposition 6.1]{TV}. Then for each $j \in \Z/2\Z$:
 \begin{enumerate}
 \item (Admissibility Conjecture) $\rT^j(\Pi)$ is admissible as a representation of $H(E)$.
 \item (Functoriality Conjecture) The $L$-parameter of every irreducible $H(E)$-subquotient of $\rT^{j}(\Pi)$ is sent by the $\sigma$-dual homomorphism $\chpsi$ to the $L$-parameter for $\Pi^{(\ell)} := \Pi \otimes_{k, \Frob_\ell} k$, the Frobenius twist of $\Pi$. 
 \end{enumerate}
 \end{conj}
 
 Part (1) of the Conjecture has been proved by Dhar--Nadimpalli \cite{dhar2025notefinitenesstatecohomology}. We will focus on Part (2). At the time of its formulation, the Local Langlands Correspondence was constructed only for certain families of groups, so the precise meaning of Conjecture \ref{conj: TV}(2) was left vague for general groups. We may now formulate a precise version for all $G$ using the work of Fargues-Scholze \cite{FS}; this is what we will describe next. 
 
\begin{example} For $G = \GL_n$, Vign\'{e}ras had constructed in \cite{Vig01} the Local Langlands Correspondence over $k$ (using the characteristic zero version due to Harris-Taylor \cite{HT01}), many years before the work of Fargues-Scholze. This gives a precise meaning to Conjecture \ref{conj: TV}(2) when $G$ and $H$ are both general linear groups, and in the context of cyclic base change, special cases have been proven by other authors, as will be discussed more in \S \ref{ssec: related}. Even for this case, where the statements of our results can be formulated in classical terms, the proofs will utilize the work of Fargues-Scholze. 
\end{example}

\subsection{The Fargues-Scholze correspondence} An output of the work \cite{FS} of Fargues-Scholze is a map
\begin{align}\label{eq: LLC}
\left\{ \begin{array}{@{}c@{}}  \text{irreducible admissible representations} \\  \text{$\Pi$ of $G(E)$ over $k$}\end{array} \right\}/\sim  & \longrightarrow \left\{ \begin{array}{@{}c@{}}  \text{semisimple $L$-parameters} \\ 
\rho_{\Pi} \co W_E \rightarrow \ld G(k)  \end{array} \right\}/\sim
\end{align}
that we call the \emph{Fargues-Scholze correspondence}. We refer to $\rho_\Pi$, which is most naturally regarded as an element of $\rH^1(W_E; \wh{G}(k))$, as the \emph{Fargues-Scholze parameter} of $\Pi$. The map $\Pi \mapsto \rho_\Pi$ is expected to be the semi-simplification of ``the'' hypothetical Local Langlands Correspondence. For specific groups including tori and $\GL_n$, the Local Langlands Correspondence has been constructed previously (by class field theory and by Harris-Taylor, respectively), and in these cases it is known that the Fargues-Scholze correspondence is compatible with the previous construction. 

For most groups, the Fargues-Scholze correspondence is quite mysterious. For example, it is expected that \eqref{eq: LLC} is surjective and has finite fibers, but proving this is wide open. For \emph{regular supercuspidal representations} of quite general groups, Kaletha \cite{Kal19} has prescribed explicit constructions for the $L$-parameters, which are strongly supported by the fact that the correspondence of Kaletha satisfies the expected endoscopic character relations \cite[Theorem 4.4.4]{FKS23}; while for general groups almost nothing is known about the Fargues-Scholze parameters of such representations. We remark that in the cases where the Fargues-Scholze correspondence has been explicated, there have been important geometric implications; an example is the work of Koshikawa \cite{Kosh21} towards torsion vanishing in the cohomology of Shimura varieties, where the relevant compatibility with the classical LLC is proven in \cite{HKW22}. 

\subsubsection{The functoriality conjecture}We prove the following result concerning modular functoriality in the Fargues-Scholze correspondence, which (at least away from small $\ell$) we take as fulfilling Conjecture \ref{conj: TV}(2). Write
$\ld\psi\co\ld H\to\ld G$ for the $L$-group form of the
$\sigma$-dual homomorphism, $\rho^{\mrm{ss}}$ for the
semisimplification of an $L$-parameter $\rho$
\cite[Proposition VIII.3.2]{FS}, and $\rho \sim \rho'$ if the $L$-parameters $\rho, \rho'$ are conjugate by an element of $\wh{G}(k)$.

\begin{thm}\label{thm: intro TV conj}
Let $\Pi$ be an irreducible smooth representation of $G(E)$ whose isomorphism class is fixed by $\sigma$, equipped with its unique extension to a $G(E) \rtimes \tw{\sigma}$-action by \cite[Proposition 6.1]{TV}. Then for each $j \in \Z/2\Z$ and every $H(E)$-irreducible subquotient $\pi$ of $\rT^{j}(\Pi)$, we have $\rho_{\Pi^{(\ell)}}\sim(\ld\psi\circ\rho_\pi)^{\mrm{ss}}$. 
\end{thm}

This result is deduced from Theorem \ref{thm: TV} in Example \ref{ex: TV conj}. Theorem \ref{thm: TV} treats ``derived smooth representations'' $\Pi \in D^b(\Rep^{\mrm{sm}}_k G(E))$. The derived version is more useful for calculations but less elementary to formulate, so we do not state it here. 

Theorem \ref{thm: intro TV conj} does not seem to be made any easier by assuming the full categorical conjectures of \cite{FS}, or the conjectural compatibility with Kaletha's explicit Local Langlands Correspondence, or any other standard conjectures about the Local Langlands Correspondence that we know. 
On the other hand, Theorem \ref{thm: intro TV conj} may well be useful for proving some of these other conjectures. For example, as an application of (the derived version of) Theorem \ref{thm: intro TV conj}, we calculate explicitly the Fargues-Scholze parameters of certain classes of representations.

\subsubsection{Explicit calculation of Fargues-Scholze parameters} The basic idea is that if the $L$-parameter $\rho_{\Pi} \co W_E \rightarrow \ld G(k)$ admits some factorization 
\begin{equation}\label{eq: L-param factorization}
\begin{tikzcd}
W_E \ar[rr, "\rho_{\Pi} "] \ar[dr, dashed] &  &  \ld G(k) \\
& \ld H(k) \ar[ur, "\ld \psi"'] 
\end{tikzcd}
\end{equation}
through an $L$-parameter into $\ld H(k)$ where $H$ is a \emph{torus} arising as the fixed points of some appropriate $\sigma$, then Theorem \ref{thm: intro TV conj} identifies $\rho_\Pi$ explicitly in terms of the $L$-parameter of $\rT^j(\Pi)$, which is computable since the Fargues-Scholze correspondence is completely understood for tori. Note that many supercuspidal parameters factor through the $L$-group of \emph{some} (not necessarily split) torus; for example, this is the case for regular supercuspidal parameters when $G$ is tamely ramified and $p$ does not divide the order of the Weyl group of $G$. 

We apply this idea to the toral representations considered in work of Chan-Oi \cite{CO21}. They are generalizations of depth-zero supercuspidals, with an analogous construction but instead using ``deeper-level Deligne-Lusztig representations'' studied in work of Chan-Ivanov \cite{CI21}; consequently, they include supercuspidal representations with arbitrarily high depth. The input for the (modular version of the) Chan-Oi construction is an elliptic unramified torus $T \subset G$ and a character $\theta \co T(E) \rightarrow k^\times$.

Under the assumption that $T$ contains a strongly regular element of order $\ell$, we prove (Corollary \ref{cor: cohomology constituent L-parameter}) that some irreducible constituent in the Chan-Oi construction has Fargues-Scholze parameter of the form
\begin{equation}\label{eq: intro L-param}
W_E \xrightarrow{\ld \theta} \ld T(k) \xrightarrow{ \ld j} \ld G(k),
\end{equation}
where $\ld j$ is the canonical embedding of the $L$-group of an unramified maximal torus. This is morally in accordance with the prediction of Kaletha's explicit Local Langlands Correspondence for regular supercuspidal representations in \cite{Kal19}; we say ``morally'' because Kaletha does not consider modular coefficients. Because the precise definition of a toral representation is complicated, we defer the precise formulation of our result to Theorem \ref{thm: derived parameter}. 

We emphasize that the simple appearance of \eqref{eq: intro L-param} belies the intricacy of Kaletha's recipe, which for example involves twisting the most natural (from a representation-theoretic perspective) guess for the $L$-parameter by a subtle ``twisting character''. In the relatively simple example of epipelagic representations of unitary groups, these characters were explicated in \cite{FRT} and found to be already very complicated there. The geometric construction of Chan-Oi somehow bakes this twisting character into the geometry (see \cite[\S 8]{CO21} for more discussion of this point) which is then reflected in our computation.  



\subsection{Further results} We mention some further results. 
 
 \subsubsection{Existence of functorial lifts} The following Theorem guarantees the existence of functorial lifts along any $\sigma$-dual homomorphism. 

\begin{thm}\label{thm: intro functorial lift}
Let $\pi$ be an irreducible smooth representation of $H(E)$ over $k$, with Fargues-Scholze parameter $\rho_{\pi} \in \rH^1(W_E; \wh{H}(k))$. Then there exists an irreducible smooth representation $\Pi$ of $G(E)$ over $k$ with Fargues-Scholze parameter $\rho_{\Pi}\sim(\ld\psi\circ\rho_\pi)^{\mrm{ss}}\in\rH^1(W_E;\wh G(k))$. 
\end{thm}

The more general version (allowing non-split groups) is Theorem \ref{thm: mod ell lifting}. This version allows us to treat further examples such as the following. 

\begin{example}[Base change]\label{ex: base change}
Taking $H$ to be a reductive group over $E$ and $G = \Res_{E'/E}(H_{E'})$ for a cyclic $\ell$-extension $E'/E$ and $\sigma$ a generator of $\Gal(E'/E)$ acting in the natural way, Theorem \ref{thm: intro functorial lift} asserts the existence of base change along $E'/E$. In equal characteristic and for the Genestier-Lafforgue correspondence, this was established in \cite[Theorem 1.1]{Fe23}. 
\end{example}

\begin{example}
Take $\sigma$ to be conjugation by a strongly regular element of $G(E)$ of order $\ell$. Then $H$ is a torus, so every $\rho \in \rH^1(W_E; \wh{H}(k))$ is realized as a Fargues-Scholze parameter. Theorem \ref{thm: mod ell lifting} therefore implies that every semisimple $\rho\in\rH^1(W_E;\wh G(k))$ which factors through $\ld\psi$ is realized as a Fargues-Scholze parameter. It would be interesting to investigate which $\ld \psi$ can arise as a $\sigma$-dual homomorphism, as this might help to show surjectivity of the Fargues-Scholze correspondence over $k$: recall that if $p$ is not too small relative to $G$, then every supercuspidal $L$-parameter factors through the $L$-group of some torus, and we know that the Fargues-Scholze correspondence is surjective for tori. 	
\end{example}

\subsubsection{Functoriality for the Bernstein center} Let $\mf{Z}(G;k)$ be the Bernstein center of $G$ with coefficients in $k$, and similarly for $H$. The Fargues-Scholze correspondence \eqref{eq: LLC} is deduced from the construction of a $k$-algebra homomorphism  
\begin{equation}\label{eq: intro FS homomorphism}
\FS_G \co \Exc_k(W_E, \wh{G}) \rightarrow \mf{Z}(G;k)
\end{equation}
where $\Exc_k(W_E, \wh{G})$ is the \emph{excursion algebra} (over $k$). 

Building on ideas of Treumann-Venkatesh, we construct a map of Bernstein centers
\begin{equation}\label{eq: intro TV homomorphism} 
\mf{Z}_{\TV} \co \mf{Z}(G;k) \rightarrow  \mf{Z}(H;k)
\end{equation}
which we call the \emph{Treumann-Venkatesh homomorphism}. We show in Theorem \ref{thm: FS TV} that the diagram \eqref{eq: intro bernstein center functoriality} commutes.
\begin{equation}\label{eq: intro bernstein center functoriality}
\begin{tikzcd}
\Exc_k(W_E, \wh{G})  \ar[r, "\ld\psi^*"] \ar[d, "\FS_G"']  & \Exc_k(W_E, \wh{H})  \ar[d, "\FS_H"]\\
\mf{Z}(G;k)\ar[r, "\mf{Z}_{\TV}"] & \mf{Z}(H;k)
\end{tikzcd}
\end{equation}
This shows that $\mf{Z}_{\TV}$ realizes functoriality for the Bernstein center. In fact, it is the main ingredient in the proof of Theorem \ref{thm: intro functorial lift}.

  \subsection{Methods} We will outline the ideas behind some of the proofs. In the sketch below, we use some abbreviated and simplified notation which does not match that of the main body of the text. 

We begin by commenting on the proof of Theorem \ref{thm: intro sigma-dual}. As explained earlier, the proof is based on categorifying the Brauer homomorphism of Treumann-Venkatesh to a \emph{Brauer functor}, from perverse sheaves on the affine Grassmannian $\Gr_G$ for $G$ to perverse sheaves on the affine Grassmannian $\Gr_H$ for $H$, which should be Tannakian (i.e., additive, symmetric monoidal, and compatible with the fiber functor). 

The construction of the Brauer homomorphism is based on the observation that the \emph{restriction of $\sigma$-equivariant functions} from $G(E)/K$ to $H(E)/K^\sigma$ has the miraculous property of being compatible with convolution specifically in characteristic $\ell$ (the order of $\sigma$). The naive categorification of this operation would be restriction of sheaves from $\Gr_G$ to $\Gr_H$, but the miracle does not (naively) persist to the level of sheaves. To overcome this, we use an operation called \emph{Smith-Treumann localization}, which has been studied in other contexts \cite{RW, LL}, and combine it with the theory of the \emph{Tate construction}.

\begin{figure}[!h]
    \centering
    \begin{tikzpicture}
        \matrix (m) [matrix of nodes, row sep=1cm, column sep=4cm] {
            |[draw, circle, shift={(0.75cm,0)}]| \(\Rep_k(\chG)\) & |[draw, circle, shift={(0.75cm,0)}]| \(\Rep_k(\chH)\) \\
            |[draw, circle, shift={(0.75cm,0)}]| \(\Sat(\Gr_G;k)\) & |[draw, circle, shift={(0.75cm,0)}]| \(\Sat(\Gr_H;k)\) \\
            |[draw]| Tate$(\rR\Gamma_c(\Sht_G, -))$ & |[draw]|  Tate$(\rR\Gamma_c(\Sht_H, -))$\\
        };
        
        \draw[->] ($(m-2-1.south) + (0.4cm, 0.1)$) -- node[midway, right] {$\Sat_G$} ($(m-3-1.south east)!0.5!(m-3-1.north east) + (-0.5,0.1)$);
        \draw[->] ($(m-2-2.south)+ (.4cm, 0.1)$) -- node[midway, left] {$\Sat_H$}  ($(m-3-2.south east)!0.5!(m-3-2.north east) +  (-0.5,0.1)$);
        \draw[->] (m-1-1.south) -- node[midway, right]{Geom. Satake} (m-2-1.north);
        \draw[->] (m-1-2.south) -- node[midway, left]{Geom. Satake} (m-2-2.north);
        \draw[->] (m-2-1.east) -- node[midway, above] {$\wt{\cbr}$} (m-2-2.west);
        \draw[->] (m-1-1.east) -- node[midway, above] {$\chpsi^*$} (m-1-2.west);
        
        \draw[dashed, <->] (m-3-2.west) -- node[midway, above] {equiv. localization} (m-3-1.east);
    \end{tikzpicture}
    \caption{This cartoon (produced with the aid of ChatGPT-3.5 after much coaxing) depicts the Brauer functor $\wt{\cbr}$ interfacing with the Tate cohomology of moduli of local shtukas and with the $\sigma$-dual homomorphism $\chpsi$.}\label{fig: sht}
\end{figure}

Next we turn towards the proof of Theorem \ref{thm: intro TV conj}. For this we need to integrate the Brauer functor with the construction of the Fargues-Scholze correspondence for $G$ and for $H$. The Fargues-Scholze correspondence for $G$ can be obtained by constructing excursion operators on the cohomology of \emph{moduli spaces $\Sht_G$ of local $G$-shtukas} with coefficients in Satake sheaves coming from $\Gr_G$. We need to compare such cohomology groups to the ones obtained from applying the Brauer functor to get Satake sheaves on $\Gr_H$, transporting them to the moduli spaces $\Sht_H$ of local $H$-shtukas, and then taking cohomology. This comparison is mediated by \emph{equivariant localization} for Tate cohomology. Figure~\ref{fig: sht} depicts a cartoon of the strategy. From this comparison, we extract relations among certain excursion operators for $\chG$ and for $\chH$, which are ultimately used to establish Theorem \ref{thm: intro TV conj}.

\subsection{Related work}\label{ssec: related} We note some related work on the problems studied here. 

\subsubsection{The functoriality conjecture} We focus first on the ``functoriality conjecture'', Conjecture \ref{conj: TV}(2). Ronchetti \cite{Ron16} studied Conjecture \ref{conj: TV}(2) in the special case of cyclic base change for $\GL_n$, proving it for depth-zero cuspidal representations ``of level zero and minimal-maximal type''. 

Dhar--Nadimpalli \cite{DS23} studied Conjecture \ref{conj: TV}(2) in the special case of cyclic base change for $\GL_n$, proving it for generic representations (i.e., those possessing mod-$\ell$ Whittaker models). 

The earlier work of the author \cite{Fe23} proved the statement analogous to Theorem \ref{thm: intro TV conj} for the Genestier-Lafforgue correspondence (in equal characteristic) in the special case of cyclic base change, and in doing so introduced an early form of some ideas developed here. We point out several differences. 
\begin{itemize}
\item The methods of \cite{Fe23} were based on local-global compatibility, and were thus fundamentally restricted to the function field setting. Here our methods are purely local, and we encompass both mixed characteristic and equal characteristic local fields. To do this, we work in the context of $p$-adic geometry, which presents substantial new difficulties. 
\item The argument of \cite{Fe23} was restricted to the case of cyclic base change, for both local and global reasons. On the \emph{local} side, the $\sigma$-dual homomorphism is obvious for cyclic base change, and the construction of the Brauer functor in that case exploited some simplifying special features of cyclic base change; the present paper introduces more general and conceptual arguments to treat the local aspects in arbitrary generality. We still do not know how to generalize the \emph{global} arguments even in the function field case, but fortunately they are irrelevant for the new approach. 
\end{itemize} 

The author's work \cite{BFHKT} with B\"{o}ckle-Harris-Khare-Thorne explicates Conjecture \ref{conj: TV}(2) for the Genestier-Lafforgue correspondence in the special case of cyclic base change for toral supercuspidals, assuming several conjectures about torsion in the cohomology of deep-level Deligne-Lusztig varieties. The analogous results for the Fargues-Scholze correspondence are subsumed by the $L$-parameter calculations in \S \ref{sec: toral supercuspidals}. 
 
\subsubsection{The $\sigma$-dual homomorphism} The work \cite{RW} of Riche-Williamson, which gives a geometric proof of the linkage principle, is not logically related to problems we study here, but has some philosophical similarities. We apply Smith--Treumann localization from $\Gr_G$ to $\Gr_H$ coming from an automorphism of $G$, while Riche-Williamson apply (a slightly different form of) Smith--Treumann localization to a self-embedding of $\Gr_G$ coming from the action of $\mu_\ell \subset \G_m$ via loop rotation. 

The paper \cite{LL} of Leslie-Lonergan also applies Smith--Treumann localization along this self-embedding, in an attempt to give a geometric construction of the Frobenius contraction functor on $\Rep_k(\chG)$. Our work also has some common features to theirs.

 \subsection{Organization of the paper}
 The paper begins in earnest with \S \ref{sec: Smith--Treumann}, which develops Smith--Treumann localization
for diamonds, and \S \ref{sec: Tate cohomology}, which studies the relation of this theory to the Tate construction and Tate cohomology. In 
\S \ref{sec: fixed point} we calculate the fixed point subspaces for various affine Grassmannian and shtuka spaces that will arise. Then we categorify the Treumann--Venkatesh theory of the Brauer homomorphism in \S \ref{sec: categorical TV}. Using Tannakian reconstruction, we turn around to construct the $\sigma$-dual homomorphism and establish its compatibilities in \S \ref{sec: sigma-dual homomorphism}. In \S \ref{sec: local shtukas} we interface this theory with cohomology of local shtukas and excursion operators. In \S \ref{sec: TV conjecture} we deduce consequences for the Treumann--Venkatesh Conjecture. Finally, in \S \ref{sec: toral supercuspidals} we illustrate some first applications to computing the Fargues--Scholze parameters of certain supercuspidal representations. 

\begin{figure}[tbp]
\centering
\begin{tikzpicture}[
  box/.style={draw, line width=.4pt, align=center, font=\small,
    inner xsep=5pt, inner ysep=5pt, minimum height=1.45cm},
  input/.style={box, text width=3.15cm, minimum height=1.8cm},
  arrow/.style={->, >=stealth, line width=.45pt, shorten >=1pt, shorten <=1pt}
]
\node[input] (smith) at (-4.15,0)
  {Smith--Treumann\\localization for diamonds\\\S\ref{sec: Smith--Treumann}};
\node[input] (tate) at (0,0)
  {Tate construction\\and cohomology\\\S\ref{sec: Tate cohomology}};
\node[input] (fixed) at (4.15,0)
  {Fixed-point geometry\\of Grassmannians\\and local shtukas\\\S\ref{sec: fixed point}};

\node[box, text width=3.9cm] (brauer) at (-2.5,-2.15)
  {Brauer functor\\constant terms and fusion\\\S\ref{sec: categorical TV}};
\node[box, text width=3.9cm] (dual) at (2.5,-2.15)
  {$\sigma$-dual homomorphism\\and its compatibilities\\\S\ref{sec: sigma-dual homomorphism}};

\node[box, text width=5.8cm] (shtukas) at (0,-4.3)
  {Local-shtuka cohomology\\and excursion compatibility\\\S\ref{sec: local shtukas}};
\node[box, text width=7.2cm] (tv) at (0,-6.35)
  {Derived Treumann--Venkatesh conjecture\\Bernstein centers and functorial lifts\\\S\ref{sec: TV conjecture}};
\node[box, text width=5.8cm] (toral) at (0,-8.4)
  {Fargues--Scholze parameters\\of toral supercuspidals\\\S\ref{sec: toral supercuspidals}};

\draw[arrow] (smith.east) -- (tate.west);
\draw[arrow] (smith.south) -- ($(brauer.north west)!0.5!(brauer.north)$);
\draw[arrow] (tate.south) -- (brauer.north);
\draw[arrow] (fixed.south west) -- (brauer.north east);
\draw[arrow] (brauer.east) -- (dual.west);
\draw[arrow] (brauer.south) -- ($(shtukas.north west)!0.5!(shtukas.north)$);
\draw[arrow] (dual.south) -- ($(shtukas.north)!0.5!(shtukas.north east)$);
\draw[arrow] (fixed.east) -- (6.3,0) -- (6.3,-4.3) -- (shtukas.east);
\draw[arrow] (shtukas.south) -- (tv.north);
\draw[arrow] (tv.south) -- (toral.north);
\end{tikzpicture}

\caption{Leitfaden of the paper.}\label{fig: leitfaden}
\end{figure}

\subsection{AI methodology} After the initial referee report was received in early 2026, various AI models were used to proofread the text for errors. The initially submitted version of this paper required an assumption that $\ell$ is sufficiently large (with effective bounds). Most recently, a review by GPT-6 uncovered a gap in a proof of a Lemma (no longer present in the manuscript), and suggested a fix for the gap that used the Tate construction. Upon contemplating this issue, the author realized that parts of the paper could be reworked to eliminate the assumptions on $\ell$ entirely. The author had been thinking about how to eliminate the technical assumptions on $\ell$, for application to his sequel work with Sean Cotner based on these results. 

\subsection{Acknowledgments} It is a pleasure to acknowledge Brian Conrad, Olivier Dudas, Ian Gleason, Jesper Grodal, Tasho Kaletha, Teruhisa Koshikawa, Gopal Prasad, Simon Riche, Peter Scholze, Sug Woo Shin, Jay Taylor, David Treumann, Akshay Venkatesh, and Geordie Williamson for relevant discussions. Special thanks to Venkatesh for posing the question of finding a uniform construction of the $\sigma$-dual homomorphism to me in 2017, to Scholze for the proof of Lemma \ref{lem: extra small open embedding}, and to Conrad for help with Lemma \ref{lem: fixed borus}. Thanks also to Johannes Ansch\"{u}tz, Shachar Carmeli, Jessica Fintzen, David Hansen, Naoki Imai, Mingjia Zhang, and the anonymous referee for comments on a draft. Parts of this paper were completed at the Hausdorff Institute for Mathematics, funded by the Deutsche Forschungsgemeinschaft (DFG, German Research Foundation) under Germany's Excellence Strategy – EXC-2047/1 -- 390685813. This work was supported by NSF Postdoctoral Fellowship DMS-1902927 and NSF Grant DMS-2302520.

\section{Notation}\label{sec: notation}

We fix a prime $p$. 

\subsection{Local fields}
We let $E$ be a local field of residue characteristic $p$, $\cO_E$ its ring of integers, $\varpi_E$ a uniformizer, and $\F_q  = \cO_E/\varpi_E$ its residue field. 

Let $\brE$ be the completion of the maximal unramified extension of $E$. For a reductive group $G/E$, $B(G)$ denotes the Kottwitz set of $G$, i.e., elements of $G(\brE)$ modulo Frobenius-conjugacy. 

Let $E^s$ be a separable closure of $E$ and $W_E \subset \Gal(E^s/E)$ be the Weil group of $E$.

\subsection{The group $\Sigma$} Throughout, $\ell$ denotes a prime number different from $p$ and $\Sigma$ denotes a finite cyclic group of order $\ell$. The notation $\sigma$ always denotes a generator of $\Sigma$; conversely, if an automorphism $\sigma$ is constructed first then $\Sigma$ always denotes the group generated by $\sigma$. Some constructions (e.g., Tate cohomology) are phrased in terms of $\sigma$, but ultimately all constructions are independent of the choice of $\sigma$. 

We denote by $N$ the element $1 + \sigma + \ldots + \sigma^{\ell-1} \in \Z[\Sigma]$.  

For an object $Y$ with an action of $\Sigma$, we denote by $Y^\sigma$ or $\Fix(\sigma, Y)$ the $\sigma$-fixed points of $Y$. 

\subsection{Reductive groups} Throughout, $G$ denotes a reductive group over $E$. Throughout, reductive groups are understood to be connected. When the notation $H$ appears, it refers to a reductive group arising as the subgroup of $G$ fixed by an action of $\Sigma$. We denote by $\iota \co H \rightarrow G$ the tautological embedding, and we use the same notation for induced maps such as $H(E) \rightarrow G(E)$, $B(H) \rightarrow B(G)$, $\Gr_H \rightarrow \Gr_G$, etc.

For a split torus $T$, we denote by $X_*(T)$ and $X^*(T)$ the cocharacter and character groups of $T$, respectively. 

We denote by $\Rep^{\sm}_k G(E)$ the category of smooth representations of $G(E)$ over $k$. 

\subsection{Perfectoid spaces}

The notation $(C,C^+)$ will always mean that $C$ is an algebraically closed perfectoid field and $C^+ \subset C^\circ$ is an open integrally closed subring, where $C^\circ$ denotes the ring of power-bounded elements in $C$. 

Following \cite{FS}, we denote by $\Perf = \Perf_{\ol{\F}_p}$ the category of perfectoid spaces over $\ol{\F}_p$.

\subsection{Coefficients}\label{ssec: notation sheaves}
Throughout, $k$ denotes an algebraically closed field of characteristic $\ell$. We denote by $\OO := W(k)$ the Witt vectors of $k$. 

We write $\Frob_\ell \co k \rightarrow k$ for the absolute Frobenius automorphism $x \mapsto x^\ell$. 

We will use $\Lambda$ to denote a coefficient ring which is finite over $\OO$. In particular, $\Lambda$ is $\ell$-adically complete.

\subsection{Sheaves}\label{sssec:etale-sheaves} The six-functor formalism for sheaves of $\Lambda$-modules is developed in \cite{Sch17} (the adic case is in \S 26 of \emph{loc. cit.}). 

At the request of the referee, we point out some subtleties in the construction of this sheaf theory. If $Y$ is any small v-stack, then \cite{Sch17} defines a derived category of \'etale sheaves $D_{\et}(Y; \Lambda)$. In general, this is defined by v-descent from strictly totally disconnected spaces, hence can be tricky to understand concretely. For discrete coefficients, if $Y$ is a locally spatial diamond, then it has an \'etale site $Y_{\et}$ and \cite[Proposition 14.15]{Sch17} says that $D_{\et}(Y; \Lambda)$ is the left completion of $D(Y_{\et}; \Lambda)$. In particular, the bounded subcategories of the two are identified. For adic coefficients, we define $D_{\et}(Y; \Lambda)$ as a derived inverse limit from the torsion case. 

Also, if $Y$ is a locally spatial diamond then $Y_{\et}$ has enough points by \cite[Proposition 14.3]{Sch17}, and geometric stalks form a conservative family.

\begin{defn}\label{def: shriekable}
We say that a map $f \co Y \rightarrow Y'$ of small v-stacks on $\Perf$ is \emph{shriekable} if it is compactifiable, representable in locally spatial diamonds, and has locally finite $\dimg$. This is exactly the hypothesis in \cite[\S 1]{Sch17} for $\rR f_! \co D_{\et}(Y; \Lambda) \rightarrow D_{\et}(Y'; \Lambda)$ to exist, and under the same assumptions its right adjoint $\rR f^! \co D_{\et}(Y'; \Lambda) \rightarrow D_{\et}(Y; \Lambda) $ exists.
\end{defn}


\subsection{Categories}\label{ssec: categories}

We fix some notation for abstract categorical constructions. 

\subsubsection{Equivariant derived categories} We begin with some discussion of equivariant derived categories. There are two settings of interest, abelian categories and triangulated categories, where the definitions need to be treated differently. 

For a 1-category $\msf{C}$ with an action of a group $\Gamma$, we let $\msf{C}^{B \Gamma}$ denote the category of $\Gamma$-equivariant objects in $\msf{C}$. For an $\infty$-category $\cC$ with an action of a group $\Gamma$, we let $\cC^{B \Gamma}$ denote the category of $\Gamma$-equivariant objects in $\cC$. Our notation is not completely standard (one might also see $\mathsf{C}^{h \Gamma}$ or $\mathsf{C}^\Gamma$ used to denote the same notion in the literature), but follows \cite[I.13]{FS}. We can give a precise formulation of the definitions in a somewhat uniform manner:
\begin{itemize}
\item A $\Gamma$-action on a 1-category $\msf{C}$ can be viewed as a pseudofunctor $\msf{F} \co B \Gamma \rightarrow \msf{Cat}$ from the groupoid $B \Gamma$ to the 2-category of 1-categories $\msf{Cat}$, carrying the neutral point of $B\Gamma$ to $\msf{C}$, and we have 
\[
\msf{C}^{B\Gamma} := 2\text{-}\lim_{B \Gamma} \msf{F}.
\]
\item A $\Gamma$-action on an $\infty$-category $\cC$ can be viewed as a functor $\cF \co B \Gamma \rightarrow \msf{Cat}_\infty$ from the $\infty$-groupoid $B\Gamma$ to the $\infty$-category of $\infty$-categories $\msf{Cat}_\infty$, carrying the neutral point of $B\Gamma$ to $\cC$, and we have 
\[
\cC^{B\Gamma} := \lim_{B \Gamma} \cF.
\]
\end{itemize}
Finally, if a triangulated category $\msf{C}$ is the homotopy category of a stable $\infty$-category $\cC$, then we denote by $\msf{C}^{B \Gamma}$ the homotopy category of $\cC^{B\Gamma}$. Often we abuse notation by suppressing the discussion of $\cC$ when there is a natural choice; for example, if $\msf{C}$ is the triangulated derived category of sheaves on a site, then $\cC$ is taken to be the stable $\infty$-category of sheaves on that site. 

\subsubsection{Tensor products}
Let $A$ be a ring and let $\msf{C}, \msf{D}$ be $A$-linear finitely cocomplete\footnote{i.e., admitting finite colimits.} abelian categories. We abbreviate $\msf{C} \otimes_A \msf{D}$ for the (Kelly) tensor product of $\msf{C}$ and $\msf{D}$ relative to $A$ \cite[\S 2.4]{LF13}. We note that this coincides with Deligne's tensor product if the latter exists, by \cite[Theorem 1]{LF13}.

\section{Smith--Treumann localization for diamonds}\label{sec: Smith--Treumann}

In this section we develop for diamonds a form of \emph{Smith--Treumann localization}, which refers to a type of sheaf-theoretic equivariant localization from a $\Sigma$-equivariant space to its $\Sigma$-fixed points. There is a loose analogy between Smith--Treumann localization and the perhaps more familiar hyperbolic localization for spaces with a $\G_m$-action.

The output of the theory looks similar to that for complex algebraic varieties or schemes (developed in \cite{Tr19} and \cite{RW}). However, the details of the proofs are quite different, owing to the different behavior of \'{e}tale sheaves on adic spaces. For example, the theory relies crucially on finiteness results, which in the case of schemes comes from quasi-compactness (among other things). On the other hand, non-quasicompactness is ubiquitous in $p$-adic geometry. For example, the complement of a closed subspace is usually not quasi-compact, hence the extension-by-zero from such a complement is not constructible; this example already illustrates that we will frequently need to contend with non-constructible sheaves. 

The contents of this section are as follows. In \S \ref{ssec: extra small} we codify the notion of ``extra small'' v-stacks. The adjective ``extra small''  is a slight strengthening of ``small'', and guarantees that we can approximate the v-stack by quasicompact subspaces. In \S \ref{ssec: permanence big tate category} we establish bounds on tor-amplitude.

\subsection{Extra small v-stacks}\label{ssec: extra small}

Let $\omega_1$ be the first uncountable cardinal. Recall that an \emph{$\omega_1$-cofiltered} inverse system $I$ is one for which any functor $J \rightarrow I$, with $J$ a countable category, extends to $J^{\triangleleft} \rightarrow I$, where the cone $J^{\triangleleft}$ consists of $J$ plus an extra object with a unique map to each object of $J$. 

\begin{defn}[Extra small v-stacks] We say that a v-stack $Y$ is \emph{extra small} if for any $\omega_1$-cofiltered inverse system $\{S_i = \Spa(R_i, R_i^+)\}_{i \in I}$ of affinoid perfectoid spaces with inverse limit $S = \Spa(R, R^+)$, the natural map 
\[
\colim_{i \in I} Y(S_i) \rightarrow Y(S)
\]
is an isomorphism. 
\end{defn}

\begin{defn}[Relative extra smallness]\label{def: relative extra small}
A morphism $f\co X\to Y$ of v-stacks is \emph{relatively extra small} if, for every $\omega_1$-cofiltered inverse system of affinoid perfectoid spaces $S_i\to Y$ with inverse limit $S\to Y$, the natural map
\[
\colim_i \Hom_Y(S_i,X)\longrightarrow \Hom_Y(S,X)
\]
is an isomorphism.
\end{defn}

\begin{remark}\label{rem:relative-extra-small} It is clear from the definition that relative extra smallness is preserved by arbitrary base change. 
    
Moreover, if $X$ and $Y$ are extra small, then every morphism $X\to Y$ is relatively extra small. Indeed, a map $S\to X$ factors through some $S_i$ by extra smallness of $X$; after increasing $i$, it is easy to see that this factorization is as maps over $Y$. 
\end{remark}
 
The following Lemma justifies the terminology. 

\begin{lemma}
If a v-stack $Y$ is extra small, then $Y$ is small. 
\end{lemma}

\begin{proof}
The proof is contained in the first paragraph of \cite[Proof of Proposition III.1.3]{FS}. 
\end{proof}

\begin{remark}Informally speaking, all the small v-stacks that have come up in our experience are also extra small. The remark below \cite[Proposition III.1.3]{FS} suggests that all ``reasonable'' v-stacks are extra small. 
\end{remark}

\begin{example}\label{ex: extra small Bun}
The last paragraph of the proof of \cite[Proposition III.1.3]{FS} shows that $\Bun_G$ is extra small. By minor variations on this argument, the following v-stacks are also extra small.
\begin{enumerate}
\item The local Hecke stack $\cHck_{G,(\Div^1_X)^I}$ \cite[Definition VI.1.6]{FS}, and the global Hecke stack $\Hck_G^I$ \cite[\S IX.2]{FS}. 
\item The Beilinson-Drinfeld Grassmannian $\Gr_{G,(\Div^1_X)^I}$ \cite[Definition VI.1.8]{FS}. 
\item The moduli spaces of local shtukas $\Sht_{(G, b, \mu_{\bu}), K}$ \cite[Lecture 23]{SW20}. 
\end{enumerate}
It follows from the same proofs, or alternatively from Remark \ref{rem:relative-extra-small}, that in all of these cases the projection morphisms to $(\Div_X^1)^I$ are relatively extra small. 
\end{example}




The importance of extra smallness (for us) is to control the cohomological dimension of direct image along open embeddings. For quasi-compact open embeddings, the direct image is actually exact \cite[Lemma 21.13]{Sch17}, but we will typically be contending with non-quasicompact open embeddings. Under hypotheses that spaces are extra small, we will be able to approximate open embeddings by quasicompact ones, as articulated in  the following Lemma (pointed out by Peter Scholze). 

\begin{lemma}\label{lem: extra small open embedding}
Let $Y$ be a locally spatial diamond admitting a countable open cover by spatial diamonds $Y_n$. Let $Z \inj Y$ be a relatively extra small closed embedding. Then $U := Y \setminus Z$ is a countable union of quasicompact open subspaces. 
\end{lemma}

\begin{proof}
Let $Z_n := Z \cap Y_n$. If $Y_n \setminus Z_n$ is a countable union of quasicompact open subsets for each $n$, then 
\[
Y \setminus Z = \bigcup_{n} (Y_n \setminus Z_n)
\]
is also a countable union of quasicompact open subsets. Therefore, renaming $Y_n$ to $Y$ and $Z_n$ to $Z$, we may assume that $Y$ is spatial and $Z\to Y$ is relatively extra small, with the goal of showing that $U := Y \setminus Z$ is a countable union of quasicompact open subsets. 

By \cite[Proposition 11.24]{Sch17}, $Y$ admits a surjective universally open quasi-pro-\'etale cover by a strictly totally disconnected space. Since relative extra smallness is preserved by pullback, and such a cover carries quasicompact opens to quasicompact opens, we may pass to the cover and assume that $Y$ is strictly totally disconnected. Since $|Z|$ is closed and generalizing, it is the inverse image of a closed subset of $\pi_0(Y)$, and $Y\setminus Z$ is a union of clopen subsets.

Consider the collection $\cZ$ of closed sub-diamonds $Z_i \subset Y$ such that 
\begin{itemize}
\item $Z_i \supset Z$, and 
\item $Y \setminus Z_i$ is a countable union of quasi-compact open subsets.
\end{itemize}
Viewing $\cZ$ as a partially ordered set under inclusion, we claim that $\cZ$ is $\omega_1$-cofiltered. Indeed, if $I$ is countable then we have that
\[
Y \setminus \bigcap_{i \in I} Z_i = \bigcup_{i \in I} (Y \setminus Z_i) 
\]
is a countable union of countable unions of quasi-compact open subsets, so $\bigcap_{i \in I} Z_i \in \cZ$. The poset $\cZ$ is non-empty, since it contains $Y$.

There is an obvious map 
\begin{equation}\label{eq: intersection of Z_i}
Z \rightarrow \limit_{Z_i \in \cZ} Z_i,
\end{equation}
which we claim is an isomorphism. Recall that closed sub-diamonds of $Y$ are determined by their underlying topological spaces, as follows from \cite[Proposition 11.20]{Sch17} (the generalization of this statement to v-stacks is \cite[Lemma 2.7]{AGLR}), so it suffices to check \eqref{eq: intersection of Z_i} at the level of topological spaces. We will show that 
\begin{equation}\label{eq:top-space-equalities}
|\limit_{Z_i \in \cZ} Z_i| = \limit_{Z_i \in \cZ} |Z_i| = |Z|,
\end{equation}
The first equality above follows from \cite[Lemma 11.22]{Sch17}, since this is an inverse limit of spatial diamonds with qcqs transition maps (here the transition maps are closed immersions). To see the second equality, we note that any point of the open subspace $|Y \setminus Z|$ lies in a clopen subset thereof since $Y$ is strictly totally disconnected. If $y \in |Y|\setminus |Z|$ and $V \subset |Y|\setminus |Z|$ is such a clopen neighborhood, then the closed subdiamond $Y\setminus V$ belongs to $\cZ$ and does not contain $y$. This shows that $\limit_{Z_i \in \cZ} |Z_i| \subset |Z|$, and this containment is an equality because every $Z_i$ contains $Z$.

Since $Y$ is spatial, there exists \cite[Proposition 11.24]{Sch17} a quasi-pro-\'etale cover of $Y$ by a strictly totally disconnected space, which is of the form 
\begin{equation}\label{eq: extra small cover}
 \Spa (B,B^+) \surj Y
 \end{equation}
 by \cite[Proposition 1.15]{Sch17}. Recall that a closed subspace of a strictly totally disconnected space is again strictly totally disconnected by \cite[Proposition 7.16]{Sch17}, and then affinoid by \cite[Proposition 1.15]{Sch17}. Hence for each $Z_i \in \cZ$, the cover \eqref{eq: extra small cover} pulls back to a cover $\Spa (A_i, A_i^+) \surj Z_i$ where $\Spa (A_i, A_i^+)$ is strictly totally disconnected. The system $\{ \Spa(A_i, A_i^+)\}_{Z_i \in \cZ}$ is $\omega_1$-cofiltered, since it is the pullback of the $\omega_1$-cofiltered system of closed subdiamonds $Z_i \subset Y$ along the cover \eqref{eq: extra small cover}. Then, letting $\Spa (A,A^+)  := \varprojlim_{Z_i \in \cZ} \Spa (A_i, A_i^+)$,\footnote{So $A = \colim A_i$, which we note is already complete since any Cauchy sequence already lies in some $A_i$ by the $\omega_1$-cofiltered condition.} the composite map 
\[
\Spa (A,A^+)  = \limit_{Z_i \in \cZ} \Spa (A_i, A_i^+) \rightarrow \limit_{Z_i \in \cZ} Z_i = Z
\]
factors, as a map over $Y$, through some $g_i\co\Spa(A_i,A_i^+)\to Z$ by relative extra smallness. The two pullbacks of $g_i$ to $\Spa(A_i,A_i^+)\times_{Z_i}\Spa(A_i,A_i^+)$ have the same composite with $Z\inj Y$ and hence agree, since $Z\inj Y$ is a monomorphism. Consequently $g_i$ descends along the v-cover $\Spa(A_i,A_i^+)\surj Z_i$ to a map $Z_i\to Z$ over $Y$. As $Z\subseteq Z_i$ are closed subdiamonds of $Y$, this forces $Z_i=Z$.
Hence $Z \in \cZ$, which means by definition that $Y \setminus Z$ is a countable union of quasi-compact open subsets.
\end{proof}

Motivated by Lemma \ref{lem: extra small open embedding}, we make the following definition. 

\begin{defn}
We say that a locally spatial diamond $Y$ is \emph{($\omega_0$-)locally spatial} if $Y$ can be written as a (countable) union $\bigcup_n Y_n$ of open spatial sub-diamonds. (Recall that $\omega_0 = \aleph_0$ is the cardinality of the natural numbers.) We say that a closed subdiamond $Z \inj Y$ is \emph{$\omega_0$-locally extra small} if such a union can be arranged so that $Y_n \cap Z$ is extra small for each $n$. 

We say that a morphism of v-stacks $f \co Y' \rightarrow Y$ is ($\omega_0$-)\emph{locally spatial} if for any spatial diamond $X$ and any map $X \rightarrow Y$, the fibered product $X \times_Y Y'$ is ($\omega_0$-)locally spatial.
 
\end{defn}

\begin{example}Let $i\co Z\inj Y$ be a relatively extra small closed embedding of $\omega_0$-locally spatial diamonds. Then the complementary open embedding $j\co U\inj Y$ is $\omega_0$-locally spatial. Indeed, after any base change from a spatial diamond $X\to Y$, relative extra smallness is preserved and Lemma \ref{lem: extra small open embedding}, applied with the single spatial open $X$, writes $X\setminus(X\times_Y Z)$ as a countable union of quasicompact open subdiamonds.
\end{example}

\begin{defn}
For $d \geq 0$, we say that a locally spatial diamond $Y$ has \emph{$\ell$-cohomological dimension $\leq d$} if every quasicompact open subdiamond $W\subseteq Y$ has cohomology on \'{e}tale $\F_\ell$-sheaves of amplitude in $[0,d]$. We say that $Y$ is \emph{$\ell$-bounded} if it has $\ell$-cohomological dimension $\leq d$ for some $d<\infty$.
\end{defn}

\begin{example}The following statements are used to bound the $\ell$-cohomological dimension in practice. Recall from \cite[Definition 21.1]{Sch17} that the \emph{dimension} of a locally spectral space is the supremum of the lengths of chains of specializations. By \cite[Proposition 21.11]{Sch17}, if $Y$ is a spatial diamond of finite dimension, then 
\[
\text{$Y$ has $\ell$-cohomological dimension $\leq \dim Y + \sup_y \cd_\ell y$}
\]
where $y$ runs through maximal points of $Y$. According to \cite[Proposition 21.16]{Sch17}, if $f \co Y \rightarrow \Spa (C, C^+)$ is a map of locally spatial diamonds, then $\cd_{\ell} y \leq \dimg f$ for all maximal points $y \in Y$.
These estimates are uniform on quasicompact open subdiamonds: any such open $W\subseteq Y$ is spatial, satisfies $\dim W\leq\dim Y$, and the same bound on the cohomological dimensions of its maximal points applies.
\end{example}

\begin{lemma}\label{lem: tor dim j_*}
Let $A$ be a finite commutative $\OO$-algebra. Let $Y$ be an $\omega_0$-locally spatial diamond with $\ell$-cohomological dimension $\leq d$. Let $Z \inj Y$ be a relatively extra small closed subdiamond with (open) complement $j \co U \inj Y$.
Then $\rR j_* \co D_{\et}(U; A) \rightarrow D_{\et}(Y;A)$ has cohomological dimension $\leq d + 1$. 
\end{lemma}

\begin{proof}
The statement is local, so we may immediately reduce to the case that $Y$ is a spatial diamond; the pullback of $Z\inj Y$ remains relatively extra small.

By Lemma \ref{lem: extra small open embedding}, there is a nested countable sequence of quasicompact open embeddings $j_n \co U_n  \inj Y$ such that $U_n \subset U$ and $\colim_n U_n = U$. Let $i_n \co U_n \inj U$ denote the inclusion, so that $j_n = j \circ i_n$. Hence we have 
\[
\rR j_* \cong \Rlim_n \rR j_{n*} i_n^* \co D_{\et}(U;A) \rightarrow D_{\et}(Y;A).
\]
By \cite[Lemma 21.13]{Sch17}, the (derived) direct image along a \emph{quasi-compact} open embedding is exact, hence has cohomological dimension zero. This applies to $j_n$ for every $n$. 

The desired vanishing bounds may be checked after derived reduction modulo $\ell$, so it is enough to treat $A/\ell A$-modules. But $A/\ell A$ is a finite-dimensional $k$-algebra, so the assumed $\ell$-cohomological dimension bound also bounds the cohomology of sheaves of $A$-modules. 

Therefore, it suffices to see that the derived inverse limit functor $\Rlim_n$, along a countable index set, has finite cohomological dimension. Thanks to this countability, $\Rlim_n A_n$ can be expressed as 
\begin{equation}\label{eq: countable Rlim}
\Rlim_n A_n \cong \mrm{cone}(\rR{\prod_n} A_n \rightarrow \rR{\prod_n} A_n)[-1]
\end{equation}
where $\rR\hspace{-.1cm}\prod_n$ is the derived functor of the product, defined as the product of injective resolutions of the terms. By the standard generalities in homological algebra, $\rR\hspace{-.1cm}\prod_n$ can be calculated using a product of acyclic (for the global sections functor) resolutions, which can be taken to be of length bounded by the cohomological dimension of $Y$, which by assumption is $d<\infty$. Then from \eqref{eq: countable Rlim} we see that $\Rlim_n$ has cohomological dimension at most $d+1$.

\end{proof}

\begin{cor}\label{cor: open finite tor-dimension}
In the situation of Lemma \ref{lem: tor dim j_*}, if $\cK \in D_{\et}^b(U; A)$ has finite tor-dimension then $\rR j_* \cK \in D_{\et}(Y; A)$ has finite tor-dimension. 
\end{cor}

\begin{proof}
This is a special case of \cite[Expos\'e XVII, Th\'eor\`eme 5.2.11]{SGA4}, which applies because $\rR j_*$ has finite cohomological dimension by Lemma \ref{lem: tor dim j_*} and the fact \cite[Proposition 14.3]{Sch17} that $Y_{\et}$ has enough points, thanks to $Y$ being locally spatial. 


\end{proof}

\subsection{Fixed points and equivariant sheaves} For an endomorphism $\sigma$ of a v-stack $Y$, we write $\Fix(\sigma, Y) := Y^\sigma$ for the fibered product 
\[
\begin{tikzcd}
Y^\sigma \ar[r] \ar[d] & Y \ar[d, "\Id \times \sigma"] \\
Y \ar[r, "\Delta"] & Y \times Y
\end{tikzcd}
\]
Observe that if $Y$ is a separated locally spatial diamond, then $\Delta$ is a closed embedding, hence $Y^\sigma  \rightarrow Y$ is a closed embedding. We will only be interested in this situation. From now on, we assume that $\sigma$ is an automorphism of order $\ell$ and denote by $\Sigma = \tw{\sigma}$ the cyclic group that it generates. Recall the equivariant bounded derived category $D_{\et}^b(Y; \Lambda)^{B \Sigma}$, which is the homotopy category of $\cD_{\et}^b(Y; \Lambda)^{B \Sigma}$, the limit of the action diagram $B\Sigma\rightarrow\msf{Cat}_\infty$ on the $\infty$-category $\cD_{\et}^b(Y;\Lambda)$ of \cite[\S 17]{Sch17}. 

The $\Sigma$-action on $Y$ induces the trivial $\Sigma$-action on $Y^\sigma$, for which we have an equivalence of derived categories
\begin{equation}\label{eq: derived cat for trivial action}
D_{\et}^b(Y^\sigma;\Lambda)^{B \Sigma} \cong D^b_{\et}(Y^\sigma; \Lambda[\Sigma]).
\end{equation}
Indeed, this is a consequence of the following more general Lemma. 

\begin{lemma}\label{lem:equivariant-derived-category-trivial-action}
Let $Y$ be a v-stack equipped with the trivial action of a finite discrete group $\Gamma$. Then we have a natural equivalence of derived categories
\begin{equation}\label{eq:equivariant-derived-category-trivial-action}
D_{\et}(Y;\Lambda)^{B \Gamma} \cong D_{\et}(Y; \Lambda[\Gamma]).
\end{equation}
where $\Lambda[\Gamma]$ is the group ring of $\Gamma$ over $\Lambda$.
\end{lemma}

\begin{proof}
By definition (recall \S \ref{sssec:etale-sheaves}), both sides of \eqref{eq:equivariant-derived-category-trivial-action} are defined by v-descent from strictly totally disconnected spaces, so we reduce to the case where $Y$ is strictly totally disconnected. For discrete coefficients, $D_{\et}(Y;\Lambda)$ is (by definition) the derived category of the abelian category of $\Lambda$-modules on $Y_{\et}$, and similarly for $\Lambda[\Gamma]$. 

Since $\Gamma$ is discrete, $D_{\et}(Y;\Lambda)^{B \Gamma}$ is the derived category of the abelian category of $\Gamma$-equivariant $\Lambda$-module sheaves on $Y_{\et}$. It therefore suffices to identify this abelian category with the abelian category of $\Lambda[\Gamma]$-module sheaves on $Y_{\et}$. Since the $\Gamma$-action on $Y$ is trivial, a $\Gamma$-equivariant structure on an \'etale sheaf $\cF$ on $Y$ is tautologically equivalent to an action of $\Gamma$ on $\cF$ over $Y$, or in other words, a $\Lambda[\Gamma]$-module structure on $\cF$, as desired. The adic case follows by passing to the inverse limit from the case of torsion coefficients, since formation of $\Gamma$-equivariant objects commutes with limits.
\end{proof}

\begin{remark}\label{rem:Cartesian-fixed-points}
If $f \co Y \rightarrow Y'$ is a $\sigma$-equivariant map, then it induces a commutative diagram
\[
\begin{tikzcd}
Y^\sigma \ar[r] \ar[d] & Y \ar[d] \\
(Y')^\sigma \ar[r] & Y'
\end{tikzcd}
\]
In general, this is not Cartesian. However, if $Y,Y'$ are separated locally spatial diamonds and $f$ is a monomorphism, then it \emph{is} Cartesian. (This special case will be used in the sequel.) Indeed, in this case we have $Y^\sigma(S) = Y(S)^\sigma$ for any affinoid perfectoid space $S$, and 
\[
Y(S) \times_{Y'(S)}  (Y')^\sigma(S)  = Y(S) \cap Y'(S)^\sigma = Y(S)^\sigma = Y^\sigma(S)
\]
\end{remark}

\subsection{Bounding tor-amplitude}\label{ssec: permanence big tate category}
Here we establish technical results that show the ``permanence'' of finite tor-dimension under various operations.

Let $\Lambda$ be a finite commutative $W(k)$-algebra; we will be most interested in the cases $\Lambda = k$ or $W(k)$. Recall that $\Sigma$ is a cyclic group of order $\ell$ and $\Lambda[\Sigma]$ denotes its group ring. 


\begin{defn} Let $Y$ be a small v-stack. We define $\Flat^b(Y; \Lambda[\Sigma]) \subset D^b_{\et}(Y; \Lambda[\Sigma])$ to be the full subcategory consisting of complexes with finite tor-dimension \cite[Tag 0651]{stacks-project}, i.e., which have tor-amplitude in $[a,b]$ for some $a,b \in \Z$.
\end{defn}

For $\Lambda = k$, the following Lemma (pointed out by Jesper Grodal) gives an alternative characterization of the subcategory $\Flat^b(Y; k[\Sigma])$, showing that it is determined by conditions on geometric stalks. 

\begin{lemma}\label{lem: flat =  finite tor dimension k large} Assume that $Y$ is a locally spatial diamond. Then the subcategory $\Flat^b(Y; k[\Sigma] ) \subset D^b_{\et}(Y; k[\Sigma])$ coincides with the full subcategory spanned by objects whose stalks at all geometric points have finite tor-dimension over $k[\Sigma]$. 
\end{lemma}

\begin{proof} We apply \cite[Tag 0DJJ(1)]{stacks-project} which immediately gives one containment: if $\cK \in D^b_{\et}(Y; k[\Sigma])$ has finite tor-dimension over $k[\Sigma]$, then all its stalks at geometric points have finite tor-dimension over $k[\Sigma]$. 

Next we show the other containment. Since $Y$ is a locally spatial diamond, \cite[Proposition 14.3]{Sch17} ensures that $Y_{\et}$ has enough points. Therefore we may apply \cite[Tag 0DJJ(2)]{stacks-project}: it suffices to show that if all the geometric stalks of $\cK \in D^b_{\et}(Y; k[\Sigma])$ have finite tor-amplitude, then there is a \emph{uniform} bound on the tor-amplitude of the geometric stalks. To see this, since $\cK$ is bounded we may pick a bounded complex representing it, say concentrated in degrees $[a,b]$. It then suffices to show: if $\cK$ is a complex of $k[\Sigma]$-modules concentrated in degrees $[a,b]$ and has finite tor-dimension, then in fact its tor-amplitude lies in $[a,b]$. 

Since $k[\Sigma]$ is Artinian, the properties of being flat and projective over $k[\Sigma]$ coincide by \cite[Tag 051E]{stacks-project}. Therefore $\cK$ also has finite projective dimension \cite[Tag 0A5M]{stacks-project}. Furthermore, since $k[\Sigma]$ has finitistic dimension $0$ (by \cite{GR71}, the finitistic dimension of a commutative ring is bounded by the Krull dimension), the projective amplitude of $\cK$ lies in $[a,b]$. Hence $\cK$ is represented by a complex of projective $k[\Sigma]$-modules supported in degrees $[a,b]$, and therefore has tor-amplitude in $[a,b]$.
\end{proof}

\begin{lemma}\label{lem: f_! finite cd}
Let $A$ be a finite commutative $\OO$-algebra. Let $Y$ and $S$ be locally spatial diamonds, and let $f \co Y \rightarrow S$ be shriekable with $\dimg f<\infty$. 

(1) Then $\rR f_! \co D_{\et}^b(Y; A) \rightarrow D_{\et}^b(S; A)$ has cohomological dimension $\leq 3 \dimg f$. 

(2) If furthermore $S$ has $\ell$-cohomological dimension $\leq d$ and $f$ is $\omega_0$-locally spatial, then $\rR f_* \co D_{\et}^b(Y; A) \rightarrow D_{\et}^b(S; A)$ has cohomological dimension $\leq 3 \dimg f + d + 1$. 
\end{lemma}

\begin{proof}As in the proof of Lemma \ref{lem: tor dim j_*}, $\ell$-adic devissage reduces the relevant cohomological dimension bounds for $A$-coefficients to the case of $A/\ell A$-coefficients.

(1) By the hypothesis that $Y$ and $S$ are locally spatial, we may write $Y \cong \colim_i Y_i$ as a filtered colimit of spatial open subdiamonds $j_i \co Y_i \inj Y$. We have by definition \cite[p.134-135]{Sch17} that 
\[
\rR f_! \cong \colim_i \rR (f|_{Y_i})_! j_i^*.
\]
By \cite[Theorem 22.5]{Sch17}, each $\rR(f|_{Y_i})_!$ has cohomological dimension at most $3 \dimg f$. Since $j_i^*$ is exact and filtered colimits are exact, we deduce that $\rR f_!$ also has cohomological dimension at most $3 \dimg f$. 

(2) Similarly to (1), we have 
\[
\rR f_* \cong \Rlim_i \rR(f|_{Y_i})_* j_i^*
\]
where now we may arrange the indexing set to be countable, thanks to the hypothesis that $f$ is $\omega_0$-locally spatial. Each $\rR(f|_{Y_i})_* j_i^*$ has cohomological dimension at most $3\dimg f$ by \cite[Theorem 22.5]{Sch17}. Then we bound the cohomological dimension of the derived inverse limit as in the proof of Lemma \ref{lem: tor dim j_*}.

\end{proof}

\begin{lemma}\label{lem: pushforward preserves perfect complexes}
Let $A$ be a finite commutative $\OO$-algebra. Let $Y$ and $S$ be locally spatial diamonds, and let $f \co Y \rightarrow S$ be shriekable with $\dimg f<\infty$.

(1) Then $\rR f_! \co D^b_{\et}(Y; A) \rightarrow D^b_{\et}(S; A)$ carries $\Flat^b(Y; A)$ to $\Flat^b(S; A)$.

(2) If furthermore $S$ is $\ell$-bounded and $f$ is $\omega_0$-locally spatial, then $\rR f_* $ carries $\Flat^b(Y; A)$ to $\Flat^b(S; A)$. 
\end{lemma}

\begin{proof}
(1) Suppose $\cF \in \Flat^b(Y; A)$. We need to verify that for any sheaf $M$ (concentrated in degree $0$) in $D^b_{\et}(S; A)$, we have $\rH^{-i}(M \dotimes_A \rR f_! \cF) = 0$ for $i \gg 0$. By the projection formula \cite[Proposition 22.23]{Sch17}, we have 
\[
\rH^{-i}(M \dotimes_A \rR f_! \cF) \cong \rH^{-i}(\rR f_! (f^*M \dotimes_A \cF)).
\]
By the assumption that $\cF \in \Flat^b(Y; A)$, $f^*M \dotimes_A \cF$ is bounded. Then by Lemma \ref{lem: f_! finite cd}(1), $\rR f_! (f^*M \dotimes_A \cF)$ is bounded, so $\rH^{-i}(\rR f_! (f^*M \dotimes_A \cF)) = 0$ for $i \gg 0$, as desired.

(2) This follows from \cite[Expos\'e XVII, Th\'eor\`eme 5.2.11]{SGA4}, which applies because $\rR f_*$ has finite cohomological dimension by Lemma \ref{lem: f_! finite cd} and the fact \cite[Proposition 14.3]{Sch17} that $S_{\et}$ has enough points, thanks to $S$ being locally spatial. 
\end{proof}

\begin{lemma}\label{lem: big Tate 1}
Let $Y$ be a locally spatial diamond with a free action of $\Sigma$. Let $Y/\Sigma$ be the quotient in the category of diamonds, and $q \co Y \rightarrow Y / \Sigma$ denote the quotient map. Then for any $\cF \in D^b_{\et}(Y; k)^{B \Sigma}$, we have $q_* \cF \in \Flat^b(Y/\Sigma; k[\Sigma])  \subset D^b_{\et}(Y/\Sigma; k[\Sigma])$, where the $k[\Sigma]$-module structure on $q_*\cF$ is induced by the equivariant structure on $\cF$. 
\end{lemma}

\begin{proof}
By Lemma \ref{lem: flat =  finite tor dimension k large}, it suffices to check finite tor-dimension on geometric stalks. Let $\ol{y} \rightarrow Y/\Sigma$ be a geometric point, and choose a geometric point $\ol{x} \rightarrow Y$ over $\ol{y}$. The fiber of $q$ over $\ol{y}$ is a free $\Sigma$-orbit, and the equivariant structure on $\cF$ identifies
\[
(q_*\cF)_{\ol{y}} \cong \bigoplus_{\tau \in \Sigma} \cF_{\tau \ol{x}} \cong k[\Sigma] \otimes_k \cF_{\ol{x}}
\]
as a complex of $k[\Sigma]$-modules. Since every $k$-vector space is free, this is a bounded complex of free $k[\Sigma]$-modules, hence has finite tor-dimension.

\end{proof}

\begin{cor}\label{cor: pushforward from free is flat}
Let $Y$ and $S$ be locally spatial diamonds, and let $f \co Y \rightarrow S$ be shriekable with $\dimg f<\infty$. Suppose $\Sigma$ acts trivially on $S$ and freely on $Y$, and $f$ is $\Sigma$-equivariant. 

(1) Let $\Lambda\in\{k,\OO\}$. If $\cF\in D^b_{\et}(Y;\Lambda)^{B\Sigma}$ has finite tor-amplitude over $\Lambda$, then $\rR f_!\cF$ has finite tor-amplitude over $\Lambda[\Sigma]$. In particular, for $\Lambda=k$, the functor $\rR f_! \co D^b_{\et}(Y;k)^{B\Sigma}\rightarrow D^b_{\et}(S;k[\Sigma])$ lands in $\Flat^b(S;k[\Sigma])$.

(2) If furthermore $S$ has $\ell$-cohomological dimension $\leq d$ and $f$ is $\omega_0$-locally spatial, then $\rR f_* \co D^b_{\et}(Y; k)^{B\Sigma} \rightarrow D^b_{\et}(S; k[\Sigma])$ lands in $\Flat^b(S; k[\Sigma])$. 
\end{cor}

\begin{proof}
By the hypotheses, we may factor $f$ as the composition of $\Sigma$-equivariant maps 
\[
Y \xrightarrow{q} Y/\Sigma \xrightarrow{\ol{f}} S.
\]
Since $q$ is finite, $q_! = q_*$. 

For (1), let $\cF$ be as in the statement. Since $\Sigma$ acts freely on $Y$, the map $q \co Y \rightarrow Y/\Sigma$ is an \'etale $\Sigma$-torsor. Therefore, there exists an \'etale cover $V\to Y/\Sigma$ trivializing $q$, and choosing a section $s\co V\to Y$ gives
\[
(q_*\cF)|_V\cong \Lambda[\Sigma]\dotimes_\Lambda s^*\cF.
\]
Since finite tor-amplitude is étale-local, $q_*\cF$ has finite tor-amplitude over $\Lambda[\Sigma]$. Lemma \ref{lem: pushforward preserves perfect complexes}(1), applied to $\ol{f}$ with coefficients $\Lambda[\Sigma]$, shows that $\rR\ol{f}_!q_*\cF=\rR f_!\cF$ has finite tor-amplitude over $\Lambda[\Sigma]$. For (2), apply Lemma \ref{lem: big Tate 1} to $q_*$ and Lemma \ref{lem: pushforward preserves perfect complexes}(2) to $\rR\ol{f}_*$ with $A:= k[\Sigma]$.
\end{proof}

\begin{lemma}\label{lem: big Tate 2}Let $Y$ be an $\ell$-bounded, separated, $\omega_0$-locally spatial diamond with an action of $\Sigma$. Assume that $Y/\Sigma$ is $\ell$-bounded and $\omega_0$-locally spatial and that the induced map $\bar i\co Y^\sigma\inj Y/\Sigma$ is a relatively extra small closed immersion.
Let $U := Y \setminus Y^{\sigma}$ and $j \co U \inj Y$ be its inclusion into $Y$. Then for any $\cF \in D^b_{\et}(U; k)^{B\Sigma}$, we have $i^*\rR j_* \cF \in \Flat^b(Y^\sigma; k[\Sigma])$. 

\end{lemma}

\begin{proof}
    Let $q\colon Y\to Y/\Sigma$ be the quotient map, and let
$i\colon Y^\sigma\hookrightarrow Y$ be the fixed-point closed immersion. Let $\ol i$ be the composition $Y^\sigma \xrightarrow{i} Y \xrightarrow{q} Y/\Sigma$.

We claim that the commutative square
\begin{equation}\label{eq:totally-ramified-fiber}
\begin{tikzcd}
Y^\sigma \ar[d,equals] \ar[r,"i"] & Y \ar[d,"q"] \\
Y^\sigma \ar[r,"\ol i"] & Y/\Sigma
\end{tikzcd}
\end{equation}
is Cartesian. (Informally speaking, this expresses that $q$ is ``totally ramified over $Y^\sigma \subset Y$''.) Temporarily denote the fibered product by $Y'$; then the commutative diagram induces a map $Y^\sigma \rightarrow Y'$, which we want to show is an isomorphism. We first observe that it is qcqs: in the sequence of maps $Y^\sigma \rightarrow Y' \rightarrow Y$, the second map is qcqs (being the base change of the qcqs map $Y^\sigma \rightarrow Y/\Sigma$) and the composite map is a closed embedding, hence $Y^\sigma \rightarrow Y'$ is qcqs, by the ``cancellation property'' for qcqs morphisms. Then by \cite[Lemma 11.11]{Sch17}, we can check that $Y^\sigma \rightarrow Y'$ is an isomorphism by checking that it induces a bijection on $\Spa(C,C^+)$-points for all geometric points $\Spa(C,C^+)$. Let $(y,z)$ be a geometric point of $Y'$. The equality $q(y)=q(z)$ in the v-sheaf quotient means that, after a v-cover of the test space, there is some $\tau\in\Sigma$ for which $y=\tau z$. Since $z$ is $\Sigma$-fixed, this means that $y=z$. Thus $Y^\sigma\to Y'$ induces a bijection on geometric points, and proves the claim.

Next we claim that $\ol i$ is a closed immersion. This can be checked v-locally on the target \cite[Proposition 10.10(i)]{Sch17}. The Cartesian square \eqref{eq:totally-ramified-fiber} indicates that base changing $\ol i$ along the v-cover $Y \rightarrow Y/\Sigma$ gives the closed embedding $i \co Y^\sigma \inj Y$, establishing the claim.

Applying base change to \eqref{eq:totally-ramified-fiber}, we have an isomorphism
\begin{equation}\label{eq: big tate 2 eq 1}
i^* \rR j_* \cF \cong  \ol{i}^* q_* \rR j_* \cF  \in D^b_{\et}(Y^\sigma; k[\Sigma]).
\end{equation}
From the commutativity of the diagram 
\[
\begin{tikzcd}
U \ar[r, hook, "j"] \ar[d, "q_U"] & Y \ar[d, "q"] \\
U/\Sigma \ar[r, hook, "\ol{j}"] & Y/\Sigma
\end{tikzcd}
\]
$\ol{i}^*q_* \rR j_* \cF  \cong \ol{i}^*\rR\ol{j}_* q_{U*} \cF$. Now Lemma \ref{lem: big Tate 1} implies that $q_{U*} \cF$ has finite tor-dimension. The relative extra smallness of $\bar i$ allows us to apply Corollary \ref{cor: open finite tor-dimension} to deduce that $\rR\ol{j}_* q_{U*} \cF$ also has finite tor-dimension, hence $\ol{i}^* \rR\ol{j}_* q_{U*} \cF$ also has finite tor-dimension. It is bounded because $q_{U*}\cF$ is bounded (as $\cF$ is bounded and $q_U$ is a finite map) and $\rR \ol{j}_*$ has finite cohomological dimension by Lemma \ref{lem: tor dim j_*}. Then \eqref{eq: big tate 2 eq 1} implies that $i^* \rR j_* \cF$ has finite tor-dimension and is bounded, hence lies in $\Flat^b(Y^\sigma; k[\Sigma])$. 

\end{proof}

\begin{lemma}\label{lem: * vs !}
Retain the notation and assumptions from Lemma \ref{lem: big Tate 2}. For any $\cK \in D^b_{\et}(Y; k)^{B\Sigma}$, if we view $i^* \cK, i^! \cK \in D^b_{\et}(Y^\sigma; k)^{B \Sigma} \stackrel{\eqref{eq: derived cat for trivial action}}\cong D^b_{\et}(Y^\sigma; k[\Sigma])$,  then the cone of the natural map $i^! \cK \rightarrow i^* \cK$ belongs to $\Flat^b(Y^\sigma; k[\Sigma]) \subset D^b_{\et}(Y^\sigma; k[\Sigma])$. 
\end{lemma}

\begin{proof}
We will first recall the construction of the natural map. Let $j \co Y \setminus Y^\sigma \inj Y$. Consider the exact triangle $i_* i^! \cK \rightarrow \cK \rightarrow \rR j_* j^* \cK$ in $D_{\et}^b(Y; k)^{B\Sigma}$. Applying $i^*$ to it yields the exact triangle in $D^b_{\et}(Y^{\sigma}; k)^{B \Sigma}$,
\[
i^! \cK \rightarrow i^* \cK \rightarrow i^* \rR j_* j^* \cK,
\]
in which the first map is the one from the statement of Lemma \ref{lem: * vs !}. Then Lemma \ref{lem: big Tate 2} implies that $\mrm{Cone}(i^! \cK \rightarrow i^* \cK) \cong i^* \rR j_* j^* \cK $ lies in $\Flat^b(Y^{\sigma}; k[\Sigma])$, as desired. 
\end{proof}

\subsection{The Tate category}\label{ssec: Tate category}\label{ssec: functors on Tate cat}
\begin{defn}\label{def: Tate category} We define the \emph{(large) Tate category of $Y$} (with respect to $\Lambda$) to be the Verdier quotient category 
\[
\Shv(Y; \cT_{\Lambda}) := D^b_{\et}(Y; \Lambda[\Sigma]) / \Flat^b(Y; \Lambda[\Sigma]).
\]
We denote the tautological projection map from $D^b_{\et}(Y; \Lambda[\Sigma])$ to $\Shv(Y; \cT_{\Lambda})$ by 
\[
\TT^* \co D^b_{\et}(Y; \Lambda[\Sigma]) \rightarrow \Shv(Y; \cT_{\Lambda}).
\]
\end{defn}

\begin{remark}The applications will all use $\Lambda=k$, but the results in this subsection hold with the stated more general coefficients. 
\end{remark}

\begin{defn}\label{def: TT and epsilon}
Let $\epsilon \co \Lambda[\Sigma] \surj \Lambda$ be the augmentation and $\epsilon^* \co D^b_{\et}(Y; \Lambda) \rightarrow D^b_{\et}(Y; \Lambda[\Sigma])$ be pullback along $\epsilon$. Define the functor 
\[
\TT := \TT^* \epsilon^* \co  D^b_{\et}(Y; \Lambda) \rightarrow \Shv(Y; \cT_{\Lambda}).
\] 
\end{defn}

\begin{remark}
The notation $\TT, \TT^*, \epsilon^*$ follows that of Leslie-Lonergan \cite{LL}. 
\end{remark}

Let $f \co Y \rightarrow S$ denote a $\Sigma$-equivariant map of separated locally spatial diamonds with $\Sigma$-action.

\subsubsection{Pullback} Since the pullback functor $f^{\sigma *} \co D^b_{\et}(S^{\sigma}; k)^{B \Sigma} \rightarrow D^b_{\et}(Y^{\sigma}; k)^{B \Sigma}$ preserves stalks, Lemma \ref{lem: flat =  finite tor dimension k large} implies that it carries $\Flat^b(S^\sigma; k[\Sigma])$ to $\Flat^b(Y^\sigma; k[\Sigma])$, hence induces a functor 
\[
f^{\sigma *} \co \Shv(S^{\sigma}; \Cal{T}_{k}) \rightarrow \Shv(Y^{\sigma}; \Cal{T}_{k}).
\]

\subsubsection{Pushforward} Assume furthermore that $f \co Y \rightarrow S$ is shriekable with $\dimg f<\infty$. By Lemma \ref{lem: pushforward preserves perfect complexes}(1), the functor $\rR f^\sigma_! \co D^b_{\et}(Y^{\sigma}; k[\Sigma]) \rightarrow D^b_{\et}(S^{\sigma}; k[\Sigma])$ carries $\Flat^b(Y^\sigma; k[\Sigma])$ to $\Flat^b(S^\sigma; k[\Sigma])$, hence induces a functor
\[
\rR f^\sigma_! \co \Shv(Y^{\sigma}; \cT_{k}) \rightarrow \Shv(S^{\sigma}; \cT_{k}).
\]
{If in addition $f^\sigma$ is $\omega_0$-locally spatial and $S^\sigma$ is $\ell$-bounded, then Lemma \ref{lem: pushforward preserves perfect complexes}(2) gives a functor}
\[
\rR f^\sigma_* \co \Shv(Y^{\sigma}; \cT_{k}) \rightarrow \Shv(S^{\sigma}; \cT_{k}).
\]

\subsection{The Smith operation}\label{ssec: smith operation} We define the analogue in $p$-adic geometry of the ``Smith operation'' $\Psm$ introduced by Treumann in \cite{Tr19}.

\begin{defn}
Let $Y$ be a separated locally spatial diamond with an action of $\Sigma$ and $i \co Y^\sigma \inj Y$ the inclusion of its $\Sigma$-fixed points. The \emph{Smith operation} is the functor 
\begin{equation}\label{eq: Psm for diamonds}
\Psm  := \TT^* \circ i^* \co D_{\et}^b(Y; k)^{B \Sigma} \rightarrow \Shv(Y^{\sigma}; \Cal{T}_{k})
\end{equation}
defined as the composition of $i^* \co D_{\et}^b(Y; k)^{B \Sigma} \rightarrow D_{\et}^b(Y^{\sigma}; k)^{B \Sigma} \stackrel{\eqref{eq: derived cat for trivial action}}\cong D^b_{\et}(Y^{\sigma}; k[\Sigma])$ with the projection $\TT^*$ to $\Shv(Y^{\sigma}; \Cal{T}_{k})$.
\end{defn}

The following result may be viewed as a form of equivariant localization for spaces with $\Sigma$-action, comparing the (relative) cohomology of a space with that of its $\Sigma$-fixed points.

\begin{prop}\label{prop: equivariant localization} Let $Y$ and $S$ be separated locally spatial diamonds with actions of $\Sigma$, and let $f \co Y \rightarrow S$ be $\Sigma$-equivariant and shriekable with $\dimg f<\infty$. Denote by $f^{\sigma} \co Y^{\sigma} \rightarrow S^\sigma$ the induced map on fixed points. 

(1) Then the following diagram commutes: 
\[
\begin{tikzcd}
D^b_{\et}(Y;k)^{B \Sigma} \ar[d, "\Psm"] \ar[r, "\rR f_!"]  &  D^b_{\et}(S;k)^{B \Sigma}   \ar[d, "\Psm"]  \\
\Shv(Y^{\sigma}; \Cal{T}_k)  \ar[r, "{\rR f^\sigma_!}
"]  & \Shv(S^{\sigma}; \Cal{T}_k)
\end{tikzcd} 
\]

(2) Assume that $\Sigma$ acts trivially on $S$, that $f$ is $\omega_0$-locally spatial, and that $S$, $Y$, and $Y/\Sigma$ are $\ell$-bounded and $\omega_0$-locally spatial. Assume also that $\bar i\co Y^\sigma\inj Y/\Sigma$ is relatively extra small. Then the following diagram commutes:
\[
\begin{tikzcd}
D^b_{\et}(Y;k)^{B \Sigma} \ar[d, "\Psm"] \ar[r, "\rR f_*"]  &  D^b_{\et}(S;k)^{B \Sigma}   \ar[d, "\Psm"]  \\
\Shv(Y^{\sigma}; \Cal{T}_k)  \ar[r, "{\rR f^\sigma_*}
"]  & \Shv(S^{\sigma}; \Cal{T}_k)
\end{tikzcd} 
\]

\end{prop}

\begin{proof} For (1), base change for $\rR f_!$ allows us to replace $S$ by $S^\sigma$. In (2), the action on $S$ is trivial by hypothesis.

(1) Let $\cF \in D_{\et}^b(Y;k)^{B \Sigma}$. Writing $j \co (Y \setminus Y^\sigma) \inj Y$ for the open complement of $i \co Y^{\sigma} \inj Y$, we have an exact triangle
\[
j_! j^* \cF  \rightarrow  \cF \rightarrow  i_* i^* \cF \in D^b_{\et}(Y; k)^{B\Sigma}.
\] 
By definition $\Sigma$ acts freely on $Y \setminus Y^\sigma$, hence Corollary \ref{cor: pushforward from free is flat}(1), applied to $f \circ j \co Y \setminus Y^\sigma \to S$, implies that $\rR f_!(j_! j^* \cF) \cong \rR(f\circ j)_!(j^*\cF)$ lies in $\Flat^b(S; k[\Sigma])$. Then the cone of $\rR f_!  \cF \rightarrow \rR f^{\sigma}_! (i^*  \cF)$ lies in $\Flat^b(S; k[\Sigma])$, and therefore becomes $0$ in $\Shv(S; \cT_{k})$. Therefore
\[
\TT^*(\rR f_! \cF) \cong  \TT^* ( \rR f^{\sigma}_! (i^* \cF)) \cong {\rR f^\sigma_!} \Psm(\cF) \in \Shv(S; \cT_{k}),
\]
which exactly expresses the desired commutativity. 

(2) The argument is similar to that for (1), but using instead the exact triangle 
\[
i_* i^! \cF  \rightarrow  \cF \rightarrow  \rR j_* j^* \cF.
\] 
The Cartesian square \eqref{eq:totally-ramified-fiber} identifies $i\co Y^\sigma\inj Y$ with the base change of $\bar i$, so $i$ is relatively extra small. Hence $j$ and therefore $f\circ j$ are $\omega_0$-locally spatial. Corollary \ref{cor: pushforward from free is flat}(2), applied to $f \circ j \co Y \setminus Y^\sigma \to S$, shows that the object $\rR f_* \rR j_* j^* \cF \cong \rR(f\circ j)_*(j^*\cF)$ lies in $\Flat^b(S; k[\Sigma])$, and therefore projects to $0$ in $\Shv(S; \cT_{k})$. Lemma \ref{lem: * vs !} shows that the first arrow in the zigzag
\[
\rR f_*i_*i^*\cF\longleftarrow \rR f_*i_*i^!\cF\longrightarrow \rR f_*\cF
\]
is obtained by applying $\rR f^\sigma_*$ to a morphism whose cone lies in $\Flat^b(Y^\sigma;k[\Sigma])$. The map $f^\sigma$ is $\omega_0$-locally spatial, and since $S$ is $\ell$-bounded, Lemma \ref{lem: pushforward preserves perfect complexes}(2) implies that the first arrow still has cone in $\Flat^b(S;k[\Sigma])$ after applying $\rR f^\sigma_*$. Since the action of $\Sigma$ on $Y \setminus Y^\sigma$ is free, the same holds for the second arrow. Both arrows therefore become isomorphisms in $\Shv(S;\cT_k)$, which proves the claimed comparison.

\end{proof}

\section{The Tate construction}\label{sec: Tate cohomology}

In this section we define the Tate construction and relate it to Tate cohomology. We show that it factors through Smith--Treumann localization, which we use to establish equivariant localization results for the Tate construction.

\subsection{Tate cohomology}\label{ssec: tate cohomology}\label{sssec:tate-cohomology-of-complexes}
For a $\Lambda[\Sigma]$-module $M$, its \emph{Tate cohomology groups} were defined in \eqref{eq: Tate cohomology}. We extend the definition 2-periodically to define $\rT^j(M)$ for all $j \in \Z$. 

\subsubsection{Tate cohomology of complexes} The Tate cohomology of a \emph{complex} of $\Lambda[\Sigma]$-modules is defined in \cite[\S 3.4.1]{Fe23}. We summarize below. The exact sequence of $\Lambda[\Sigma]$-modules
\begin{equation}\label{eq:beta-SES}
0 \rightarrow \Lambda \rightarrow \Lambda[\Sigma] \xrightarrow{1-\sigma} \Lambda[\Sigma] \rightarrow \Lambda \rightarrow 0
\end{equation}
induces a morphism 
\begin{equation}\label{eq: lambda shift by 2}
\beta \co \Lambda \rightarrow \Lambda[2] \in D^b(\Lambda[\Sigma]).
\end{equation}
Note that it becomes an isomorphism in the Tate category $D^b(\Lambda[\Sigma])/\Flat^b(\Lambda[\Sigma])$, since the middle two terms project to $0$. Given a bounded-below complex of $\Lambda[\Sigma]$-modules $C^{\bu}$, we define its \emph{Tate cohomology} to be  
\begin{align}\label{eq:tate-cohomology-of-complex}
\rT^n(C^{\bu}) &= \colim_{j \rightarrow \infty} \Hom_{D^+(\Lambda[\Sigma])}(\Lambda, C^{\bu}[n+2j]) 
\end{align}
where the transition maps are the $\beta$ from \eqref{eq: lambda shift by 2}. Evidently $\rT^n(C^{\bu})$ is 2-periodic in $n$, and we occasionally view the indexing as being $n \in \Z/2\Z$. It is clear that this construction descends to the derived category.

Now we state some assertions that may be specific\footnote{Our argument uses $\Lambda = k$, but the assertions could be true in greater generality, as far as we know.} to $\Lambda = k$. By \cite[Tag 051E]{stacks-project}, a module over $k[\Sigma]$ is flat if and only if it is free. Since Tate cohomology of free $k[\Sigma]$-complexes vanishes (by inspection), this construction further factors through the Tate category, inducing 
\[
\rT^n \co \Shv(\pt; \cT_{k}) \rightarrow  \mrm{Vect}_{k},
\]
where $\pt$ denotes a geometric point $\Spa(C,C^+)$, so that $D^b_{\et}(\pt; k[\Sigma]) \cong D^b(k[\Sigma])$.
We write $\rT^* := \bigoplus_n \rT^n$ and $\rH^* := \bigoplus_n \rH^n$. 

\begin{remark}\label{rem: tate filtration}
Since filtered colimits are exact, an exact triangle $C' \rightarrow C \rightarrow C''$ of bounded-below complexes induces a $2$-periodic long exact sequence
\[
\cdots \rightarrow \rT^n(C') \rightarrow \rT^n(C) \rightarrow \rT^n(C'') \rightarrow \rT^{n+1}(C') \rightarrow \cdots.
\]
If $C$ is bounded, then applying this to the truncation filtration of $C$ gives a finite filtration of $\rT^n(C)$, functorial in $C$, whose graded pieces are subquotients of the $\rT^{n-i}(\rH^i(C))$, hence in particular subquotients of the $\rH^i(C)$.
\end{remark}

\subsubsection{Tate cohomology sheaves}\label{sssec: tate cohomology sheaves}
Let $S$ be a locally spatial diamond with trivial $\Sigma$-action, so that $\Sigma$-equivariant complexes on $S$ may be regarded as complexes of $\Lambda[\Sigma]$-module sheaves by Lemma \ref{lem:equivariant-derived-category-trivial-action}. Given $\Cal{F} \in D^+_{\et}(S; \Lambda[\Sigma])$, we define \emph{Tate cohomology sheaves} 
\[
\rT^n (\Cal{F}) := \colim_{j \rightarrow \infty} \cHom_{D^+_{\et}(S; \Lambda[\Sigma])}(\Lambda, \Cal{F}[n+2j]) 
\]
where the transition maps are induced by \eqref{eq: lambda shift by 2}. The $\rT^i(\Cal{F})$ are \'{e}tale sheaves of $\rT^0(\Lambda)$-modules, where $\rT^0(\Lambda)$ is the 0th Tate cohomology of the trivial $\Sigma$-module $\Lambda$.

\begin{remark}[Tate cohomology as Hom in the Tate category]\label{rem: tate cohomology as hom}
If $C^{\bu} \in  D^b(k[\Sigma])$ is bounded, then the argument of \cite[Proposition 4.5.1]{LL} gives a natural isomorphism 
\[
\rT^i(C^{\bu}) \cong \Hom_{\Shv(\pt; \cT_k)}(k, \TT^* C^{\bu}[i]).
\]
\end{remark}	

\begin{remark}For $\Lambda = k$ and $\Cal{F} \in D^b_{\et}(S;k[\Sigma])$, as in Remark \ref{rem: tate cohomology as hom} we also have the description
\[
\rT^i (\Cal{F}) \cong \cHom_{\Shv(S; \cT_k)}(k, \TT^* \Cal{F}[i]).
\]
\end{remark}

\subsection{The Tate construction}\label{ssec:tate-realization}

Our upcoming definitions fit into the very general formalism of the \emph{Tate construction}. In order to contextualize them properly, we recall this formalism, following the discussion of \cite[\S I.1]{NS18}. Let $X$ be a $\Sigma$-equivariant object of a stable $\infty$-category admitting limits and colimits indexed by $B \Sigma$. The \emph{Tate construction} for $X$ is \cite[Definitions I.1.10 and I.1.13]{NS18}
\[
X^{t \Sigma} := \cofib(X_{h \Sigma} \xrightarrow{N} X^{h \Sigma}),
\]
where $N$ is the norm map.

We will apply the Tate construction in the following specific situation. Let $Y$ be a small v-stack with trivial $\Sigma$-action, so that $D_{\et}(Y;k)^{B\Sigma}\cong D_{\et}(Y;k[\Sigma])$ by Lemma \ref{lem:equivariant-derived-category-trivial-action}. For $\cF\in D_{\et}(Y;k[\Sigma])$ we have $\cF_{h\Sigma}=\cF\dotimes_{k[\Sigma]}k$ and $\cF^{h\Sigma}=\cRHom_{k[\Sigma]}(k,\cF)$, so that
\begin{equation}\label{eq: modular Tate realization}
\cF^{t\Sigma}=\cofib\left(\cF\dotimes_{k[\Sigma]}k\xrightarrow{N}\cRHom_{k[\Sigma]}(k,\cF)\right)\in D_{\et}(Y;k).
\end{equation}

\begin{example}\label{ex:degree-0-tate-construction} If $\cF = M$ is a sheaf in degree $0$, then we have
\begin{equation}\label{eq:degree-0-tate-norm}
\begin{tikzcd}[ampersand replacement=\&, column sep=1.2em, row sep=1.4em]
\text{degree} \& \cdots \& -2 \& -1 \& 0 \& 1 \& 2 \& \cdots \\
M_{h\Sigma}\cong \& \cdots \ar[r,"1-\sigma"] \& M \ar[r,"N"] \& M \ar[r,"1-\sigma"] \& M \ar[r] \ar[d,"N", dashed] \& 0 \ar[r] \& 0 \ar[r] \& \cdots \\
M^{h\Sigma}\cong \& \cdots \ar[r] \& 0 \ar[r] \& 0 \ar[r] \& M \ar[r,"1-\sigma"] \& M \ar[r,"N"] \& M \ar[r,"1-\sigma"] \& \cdots
\end{tikzcd}
\end{equation}
and the norm map $M_{h \Sigma}  \rightarrow M^{h \Sigma}$ is represented by $N = 1 + \sigma + \ldots + \sigma^{\ell-1}$ in degree $0$, and $0$ in the other degrees. Thus, we have 
\begin{equation}\label{eq:degree-0-tate-complex}
M^{t \Sigma} \cong [ \ldots \xrightarrow{1-\sigma} M \xrightarrow{N} M \xrightarrow{1-\sigma} M \xrightarrow{N} M \xrightarrow{1-\sigma} \ldots ]
\end{equation}
\end{example}

In general, since $(-)_{h \Sigma}$ and $(-)^{h \Sigma}$ are exact functors, and the norm map is a natural transformation between them, $\cF^{t \Sigma}$ can be computed explicitly by picking a complex representing $\cF$, applying Example \ref{ex:degree-0-tate-construction} to the terms to obtain a double complex, and then totalizing. We see from the construction that even for bounded $\cF$, the Tate construction need not be bounded. 

\begin{lemma}\label{lem: bounded Tate commutation}
Let $F\co\cC\rightarrow\cD$ be an exact, colimit-preserving functor between presentable stable $\infty$-categories equipped with t-structures. Assume that $\cD$ is right separated, i.e. $\bigcap_n\cD^{\geq n}=0$, and that there exists $c\in\Z$ such that $F(\cC^{\geq n})\subset\cD^{\geq n+c}$ for all $n$. Let $\Gamma$ be a finite group and $A\in\cC^{\geq a}$ an object with $\Gamma$-action. Then the natural maps
\begin{equation}
F(A^{h\Gamma})\rightarrow(FA)^{h\Gamma},\qquad (FA)_{h\Gamma} \rightarrow F(A_{h\Gamma})
\end{equation}
are isomorphisms, inducing an isomorphism 
\begin{equation}\label{eq: bounded Tate commutation}
F(A^{t\Gamma}) \cong (FA)^{t\Gamma}
\end{equation}
\end{lemma}

\begin{proof} Let $\Delta$ be the simplex category and $\Delta_{\leq s} \subset \Delta$ be the full subcategory on $[0]$, \ldots, $[s]$. For a cosimplicial object $A^{\bu} \co \Delta \rightarrow \cC$, we write 
  \[
  \Tot_s(A^{\bu}) := \lim_{\Delta_{\leq s}} A^{\bu}. 
  \]
Restriction along $\Delta_{\leq s} \inj \Delta_{\leq s+1}$ gives a tower 
\[
\ldots \rightarrow \Tot_2(A^{\bu}) \rightarrow \Tot_1(A^{\bu}) \rightarrow \Tot_0(A^{\bu}) , \quad \text{such that} \quad \Tot(A^{\bu}) \cong \lim_s \Tot_s(A^{\bu}).
\]
For $A^{\bu}$ the cosimplicial object $[r]\mapsto\prod_{\Gamma^r}A$, we have $A^{h\Gamma}=\limit_s \Tot_s(A^{\bu})$. The fiber of the restriction map 
\[
\Tot_s(A^{\bu})\rightarrow \Tot_{s-1}(A^{\bu})
\]
is the shift by $-s$ of a direct summand of a finite product of copies of $A$, so it lies in $\cC^{\geq a+s}$. Since $\cC^{\geq m}$ is closed under limits, $\fib(A^{h\Gamma}\rightarrow \Tot_s(A^{\bu}))$ lies in $\cC^{\geq a+s+1}$. Since $F$ is exact, $F(\Tot_s(A^{\bu}))$ is the $s$th partial totalization for $FA$, so the fiber of the map $F(A^{h\Gamma})\rightarrow(FA)^{h\Gamma}$ lies in $\cD^{\geq n}$ for every $n$, hence vanishes by the assumption that $\cD$ is separated. 

The homotopy orbits are the colimit over $\Delta^{\mrm{op}}$ of the simplicial bar construction $[r]\mapsto\bigoplus_{\Gamma^r}A$, so preservation of colimits by $F$ gives the isomorphism $(FA)_{h\Gamma}\rightarrow F(A_{h\Gamma})$. 

Finally, both comparisons are compatible with the norm maps, which gives the isomorphism $F(A^{t\Gamma}) \cong (FA)^{t\Gamma}$. 
\end{proof}

\begin{example}\label{ex: Tate commutation instances}
Lemma \ref{lem: bounded Tate commutation} applies, with $\Gamma=\Sigma$ and the standard t-structures, to the following functors:
\begin{enumerate}
\item $g^*\co D_{\et}(S;k)\rightarrow D_{\et}(Y;k)$ for any map $g\co Y\rightarrow S$ of small v-stacks;
\item $\rR f_!\co D_{\et}(Y;k)\rightarrow D_{\et}(S;k)$ for a shriekable map $f$ of locally spatial diamonds with $\dimg f<\infty$;
\item $\cL\dotimes_k(-)$ on $D_{\et}(Y;k)$, for $\cL\in D^b_{\et}(Y;k)$.
\end{enumerate}
In particular, for $\cF\in D^+_{\et}(S;k[\Sigma])$ and $g\co S'\rightarrow S$ we have $g^*(\cF^{t\Sigma})\cong(g^*\cF)^{t\Sigma}$ and $g^*\rT^n(\cF)\cong\rT^n(g^*\cF)$. Thus the stalks of $\rT^n(\cF)$ at geometric points are the Tate cohomology groups of the stalks of $\cF$.

\end{example}

\begin{prop}\label{prop: Tate construction computes Tate cohomology}
Let $Y$ be a small v-stack with trivial $\Sigma$-action and $\cF\in D^b_{\et}(Y;k[\Sigma])$. Let $\beta\co\cF^{h\Sigma}\rightarrow\cF^{h\Sigma}[2]$ be the map from \eqref{eq: lambda shift by 2}. The canonical map $\cF^{h\Sigma}\rightarrow\cF^{t\Sigma}$ extends to an isomorphism
\begin{equation}\label{eq: tate periodicity colimit}
\colim\bigl(\cF^{h\Sigma}\xrightarrow{\beta}\cF^{h\Sigma}[2]\xrightarrow{\beta}\cF^{h\Sigma}[4]\rightarrow\cdots\bigr)\xrightarrow{\sim}\cF^{t\Sigma}.
\end{equation}
Consequently $\cF^{t\Sigma}$ is $2$-periodic, $\rH^n(\cF^{t\Sigma})\cong\rT^n(\cF)$ naturally in $\cF$, and $\cF^{t\Sigma}=0$ if and only if all $\rT^n(\cF)$ vanish.
\end{prop}

\begin{proof}
By Example \ref{ex: Tate commutation instances}(1), all the constructions involved commute with pullback to a v-cover by strictly totally disconnected perfectoid spaces. The statements in question may be checked after such a cover, so we may and do assume that $Y$ is strictly totally disconnected. In this case we have $D_{\et}(Y;k[\Sigma])=D(Y_{\et};k[\Sigma])$ \cite[Propositions 14.11(iii) and 14.15]{Sch17}, so we may represent $\cF$ by a bounded complex of \'etale sheaves.

From \eqref{eq:beta-SES} we see that the cone of the map $\beta \co \cF \rightarrow \cF[2]$, obtained by tensoring $\cF$ with \eqref{eq:beta-SES} over $k$, with the diagonal $\Sigma$-action, is built from induced objects of the form $k[\Sigma] \otimes_k \cF$, whose Tate construction vanishes. This implies that $\beta$ acts invertibly on $\cF^{t \Sigma}$, hence the natural transformation $\cF^{h \Sigma} \rightarrow \cF^{t \Sigma}$ factors through a natural transformation 
\begin{equation}\label{eq:beta-inverted}
\colim_\beta \cF^{h\Sigma} \rightarrow \cF^{t \Sigma}.
\end{equation}
To show that it is an isomorphism, we may by d\'evissage reduce (since $\cF$ is bounded) to the case where $\cF$ is an \'etale sheaf in degree $0$. In this case, it follows from the description in Example \ref{ex:degree-0-tate-construction} that \eqref{eq:beta-inverted} is an isomorphism.

The remaining assertions follow by passing to cohomology sheaves, which commute with filtered colimits and detect isomorphisms.
\end{proof}

\begin{prop}\label{prop: modular realization}
The functor $\cF\mapsto\cF^{t\Sigma}$ of \eqref{eq: modular Tate realization} factors through an exact ``realization'' functor
\begin{equation}\label{eq: realization functor}
\Real_Y\co\Shv(Y;\cT_k)\rightarrow D_{\et}(Y;k),\qquad\Real_Y(\TT^*\cF)=\cF^{t\Sigma}.
\end{equation}
We have $\rH^j\Real_Y(\TT^*\cF)\cong\rT^j(\cF)$, and $\Real_Y$ commutes with pullback along maps $Y'\rightarrow Y$, naturally in $\cF$.
\end{prop}

\begin{proof}
Let $\cF\in\Flat^b(Y;k[\Sigma])$. Each stalk of $\cF$ at a geometric point is a bounded complex of finite tor-dimension over $k[\Sigma]$, hence is represented by a bounded complex of free $k[\Sigma]$-modules, and so has vanishing Tate cohomology. By Example \ref{ex: Tate commutation instances} and Proposition \ref{prop: Tate construction computes Tate cohomology}, $\cF^{t\Sigma}=0$, since vanishing may be checked after pullback to a v-cover by a locally spatial diamond, where geometric stalks are conservative. Hence the exact functor $(-)^{t\Sigma}$ factors through the Verdier quotient $\Shv(Y;\cT_k)$. The remaining assertions follow from the same two results.
\end{proof}

\subsection{Equivariant localization}\label{ssec: Tate localization} We translate the results of \S \ref{ssec: smith operation} into equivariant localization results for the Tate construction. 

\begin{prop}\label{prop: realized Smith localization}
Let $Y$ and $S$ be separated locally spatial diamonds with $\Sigma$-action, and let $f\co Y\rightarrow S$ be $\Sigma$-equivariant and shriekable with $\dimg f<\infty$. Write $i_Y\co Y^\sigma\inj Y$ and $i_S\co S^\sigma\inj S$ for the inclusions of the fixed loci. For $\cF\in D^b_{\et}(Y;k)^{B\Sigma}$ there is an isomorphism, natural in $\cF$ and compatible with base change in $S$,
\begin{equation}\label{eq: realized Smith localization}
\bigl(i_S^*\rR f_!\cF\bigr)^{t\Sigma}\xrightarrow{\sim}\rR f^\sigma_!\bigl((i_Y^*\cF)^{t\Sigma}\bigr).
\end{equation}
In particular, if $\Sigma$ acts trivially on $S$, then $(\rR f_!\cF)^{t\Sigma}\cong\rR f^\sigma_!\bigl((i_Y^*\cF)^{t\Sigma}\bigr)$.
\end{prop}

\begin{proof}
Apply $\Real_{S^\sigma}$ to the isomorphism $\Psm(\rR f_!\cF)\cong\rR f^\sigma_!\Psm(\cF)$ of Proposition \ref{prop: equivariant localization}(1). By \eqref{eq: Psm for diamonds} and \eqref{eq: realization functor}, the left side becomes $(i_S^*\rR f_!\cF)^{t\Sigma}$ and the right side becomes $(\rR f^\sigma_!i_Y^*\cF)^{t\Sigma}$, which is naturally isomorphic to $\rR f^\sigma_!((i_Y^*\cF)^{t\Sigma})$ by Example \ref{ex: Tate commutation instances}(2). 

To justify the compatibility with base change in $S$: note that the isomorphism of Proposition \ref{prop: equivariant localization}(1) comes from base change and restriction to the fixed locus, so the desired compatibility follows from Proposition \ref{prop: modular realization} and base change for $\rR f_!$.
\end{proof}

\begin{cor}[Localization for Tate cohomology]\label{cor: Tate cohomology localization}
Let $f\co Y\rightarrow S$ be a $\Sigma$-equivariant shriekable map of separated locally spatial diamonds with $\dimg f<\infty$, where $\Sigma$ acts trivially on $S$. Let $i\co Y^\sigma\inj Y$ be the inclusion and $f^\sigma=f\circ i$. For $\cF\in D^b_{\et}(Y;k)^{B\Sigma}$, restriction to $Y^\sigma$ induces isomorphisms
\begin{equation}\label{eq: Tate cohomology localization}
\rT^n(\rR f_!\cF)\xrightarrow{\sim}\rT^n\bigl(\rR f^\sigma_!i^*\cF\bigr),\qquad n\in\Z.
\end{equation}
In particular, if $S$ is a geometric point then $\rT^n(\RGamma_c(Y;\cF))\cong\rT^n(\RGamma_c(Y^\sigma;i^*\cF))$.
\end{cor}

\begin{proof}
Let $j$ be the open embedding of the complement of $Y^\sigma$. Then we have an exact triangle 
\[
\rR(f\circ j)_!j^*\cF \rightarrow \rR f_!\cF\rightarrow\rR f^\sigma_!i^*\cF,
\]
and the leftmost term lies in $\Flat^b(S;k[\Sigma])$ by Corollary \ref{cor: pushforward from free is flat}(1), because $\Sigma$ acts freely on $Y\setminus Y^\sigma$. Its Tate cohomology vanishes by Proposition \ref{prop: modular realization}, and then the assertion follows from the long exact sequence of Remark \ref{rem: tate filtration}.
\end{proof}

\begin{remark}\label{rem: equivariance of localization}
The isomorphisms \eqref{eq: realized Smith localization} and \eqref{eq: Tate cohomology localization} are equivariant for any action of a group $\Gamma$ on $Y$, $S$ and $\cF$ that commutes with the $\Sigma$-actions and is compatible with $f$. 
\end{remark}

\begin{remark}\label{rem: localization other settings}
The same argument gives localization for compactly supported cohomology of separated schemes of finite type over an algebraically closed field of characteristic $p$, with coefficients in $k$; cf.\ \cite[\S A.1.2]{Fe23}. This will be used in \S \ref{ssec: Tate of deep level DL}.
\end{remark}

\subsection{Multiplicativity}\label{ssec: Tate cup products} We establish compatibility of the Tate construction with tensor products and cup products.

\begin{prop}\label{prop: Tate cup compatibility}Let $Y$ and $Z$ be small v-stacks with trivial $\Sigma$-actions, and let $\cF\in D^b_{\et}(Y;k[\Sigma])$ and $\cG\in D^b_{\et}(Z;k[\Sigma])$. Equip $\cF\boxtimes\cG$ with the diagonal $\Sigma$-action. Then there is a natural exterior cup product in $D_{\et}(Y\times Z;k)$
\begin{equation}\label{eq: exterior Tate cup}
\cF^{t\Sigma}\boxtimes\cG^{t\Sigma}\rightarrow(\cF\boxtimes\cG)^{t\Sigma}.
\end{equation}
It is compatible with pullback, and for shriekable maps $f\co Y\rightarrow S$ and $g\co Z\rightarrow T$ of finite $\dimg$, it is compatible with the K\"unneth isomorphism
\begin{equation}\label{eq: Kunneth compact supports}
\rR f_!\cF\boxtimes\rR g_!\cG\xrightarrow{\sim}\rR(f\times g)_!(\cF\boxtimes\cG)
\end{equation}
and with the projection formula for $\rR f_!$. For $\Sigma$-equivariant maps as in Proposition \ref{prop: realized Smith localization}, these products are compatible with the isomorphisms \eqref{eq: realized Smith localization}.
\end{prop}

\begin{proof}
The proof of \cite[Theorem I.3.1]{NS18}, with spectra replaced by $D_{\et}(Y;k)$ (using \cite[Theorem I.3.6 and Lemma I.3.8]{NS18}), shows that the Tate construction \eqref{eq: modular Tate realization} is lax symmetric monoidal, compatibly with the canonical map from homotopy invariants, and that this structure is unique. This constructs \eqref{eq: exterior Tate cup}.

The map \eqref{eq: exterior Tate cup} is the composite of $\cF^{t\Sigma}\boxtimes\cG^{t\Sigma}\cong(p_1^*\cF)^{t\Sigma}\otimes(p_2^*\cG)^{t\Sigma}$, given by Example \ref{ex: Tate commutation instances}(1), with the above lax monoidal structure. The comparison maps of Lemma \ref{lem: bounded Tate commutation} are induced by the natural maps for homotopy invariants and coinvariants, and the lax structure on $(-)^{t\Sigma}$ is determined by the one on $(-)^{h\Sigma}$ through the universal property of \cite[Theorem I.3.6(ii)]{NS18}. The K\"unneth and projection-formula compatibilities therefore reduce to those for $(-)^{h\Sigma}$.
 
 Note that the restriction maps $\cF\rightarrow i_{Y*}i_Y^*\cF$ commute with exterior products and with \eqref{eq: Kunneth compact supports}. Applying $(-)^{t\Sigma}$ and Lemma \ref{lem: bounded Tate commutation} gives the compatibility with \eqref{eq: realized Smith localization}.
\end{proof}

\subsection{Some special cases}\label{ssec: Tate coefficient calculations}
We explicate Tate cohomology in some special cases.

\begin{example}[Tate cohomology of trivial coefficients]\label{ex: Tate cohomology of trivial coeff}
Equip $k$ with the trivial $\Sigma$-action. (We emphasize our running notation from \S \ref{sec: notation} that $|\Sigma|=\ell$ and $\ch k = \ell$.) 
\begin{equation}\label{eq: modular Tate coefficients}
\rT^*(k)=
\begin{cases}
k[\beta,\beta^{-1}]\otimes_k\bigwedge_k(\eta),&\ell>2,\\
k[\eta,\eta^{-1}],\quad\beta=\eta^2,&\ell=2
\end{cases}
\end{equation}
where $\beta\in\rT^2(k)$ is induced by \eqref{eq: lambda shift by 2}, and $\eta\in\rT^1(k)$ is the class sending $\sigma$ to $1$.

Equip $\OO$ with the trivial $\Sigma$-action. Then we have
\[
\rT^j(\OO)  \cong \begin{cases} k & j \equiv 0 \pmod{2}, \\ 0 & j \equiv 1 \pmod{2}. \end{cases}
\]
\end{example}

\begin{example}[Tate cohomology of trivial complexes]\label{ex: tate coh of trivial}
If $C^{\bullet} \in D^b(\OO)$, then (cf. Definition \ref{def: TT and epsilon} for notation) we have 
\[
\rT^*(\epsilon^* C^\bullet) \cong \rH^*(\mrm{Tot}(\ldots \xrightarrow{\ell} C^{\bu} \xrightarrow{0} C^{\bu} \xrightarrow{\ell} C^{\bu}  \xrightarrow{0} \ldots))
\]
so that 
\[
\rT^n(\epsilon^* C^\bullet) \cong  \bigoplus_{i \equiv n \pmod{2}} \rH^i(C^{\bu} \dotimes_{\OO} k). 
\]

If $C^{\bullet} \in D^b(k)$, then we have 
\begin{equation}\label{eq: tate trivial k-module}
\rT^*(\epsilon^* C^\bullet) \cong \rH^*(C^{\bullet}) \otimes_k \rT^*(k) 
\end{equation}
where $k$ is equipped with the trivial $\Sigma$-action. By Example \ref{ex: Tate cohomology of trivial coeff}, \eqref{eq: tate trivial k-module} simplifies to  
\[
\rT^n(\epsilon^* C^\bullet) \cong \bigoplus_{i \in \Z} \rH^i(C^{\bullet}) .
\]
\end{example}

For $\cL\in D^b_{\et}(Y;k)$ with trivial $\Sigma$-action, Example \ref{ex: Tate commutation instances}(3) gives a natural isomorphism
\begin{equation}\label{eq: realized trivial coefficients}
(\epsilon^*\cL)^{t\Sigma}\cong\cL\dotimes_kk_Y^{t\Sigma},\qquad k_Y^{t\Sigma}:=(\epsilon^*k_Y)^{t\Sigma}.
\end{equation}
By Example \ref{ex: Tate cohomology of trivial coeff}, $k_Y^{t\Sigma}$ has cohomology sheaf $k_Y$ in every degree, for every $\ell$.

\begin{example}\label{ex: Tate cohomology of Nm}
Let $V$ be a $k$-vector space. Then $V^{\otimes \ell}$ has an action of $\Sigma$, where $\sigma$ acts by cyclic rotation of the factors, and a straightforward calculation shows that 
\[
\rT^i(V^{\otimes \ell}) \cong V^{(\ell)} := V \otimes_{k, \Frob_\ell} k,
\]
the \emph{Frobenius twist} of $V$. Indeed, choosing a basis of $V$ splits $V^{\otimes \ell}$ into a sum of free and trivial $k[\Sigma]$-modules. The Tate cohomology of the free modules vanish, and the Tate cohomology of trivial modules is calculated by Example \ref{ex: Tate cohomology of trivial coeff}. 
\end{example}

\section{$\Sigma$-fixed point calculations}\label{sec: fixed point}

 Let $G$ be a reductive group over $E$ and $\sigma$ an order $\ell \neq p$ automorphism of $G$ such that $H := G^\sigma$ is (connected) reductive.
 
 \begin{example}If $G$ is semisimple and simply connected, then $H$ is automatically connected reductive. Indeed, after base change to $\ol E$, smoothness of $G^\sigma$ and reductivity of $(G^\sigma)^\circ$ follow from \cite[Proposition VIII.5.14]{FS} in positive characteristic and are classical in characteristic zero. Since $\sigma$ has finite order invertible in $\ol E$, it is semisimple, and Steinberg's theorem \cite[Theorem 8.1]{St68} implies that $G^\sigma$ is connected.
 \end{example}
 
In this section we will analyze the fixed points of $\sigma$ on various spaces affiliated with $G$, and relate them to the analogous spaces affiliated with $H$, such as:
\begin{itemize}
\item The $\bdrp$-affine Grassmannians, as well as their variants such as Beilinson-Drinfeld affine Grassmannians, convolution affine Grassmannians, and ``twisted'' affine Grassmannians (which are the target of the generalized Grothendieck-Messing period maps). 
\item The moduli spaces of local shtukas.
\end{itemize}
These calculations are used later when applying Smith--Treumann localization to such spaces. 

\subsection{Affine Grassmannians}\label{ssec: fixed points affine grassmannians} We begin by formulating the results for affine Grassmannians and their variants.

\begin{prop}\label{prop: algebraic gr fixed points} Let $\CC/E$ be an algebraically closed field. Let $\Gr_{G, \CC}^{\alg}, \Gr_{H, \CC}^{\alg}$ be the algebraic affine Grassmannians over $\CC$. Then the map $\Gr_{H,\CC}^{\alg} \rightarrow \Fix(\sigma, \Gr_{G,\CC}^{\alg})$ is an isomorphism on underlying reduced ind-schemes. 
\end{prop}

We first deduce the Beilinson--Drinfeld variants of Proposition \ref{prop: algebraic gr fixed points}, and then return to prove Proposition \ref{prop: algebraic gr fixed points}.

Let $I$ be a non-empty finite set and let $S$ be a separated locally spatial diamond over $(\Div_{X}^1)^I$. Then we define 
\[
\Gr_{G, S/(\Div^1_X)^I} := \Gr_{G, (\Div^1_X)^I} \times_{(\Div_X^1)^I} S.
\]

\begin{remark}Note that these affine Grassmannians are partially proper, so that their $\Spa(C,C^+)$-points identify with their $\Spa(C, C^\circ)$-points, where $C^\circ \subset C$ is the ring of power-bounded elements (i.e., the maximal possible $C^+$). The fixed locus is closed in the partially proper diamond $\Gr_{G,C}$, and is therefore partially proper as well. For this reason, we do not encounter subtleties with higher rank points. 
\end{remark}

\begin{prop}\label{prop: BD gr fixed points}
Let $S$ be a separated locally spatial diamond over $(\Div_{X}^1)^I$. Then the natural map 
\[
\Gr_{H, S/(\Div^1_X)^I} \rightarrow \Fix(\sigma, \Gr_{G, S/(\Div^1_X)^I})
\]
is an isomorphism.  
\end{prop}

\begin{proof}
Since $\Gr_{G, S/(\Div^1_X)^I}$ is separated, $\Fix(\sigma, \Gr_{G, S/(\Div^1_X)^I})$ is a closed subdiamond of $\Gr_{G, S/(\Div^1_X)^I}$. The map $
\Gr_{H, S/(\Div^1_X)^I} \rightarrow \Fix(\sigma, \Gr_{G, S/(\Div^1_X)^I})$ is a closed embedding by the argument of \cite[Lemma 19.1.5]{SW20}, hence qcqs, so by \cite[Lemma 11.11]{Sch17} it is an isomorphism if and only if it induces a bijection on $\Spa(C, C^+)$-points for all $(C,C^+)$. Checking this reduces to the case where $S = \Spa(C,C^+)$, in which case we abbreviate $\Gr_{G, C}:= \Gr_{G, \Spa(C,C^+)/(\Div^1_X)^I}$ and similarly for $H$.

A $\Spa(C,C^+)$-point over $(\Div_X^1)^I$ gives an $I$-tuple of untilts $\{C_i^\sharp/E\}_{i\in I}$. Let $J$ be the set of distinct untilts occurring in this tuple, and write $C_\alpha^\sharp$ for the untilt corresponding to $\alpha\in J$. By the factorization property of the Beilinson--Drinfeld Grassmannian, a point of $\Gr_{G,C}$ is equivalently a collection, for $\alpha\in J$, of $(C_\alpha^\sharp,\cO_{C_\alpha^\sharp})$-points of the one-leg $\bdrp$-affine Grassmannian for $G$. Choose an identification $B^+_{\dR}(C_\alpha^\sharp)\cong C_\alpha^\sharp[[t]]$ as in the proof of \cite[Proposition 19.2.1]{SW20}, this comparison is functorial in the group and compatible with the automorphism $\sigma$. Then, for each $\alpha\in J$, we obtain a commutative diagram
\[
\begin{tikzcd}
\Gr_{H}^{\bdrp}(C^\sharp_\alpha, \cO_{C^\sharp_\alpha}) \ar[r, hook] \ar[d, equals] & \Fix(\sigma, \Gr_{G}^{\bdrp})(C^\sharp_\alpha, \cO_{C^\sharp_\alpha})  \ar[r, hook] \ar[d, equals] &   \Gr_{G}^{\bdrp}(C^\sharp_\alpha, \cO_{C^\sharp_\alpha}) \ar[d, equals] \\
\Gr_{H,C^\sharp_\alpha}^{\alg}(C^{\sharp}_\alpha) \ar[r, hook] & \Fix(\sigma, \Gr_{G,C^\sharp_\alpha}^{\alg})(C^{\sharp}_\alpha) \ar[r, hook] & \Gr_{G,C^\sharp_\alpha}^{\alg}(C^{\sharp}_\alpha)
\end{tikzcd}
\]
By Proposition \ref{prop: algebraic gr fixed points}, the map $\Gr_{H, C^\sharp_\alpha}^{\alg}(C^{\sharp}_\alpha) \rightarrow \Fix(\sigma, \Gr_{G, C^\sharp_\alpha}^{\alg})(C^{\sharp}_\alpha)$ is a bijection, hence so is the parallel map in the top row. Taking the product over the finitely many distinct legs $\alpha\in J$ gives the desired bijection for the fiber over the collection of untilts $\{C_i^\sharp/E\}_{i\in I}$. Thus the map 
\[
\Gr_{H, C}(C,C^+) \rightarrow \Fix(\sigma, \Gr_{G,C})(C,C^+)
\]
is also a bijection, as desired. 
\end{proof}

\begin{remark}
The same argument works for variants like $(\Div_{\cY}^1)^I$ instead of $(\Div_X^1)^I$. 
\end{remark}

We also deduce variants for convolution Grassmannians and the ``twisted'' Grassmannians $\Gr^{\twi}_G$ from \cite[\S 23.5]{SW20}.

\begin{cor}\label{cor: fixed points conv Gr} Let $m  \in \Z_{\geq 1}$. 
\begin{enumerate}
\item Let $\CC/E$ be an algebraically closed field. Let $\wt{\Gr}_{G, \CC}^{\alg, (m)}$ be the $m$-step convolution Grassmannian for $G$, and similarly for $H$. Then the natural map 
\[
\wt{\Gr}_{H, \CC}^{\alg, (m)}  \rightarrow \Fix(\sigma, \wt{\Gr}_{G, \CC}^{\alg, (m)} )
\]
is an isomorphism on underlying reduced ind-schemes.
\item Let $I$ be a non-empty finite set and let $S$ be a separated locally spatial diamond with a map $S \rightarrow (\Div_{X}^1)^I$. Let $\wt{\Gr}_{G, S/(\Div_{X}^1)^I}^{(m)}$ be the $m$-step convolution Beilinson-Drinfeld Grassmannian for $G$, and similarly for $H$. Then the natural map 
\[
\wt{\Gr}_{H, S/(\Div_{X}^1)^I}^{(m)} \rightarrow \Fix(\sigma, \wt{\Gr}_{G, S/(\Div_{X}^1)^I}^{(m)})
\]
is an isomorphism. 

\item Let $I$ be a non-empty finite set. Let $\Gr_{G,I}^{\twi} = \colim_{\mu_{\bu}} \Gr_{G, \prod_{i \in I} \Spd E_i, \leq \mu_{\bu}}^{\twi}$ be the twisted affine Grassmannian (cf. \cite[Definition 23.5.1]{SW20}). Then the natural map
\[
\Gr_{H, I}^{\twi} \xrightarrow{\sim} \Fix(\sigma, \Gr_{G,I}^{\twi})
\]
is an isomorphism. 
\end{enumerate}
\end{cor}

\begin{proof}
(1) The map $\wt{\Gr}_{G}^{\alg, (m)} \rightarrow (\Gr_G^{\alg})^m$ sending $(\cE^0 \dashrightarrow \cE_1 \dashrightarrow \ldots \dashrightarrow \cE_m)$ to $(\cE^0 \dashrightarrow \cE_1, \cE^0 \dashrightarrow \cE_2, \ldots, \cE^0 \dashrightarrow \cE_m)$ is a $\Sigma$-equivariant isomorphism. Under this isomorphism, the statement transports to (the $m$th power of) that of Proposition \ref{prop: algebraic gr fixed points}.

(2) As in the proof of Proposition \ref{prop: BD gr fixed points}, the statement reduces to the case where $S = \Spa(C, C^+)$. Indeed, over a geometric point the convolution Grassmannian factorizes over the finitely many distinct untilts. Then it follows from the algebraic case over each untilt, which was treated in (1). 

(3) As in the proof of Proposition \ref{prop: BD gr fixed points}, the statement reduces to the case $S = \Spa(C,C^+)$. It is then a consequence of (2).
\end{proof}

\subsection{Borels} Recall that a \emph{Borus} of a reductive group $G$ is a pair $(B,T)$ of a Borel subgroup $B <G$ and a split maximal torus $T < B$. 

\begin{lemma}\label{lem: fixed borus}
Suppose $\Sigma$ stabilizes a Borus $(B,T)$ of $G$. Then $(B^\sigma, T^\sigma)$ is a Borus of $H$. 
\end{lemma}

\begin{proof}
Suppose for the moment that we know $B^\sigma$ is a Borel subgroup of $H$. We will argue that $T^\sigma$ is a split maximal torus of $H$. Let $U$ be the unipotent radical of $B$. Applying \cite[Proposition A.8.10(2)]{CGP15}, which implies that the map on $\Sigma$-fixed points of a smooth $\Sigma$-equivariant morphism is again smooth, to $B \rightarrow B/U \cong T$ shows that $B^\sigma \rightarrow (B/U)^\sigma  \cong T^\sigma$ is a smooth map whose non-empty fibers are $U^\sigma$-torsors. Since the section $T^\sigma \inj B^\sigma$ shows that $B^\sigma \rightarrow T^\sigma$ is surjective, we conclude that the map $B^\sigma/U^\sigma \rightarrow (B/U)^\sigma \cong T^\sigma$ is an isomorphism. Note that $U^\sigma$ is a normal unipotent subgroup of $B^\sigma$ and is smooth by \cite[Proposition A.8.10(2)]{CGP15}. It is also connected: since $\sigma$ preserves $(B,T)$, $U$ has a $\Sigma$-stable filtration by vector groups, whose $\Sigma$-fixed points are (connected) vector subgroups. Thus $U^\sigma$ is contained in the unipotent radical of $B^\sigma$; on the other hand, the image of the unipotent radical in $B^\sigma/U^\sigma \cong T^\sigma$ is a unipotent subgroup of a torus, hence must necessarily be trivial. This shows that $T^\sigma$ is a maximal torus of $H$ (and in particular is connected).

Now we return to showing that $B^\sigma$ is a Borel subgroup of $H$. The group $B^\sigma$ is smooth by \cite[Proposition A.8.10(2)]{CGP15} and is solvable. By assumption $G/B$ is projective, hence its closed subscheme $(G/B)^\sigma$ is projective. We will prove that $H/B^\sigma$ is closed and open in $(G/B)^\sigma$, which will imply that $H/B^\sigma$ is projective. Evidently $H$ acts by translation on $(G/B)^\sigma$, and it suffices to show that each $H$-orbit in $(G/B)^\sigma$ is open, since then the orbits are finite in number and pairwise disjoint, hence open-closed. Indeed, for any geometric point $g$ of $G$ representing a point in $(G/B)^\sigma$, the orbit map $G^\sigma \rightarrow (G/B)^\sigma$ through $gB$ is a smooth morphism by \cite[Proposition A.8.10(2)]{CGP15} again, hence has open image. 

Indeed, the natural map
\[
H/B^{\sigma,\circ}\longrightarrow H/B^\sigma
\]
is finite. Hence projectivity of $H/B^\sigma$ implies that $B^{\sigma,\circ}$ is parabolic. Since it is solvable, it is a Borel subgroup of $H$. 
It remains to show that $B^\sigma = B^{\sigma,\circ}$ is connected. The group $B^\sigma$ normalizes its identity component, whereas a parabolic subgroup is self-normalizing by \cite[Theorem 1.3.1]{CF}. Thus we may conclude that $B^\sigma=B^{\sigma,\circ}$, completing the proof. 
\end{proof}

\begin{prop}\label{prop: GrB fixed points}
Let $S$ be a separated locally spatial diamond over $(\Div_{X}^1)^I$. Suppose that $G$ has a $\Sigma$-stable Borel subgroup $B$. Let $B_H := B^\sigma$, a Borel subgroup of $H$ by Lemma \ref{lem: fixed borus}. Then the natural map 
\[
\Gr_{B_H, S/(\Div_{X}^1)^I} \rightarrow \Fix(\sigma, \Gr_{B, S/(\Div_{X}^1)^I})
\]
is an isomorphism. 
\end{prop}

\begin{proof} 
In the proof of \cite[Lemma 19.1.5]{SW20} there is a general argument for when a closed embedding of affine algebraic groups $G_1 \inj G_2$ induces a closed embedding of affine Grassmannians. More precisely, the formulation of \emph{loc. cit.} requires the groups to be reductive, but the proof requires only that $G_2/G_1$ be affine and that $G_1$-torsors on $\Spec B_{\dR}^+(R^\sharp)$ be \'etale locally trivial on $\Spa (R,R^+)$. We verify these conditions in the present case. 

First we show that $B/B_H$ is affine. This may be checked after base change to $\ol E$. We may write
\[
B=T\ltimes U,\qquad B_H=T^\sigma\ltimes U^\sigma,
\]
for a maximal torus $T \subset B$ and the unipotent radical $U \subset B$. Then 
\[
B/B_H\simeq T\times^{T^\sigma}(U/U^\sigma).
\]
The homogeneous space $U/U^\sigma$ is clearly affine. Choosing a splitting of the sequence
\[
1\longrightarrow T^\sigma\longrightarrow T\longrightarrow T/T^\sigma\longrightarrow 1
\]
gives an isomorphism $T\times^{T^\sigma}(U/U^\sigma) \cong (T/T^\sigma)\times(U/U^\sigma)$, which is clearly affine. 

Second, the local-triviality argument for $B_H$-torsors used in the proof of \cite[Lemma 19.1.5]{SW20} requires only that $B_H$ be smooth: a $B_H$-torsor over $B^+_{\dR}(R^\sharp)$ acquires a section after an \'etale localization of $R^\sharp$, and this section lifts across the henselian pair $(B^+_{\dR}(R^\sharp),\ker\theta)$. Hence it is trivial \'etale-locally on $\Spa(R,R^+)$. Therefore, \cite[Lemma 19.1.5]{SW20} applies to show that the map 
\[
\Gr_{B_H, S/(\Div_{X}^1)^I} \rightarrow \Fix(\sigma, \Gr_{B, S/(\Div_{X}^1)^I})
\]
is a closed embedding, hence qcqs. Then by \cite[Lemma 11.11]{Sch17} it suffices to check that it induces a bijection on geometric points. A variation on \cite[Proposition VI.3.1]{FS} shows that the maps $\Gr_{B_H, S/(\Div_{X}^1)^I} \rightarrow \Gr_{H, S/(\Div_{X}^1)^I}$ and $\Gr_{B, S/(\Div_{X}^1)^I} \rightarrow \Gr_{G, S/(\Div_{X}^1)^I}$ are bijections on geometric points. Hence in the commutative diagram
\[
\begin{tikzcd}
\Gr_{B_H, S/(\Div_{X}^1)^I} \ar[r] \ar[d] &  \Fix(\sigma, \Gr_{B, S/(\Div_{X}^1)^I}) \ar[d] \\
\Gr_{H, S/(\Div_{X}^1)^I} \ar[r] & \Fix(\sigma, \Gr_{G, S/(\Div_{X}^1)^I}) 
\end{tikzcd}
\]
both vertical maps and the bottom horizontal map induce bijections on geometric points (the latter by Proposition \ref{prop: BD gr fixed points}). Therefore the upper horizontal map also induces a bijection on geometric points, concluding the proof. 


\end{proof}

\subsection{Proof of Proposition \ref{prop: algebraic gr fixed points}}
In this subsection, all loop groups are algebraic loop groups. We write $LG^{\alg}$ for the loop group of $G$ with respect to the loop variable $t$, so $LG^{\alg}(R) = G(R((t)))$. We write $L^+G^{\alg} \subset LG^{\alg}$ for the arc group, with functor of points $L^+G^{\alg}(R) = G(R[[t]])$. The algebraic affine Grassmannian is the fppf quotient $LG^{\alg}/L^+G^{\alg}$, which has a natural ind-scheme structure. 

For the purpose of proving Proposition \ref{prop: algebraic gr fixed points}, we may base change $G$ to the algebraically closed field $\CC$. Since $\sigma$ is a semisimple automorphism of $G$ over an algebraically closed field, a theorem of Steinberg \cite[Theorem 7.5]{St68} implies that $\sigma$ stabilizes a Borus $(B,T)$ of $G$. By Lemma \ref{lem: fixed borus}, $(B^\sigma, T^\sigma)$ is a Borus of $H$.

\subsubsection{Iwahori stratification} 

Let $\Iw^G$ be the Iwahori subgroup of $L^+G^{\alg}$ corresponding to $B$. This induces a stratification by $\Iw^G$-orbits (cf. \cite[(4.6)]{RW})
\[
(\Gr_{G, \CC}^{\alg})_{\red} = \coprod_{\lambda \in X_*(T)} \Gr_{G, \lambda}^{\alg}.
\]
Furthermore, $\Iw^G$ is stable under $\sigma$ and $\Iw^H := (\Iw^G)^\sigma$ is the Iwahori subgroup of $L^+H^{\alg}$ corresponding to $B^\sigma$. Thus we have an analogous stratification 
\[
(\Gr_{H, \CC}^{\alg})_{\red} = \coprod_{\lambda \in X_*(T^\sigma)} \Gr_{H, \lambda}^{\alg}. 
\]
The action of $\sigma$ on $G$ (stabilizing $T$) induces an action on $X_*(T)$. We will show that: 
\begin{enumerate}
\item If $\lambda$ is not fixed by $\sigma$, then $\Gr_{G, \lambda}^{\alg} \cap (\sigma \cdot \Gr_{G, \lambda}^{\alg}) = \emptyset$. 
\item If $\lambda$ is fixed by $\sigma$, then $\Fix(\sigma, \Gr_{G, \lambda}^{\alg})  \xleftarrow{\sim} \Gr_{H, \lambda}^{\alg}$. 
\end{enumerate}
The combination of (1) and (2) clearly suffices to prove Proposition \ref{prop: algebraic gr fixed points}

Item (1) is immediate: since we arranged $\sigma$ to preserve $\Iw^G$, we have $\sigma \cdot \Gr_{G, \lambda} = \Gr_{G, \sigma(\lambda)}$. Because distinct $\Iw^G$-orbits are disjoint, the intersection of any two distinct $\Iw^G$-orbits is empty.

Next we turn to (2). For this we analyze the structure of the strata.

\subsubsection{Fixed points of strata}

Suppose $\sigma$ fixes $\lambda$. Let
\[
\Iw^G_{\lambda}:=\Iw^G \cap t^\lambda L^+G^{\alg} t^{-\lambda}
\]
be the stabilizer of $t^\lambda$ in $\Iw^G$. Then
\[
\Gr_{G,\lambda}^{\alg}\cong \Iw^G/\Iw^G_{\lambda}.
\]
Similarly, since $\lambda \in X_*(T)^\sigma=X_*(T^\sigma)$, the stabilizer of $t^\lambda$ in $\Iw^H=(\Iw^G)^\sigma$ is $\Iw^H_{\lambda}:=(\Iw^G_{\lambda})^\sigma$, and
\[
\Gr_{H,\lambda}^{\alg}\cong \Iw^H/\Iw^H_{\lambda}.
\]
We claim that the natural map
\[
\Iw^H/\Iw^H_{\lambda}\longrightarrow \Fix(\sigma,\Iw^G/\Iw^G_{\lambda})_{\red}
\]
is an isomorphism on underlying reduced schemes. The map is a locally closed immersion of finite type reduced $\CC$-schemes, being the restriction of the closed immersion $\Gr_{H,\CC}^{\alg}\hookrightarrow\Gr_{G,\CC}^{\alg}$ to the indicated strata. Hence it suffices to check that it induces a bijection on $\CC$-points. Let $g\Iw^G_{\lambda}\in (\Iw^G/\Iw^G_{\lambda})(\CC)$ be fixed by $\sigma$. By the long exact sequence in non-abelian cohomology, it suffices to show that the map 
\[
\rH^1(\Sigma,\Iw^G_{\lambda}(\CC))\longrightarrow \rH^1(\Sigma,\Iw^G(\CC))
\]
is injective. Let
\[
r\colon \Iw^G\longrightarrow B\longrightarrow B/U\simeq T
\]
be evaluation at $t=0$ followed by the quotient map, and put
\[
N:=\ker(r),\qquad N_\lambda:=N\cap\Iw^G_\lambda.
\]
Since the constant torus $T$ is contained in both $\Iw^G$ and $t^\lambda L^+G^{\alg}t^{-\lambda}$, restriction of $r$ to $\Iw^G_\lambda$ is surjective. We therefore have $\Sigma$-equivariantly split exact sequences
\begin{equation}\label{eq:iwahori-SES}
1\to N_\lambda\to\Iw^G_\lambda\to T\to1,\qquad
1\to N\to\Iw^G\to T\to1.
\end{equation}
We claim that the nonabelian cohomology sequence for these two split extensions in \eqref{eq:iwahori-SES} now shows that both maps
\[
\rH^1(\Sigma,\Iw^G_\lambda(\CC))\longrightarrow\rH^1(\Sigma,T(\CC)),
\qquad
\rH^1(\Sigma,\Iw^G(\CC))\longrightarrow\rH^1(\Sigma,T(\CC))
\]
are bijective. Surjectivity is clear from the splitting. For injectivity, note that $N_\lambda$ and $N$ are filtered by $\Sigma$-stable vector groups, so $\rH^1(\Sigma, N_\lambda)$ and $\rH^1(\Sigma, N)$ are trivial because $\ell$ is invertible in the coefficient ring $\CC$.
 \qed

\subsection{Moduli of local shtukas}\label{ssec: fixed points sht} We now turn to examine the moduli spaces of local shtukas, for which a reference is Scholze's Berkeley lectures, especially \cite[Lecture XXIII]{SW20}. 

Let $(G,b, \{\mu_i\}_{i \in I})$ be a local shtuka datum (cf. \cite[Definition 23.1.1]{SW20}, except we allow the equal characteristic case as well): $G$ is a reductive group over $E$, $b \in B(G)$, and $\mu_i$ is a conjugacy class of cocharacters $\G_m \rightarrow G_{E^s}$. For any compact open subgroup $K \subset G(E)$ there is a \emph{moduli space of local shtukas with $K$-level structure} $\Sht_{(G,b, \{\mu_i\}_{i \in I}), K}$, which is functorial in $K$. There is a moduli description of $\Sht_{(G, b, \{\mu_i\}_{i \in I}),K}$ given (in the mixed characteristic case) in \cite[\S 23]{SW20}. 


For $\{\mu_i\}_{i \in I} \leq \{\mu_i'\}_{i \in I}$ in the component-wise Bruhat order, there are closed embeddings 
$\Sht_{(G, b, \{\mu_i\}),K} \inj \Sht_{(G, b, \{\mu_i'\}),K}$. It will also be convenient to consider the following variant without any cocharacter ``bounds'': 

\begin{defn} We define $\Sht_{(G, b,I),K} := \varinjlim_{\{\mu_i\}_{i \in I}} \Sht_{(G,b,\{ \mu_i\}_{i \in I}),K}$. 
\end{defn}


\subsubsection{Grothendieck-Messing period map}
We recall some facts about the generalized Grothendieck-Messing period map
\begin{equation}\label{eq: Grothendieck-Messing}
\pi_{\GM}^G \co \Sht_{(G,b,I),K} \rightarrow \Gr^{\twi}_{G,I}
\end{equation}
from \cite[\S 23.5]{SW20}. Note that we have not imposed any ``bound'' above; our \eqref{eq: Grothendieck-Messing} comes from the colimit over $\{\mu_i\}_{i \in I}$ of the map from \cite[Corollary 23.5.3]{SW20}. 

 The image of $\pi_{\GM}^G$ is an open subset of $\Gr^{\twi, a}_{G,I} \subset \Gr^{\twi}_{G,I}$ called the \emph{admissible} locus. In terms of points, $\Gr^{\twi}_{G,I}$ parametrizes $G$-torsors $\cP_{\eta}$ on $\cY_{(0, \infty)}$ plus additional data $\varphi_{\cP_{\eta}}$ and $\iota_r$ from \cite[Definition 23.5.1]{SW20} that we do not need to reference right now. The admissible locus is cut out by the condition that the Newton point $\nu_{\cP_{\eta}}$ and the Kottwitz point $\kappa_{\cP_\eta}$ are both identically zero \cite[Theorem 22.6.2]{SW20}. 

\begin{defn}
For a basic $b \in B(G)$, denote by $\Gr^{\twi, b}_{G,I} $ the subspace where $\nu_{\cP_\eta} = \nu(b)$ and $\kappa_{\cP_{\eta}} = \kappa(b)$ after pulling back to each geometric point. This is an open subspace, and when $b=1_G$ this recovers the admissible locus, i.e., $\Gr^{\twi, 1_G}_{G,I} = \Gr^{\twi, a}_{G,I} \subset \Gr^{\twi}_{G,I}$. 
\end{defn}

\begin{example}
If $b \in B(G)$ is basic, then there is an extended pure inner twist $G_b$ (sometimes denoted $J_b$ in the literature) and a canonical isomorphism 
\begin{equation}\label{eq: basic twist}
\tr_b \co  B(G) \xrightarrow{\sim} B(G_b)
\end{equation}
sending $b \in B(G)$ to $1_{G_b} \in B(G_b)$. There is also a compatible isomorphism $\Gr_{G,I}^{\twi} \cong \Gr_{G_b,I}^{\twi}$ as in \cite[\S III.4.1]{FS}, which takes $\Gr_{G,I}^{\twi, b}$ to $\Gr_{G_b,I}^{\twi,a}$. So we could have formulated the subspaces $\Gr^{\twi, b}_{G,I}$ as admissible loci for extended pure inner twists $G_b$. However, we prefer to keep the distinction, for psychological reasons if nothing else. 
\end{example}

Let $\iota \co B(H) \rightarrow B(G)$ be the map induced by the inclusion $\iota \co H \inj G$. The induced map $\iota \co  \Gr_{H,I}^{\twi} \inj \Gr_{G,I}^{\twi}$ is also a closed embedding (by the proof of \cite[Lemma 19.1.5]{SW20}).


\begin{lemma}\label{lem: fixed points of admissible} We have
\[
\Fix(\sigma, \Gr^{\twi,a}_{G,I})  = \coprod_{b' \in \iota^{-1}(1_G)} \Gr^{\twi, b'}_{H,I}.
\]
as subdiamonds of $\Gr^{\twi}_{G,I}$. 
\end{lemma}

\begin{proof}
In Corollary \ref{cor: fixed points conv Gr}(3) we have seen that $\Fix(\sigma, \Gr_{G,I}^{\twi}) = \Gr_{H,I}^{\twi}$ (with the obvious embedding). Therefore, we may interpret (say, by Remark \ref{rem:Cartesian-fixed-points}) 
\[
\Fix(\sigma, \Gr^{\twi,a}_{G,I}) = \Gr^{\twi,a}_{G,I} \cap \Gr^{\twi}_{H,I}.
\]
The admissible locus $\Gr^{\twi,a}_{G,I}  \subset \Gr^{\twi}_{G,I}$ is cut out by the condition that $\nu(\cP_{\eta}) = 0$ and $\kappa(\cP_\eta) = 0$. Hence $\Gr^{\twi,a}_{G,I} \cap \Gr^{\twi}_{H,I}$ is the pullback of the simultaneous vanishing loci of the Newton and Kottwitz maps on $\Gr_{G,I}^{\twi}$ along the closed embedding $\Gr_{H,I}^{\twi}\hookrightarrow \Gr_{G,I}^{\twi}$, which is precisely the expression on the RHS.
\end{proof}

\subsubsection{Fixed points of moduli of local shtukas} We now recall more about the structure of $\Sht_{(G, b, I),K}$ in terms of the Grothendieck-Messing period map. According to \cite[Corollary 23.5.3]{SW20}, $\Gr_{G, I}^{\twi,a}$ carries a $\ul{G(E)}$-torsor $\PP_\eta^G$, and $\pi_{\GM}^G \co \Sht_{(G, b, I),K} \rightarrow \Gr_{G, I}^{\twi,a}$ is the \'{e}tale cover parametrizing $K$-lattices $\PP^G \subset \PP_{\eta}^G$.

%
%


We write $\iota^*K=K\cap H(E)=K^\sigma$.

\begin{lemma}\label{lem: coprime level fixed descent}
Let $K$ be a profinite group of prime-to-$\ell$ pro-order with a continuous $\Sigma$-action. If a v-sheaf $X$ has a compatible $\Sigma$-action and a free right $\ul K$-action, and $X\to X/\ul K$ is a v-torsor, then
\begin{equation}\label{eq: coprime level fixed descent}
X^\Sigma/\ul{K}^\Sigma\xrightarrow{\sim}(X/\ul K)^\Sigma.
\end{equation}
\end{lemma}

\begin{proof}
We want to show that the map $X^\Sigma/ K^\Sigma \rightarrow (X/\ul K)^\Sigma$ is an isomorphism of v-sheaves. It is clearly an injection of v-sheaves, so the point is to show surjectivity. By definition, any section $y$ of $(X/\ul K)^\Sigma$ lifts v-locally to $x\in X$; we want to adjust the lift to lie in $X^\Sigma$ after possible further v-localization. By the assumed freeness of the action of $\ul K$, there is a unique cocycle $(\sigma \mapsto a_\sigma)$ with values in $\ul K$ such that $\sigma(x)=xa_\sigma$. Since $K$ has prime-to-$\ell$ pro-order, we have $H^1(\Sigma, K)=0$, hence the coboundary map $K\to Z^1(\Sigma,K)$ is a continuous surjection of profinite sets, so after a further v-cover we may write $a_\sigma=k\sigma(k)^{-1}$. Then $xk$ is a section of $X^\Sigma$ lying over $y$.

\end{proof}

\begin{prop}\label{prop: fixed points of sht}
Let $I$ be nonempty and finite, and let $K\subset G(E)$ be a $\Sigma$-stable compact open subgroup of prime-to-$\ell$ pro-order. The natural embedding identifies
\begin{equation}\label{eq: fixed framed shtukas}
\Sht_{(H,1_H,I),\iota^*K}\xrightarrow{\sim}\Fix(\sigma,\Sht_{(G,1_G,I),K}).
\end{equation}
This identification is compatible with base change in $\prod_{i\in I}\Spd\brE$ and with level covers. 
\end{prop}


\begin{proof}
It suffices to show the statement at infinite level: 
\begin{equation}\label{eq:infinite-level-fixed-points}
\Sht_{(H,1_H,I),\infty}\xrightarrow{\sim}\Fix(\sigma,\Sht_{(G,1_G,I),\infty}).
\end{equation}
Indeed, once we know this, \eqref{eq: fixed framed shtukas} follows from it by Lemma \ref{lem: coprime level fixed descent}. 

We will analyze the $\Sigma$-fixed points by examining the infinite-level Grothendieck--Messing period map $\pi_{\GM}^G$, first identifying the $\Sigma$-fixed points of the period space, and then the $\Sigma$-fixed points of the fibers. By Lemma \ref{lem: fixed points of admissible}, the former is
\[
\Fix(\sigma, \Gr^{\twi,a}_{G,I})  = \coprod_{b' \in \iota^{-1}(1_G)} \Gr^{\twi, b'}_{H,I}.
\]
For an $S$-valued point $x \in \Gr^{\twi, b'}_{H,I}$, let $\cE_{b'}$ denote its associated $H$-bundle on $X_S$, of geometric class $b'$. The fiber of $\pi_{\GM}^G$ over $\iota(x)$ is the $\ul{G(E)}$-torsor $\Isom_G(\cE_{1_G}, \cE_{b'} \times^H G)$, given by 
\[
(T\to S) \mapsto \Isom_{\text{$G$-bundles on $X_T$}}(\cE_{1_G}|_{X_T}, (\cE_{b'} \times^H G)|_{X_T}).
\]
Its $\Sigma$-fixed points form the $H(E)$-pseudo-torsor $\Isom_H(\cE_{1_H}, \cE_{b'})$, as can be checked by a local calculation using $H = G^\sigma$. This shows that:
\begin{itemize}
\item For $b' \neq 1_H$, $\Fix(\sigma,\Sht_{(G,1_G,I),\infty})$ has empty fiber over $\iota(x)$.
\item For $b' = 1_H$, the fibral maps from $(\pi_{\GM}^H)^{-1}(x) = \Isom_H(\cE_{1_H}, \cE_{b'})$ to $\Fix(\sigma,(\pi_{\GM}^G)^{-1}(\iota(x)))$ are isomorphisms. 
\end{itemize} 
This completes the proof.
\end{proof}

\section{The Brauer functor}\label{sec: categorical TV}

Let $G$ be a reductive group over $E$ with an action of $\Sigma$, and let $H = G^{\sigma} \xrightarrow{\iota} G$. We assume that $H$ is reductive. Treumann--Venkatesh defined in \cite{TV} the \emph{(normalized) Brauer homomorphism} $\br$ from the spherical Hecke algebra of $G$ to that of $H$. In this section we will categorify it to a \emph{Brauer functor} from the Satake category of $G$ to that of $H$. 

\subsection{Geometric Satake}\label{ssec: geom satake} We first record some generalities on the Geometric Satake equivalence. 

\subsubsection{The Satake category} 
Let $S$ be any small v-stack over $\Div_X^1$. We denote the \emph{Satake category} 
\[
\Sat(\Gr_{G, S/\Div_X^1}; \Lambda) \subset \DulacHck{G}{\Lambda}{1}
\]
to be the full subcategory spanned by flat perverse sheaves over $S$, in the sense of Fargues--Scholze \cite[\S VI.7]{FS}. We recall some details of this definition for the convenience of the reader. 

\begin{defn} The perverse t-structure on $D_{\et}(\cHck_{G,S/\Div_X^1};\Lambda)^{\bd}$ is determined by the condition that $A\in D_{\et}(\cHck_{G,S/\Div_X^1};\Lambda)^{\bd}$ lies in ${}^pD_{\et}^{\leq0}(\cHck_{G,S/\Div_X^1};\Lambda)^{\bd}$ if and only if, after pullback to every geometric point of $S$, its restriction to the open Schubert cell indexed by $\mu$ lies in degrees $\leq-\tw{2\rho,\mu}$.
\end{defn}

By \cite[Definition--Proposition VI.7.1]{FS}, these conditions define a unique t-structure, which we call the \emph{perverse t-structure}. Finally, a perverse sheaf $A$ is ``flat'' if $A \otimes^L_\Lambda M$ is perverse for every $\Lambda$-module $M$; see \cite[Proposition VI.7.7 and Definition VI.7.8]{FS}.

By the same argument as for \cite[Proposition VI.7.2]{FS}, the pullback functor 
\[
\Sat(\Gr_{G, S/\Div_X^1}; \Lambda) \rightarrow \Dula{G}{\Lambda}{1}
\]
is fully faithful, so we may regard $\Sat(\Gr_{G, S/\Div_X^1}; \Lambda)  \subset \Dula{G}{\Lambda}{1}$. 

A similar story holds for multiple legs. For a finite non-empty set $I$, we abbreviate $\Sat_G^I(k) := \Sat(\Gr_{G, (\Div_X^1)^I}; k)$. There is an analogous relative perverse t-structure on $D_{\et}(\cHck_{G,S/(\Div_X^1)^I};k)^{\bd}$ for any $S\rightarrow(\Div_X^1)^I$. For this t-structure, pullback along any $S'\rightarrow S$ is t-exact by \cite[Proposition VI.7.4]{FS}. 

\subsubsection{The Geometric Satake equivalence}
For a v-stack $S$, we let $\Loc(S; \Lambda)$ be the category of $\Lambda$-free \'{e}tale local systems on $S$. Let $\pi_{G,S} \co \Gr_{G, S/\Div_X^1} \rightarrow S$ be the natural projection. Fargues--Scholze establish in \cite[\S VI.9--VI.11]{FS} the Geometric Satake equivalence, 
\begin{equation}\label{eq: geom sat Div}
\Sat(\Gr_{G, \Div_X^1}; \Lambda) \cong \Rep_{\Loc(\Div_X^1; \Lambda)} (\chG),
\end{equation}
where the right side is the category of representations of a certain reductive group object in the category $\Loc(\Div_X^1; \Lambda)$. The equivalence is symmetric monoidal, with the underlying monoidal structure given by convolution on the left and tensor product on the right, and carries the fiber functor $\bigoplus_i \rR^i \pi_{G,\Div_X^1 *}$ on the left to the forgetful functor on the right. 

The category $\Loc(\Div_X^1; \Lambda)$ is symmetric monoidal under tensor product. Note that there is an action of $\Loc(\Div_X^1; \Lambda)$ on $\Sat(\Gr_{G, \Div_X^1}; \Lambda)$ via pullback and tensoring. Pullback and tensoring likewise define an action of $\Loc(S;\Lambda)$ on $\Sat(\Gr_{G,S/\Div_X^1};\Lambda)$.

Let $2\rho_G$ be the sum of the positive roots with respect to a Borus of $G_{E^s}$. Since $\tw{2\rho_G,-}$ is even on coroots, its reduction modulo $2$ factors through a Galois-invariant homomorphism
\begin{align}\label{eq: dimension component parity}
\dagger_G & \co\pi_1(G)\rightarrow\Z/2\Z, \nonumber \\
\dagger_G(\lambda)& =\tw{2\rho_G,\lambda}\pmod2.
\end{align}
As in \cite[\S VI.9, before Definition/Proposition VI.9.4]{FS}, the commutativity constraint of the Satake category is modified by a sign whenever two objects supported on odd components (for $\dagger_G$) are interchanged. With this convention, the fiber functors $\msf F^{G,I}:=\bigoplus_n\rR^n\pi_{G*}$ are symmetric monoidal (see the discussion after \cite[Definition/Proposition VI.9.4]{FS}), and so are the normalized constant term functors $\CT_B[\deg_G]$ of \S\ref{sssec: constant terms}, by \cite[Proposition VI.9.6]{FS}. We omit $I$ from the notation when $|I|=1$.

\subsubsection{Fusion}
For a partition $I=I_1\sqcup I_2$, let $j$ be the inclusion of the open subset of $(\Div_X^1)^I$ where legs in different blocks are disjoint. Restriction along $j$ is fully faithful on Satake categories, and the exterior product over this open subset extends to the fusion product \cite[Proposition VI.9.3 and Definition/Proposition VI.9.4]{FS}. Fusion commutes with pullback along maps of finite sets, and its restriction to the diagonal is convolution. Normalized constant terms commute with fusion \cite[Proposition VI.9.6]{FS}.

\begin{remark}\label{rem:fusion}
There is the following geometric description of fusion, from \cite[VI.9]{FS}. Any map $\zeta \co I \rightarrow J$ of non-empty finite sets induces a closed embedding $\Gr_{G,(\Div_X^1)^I}  \times_{(\Div_X^1)^I} (\Div_X^1)^J \inj \Gr_{G, (\Div_X^1)^J}$, and the fusion product $\zeta \co \Sat_G^I(k) \rightarrow \Sat_G^J(k)$ can be interpreted as pull-push along the diagram 
\begin{equation}\label{eq:fusion-pull-push}
\begin{tikzcd}\Gr_{G,(\Div_X^1)^I} &  \Gr_{G,(\Div_X^1)^I}  \times_{(\Div_X^1)^I} (\Div_X^1)^J\ar[l, "p_\zeta"']  \ar[r, hookrightarrow, "q_\zeta"] & \Gr_{G, (\Div_X^1)^J} 
\end{tikzcd}
\end{equation}
Writing the two arrows in \eqref{eq:fusion-pull-push} as $p_\zeta$ and $q_\zeta$, this functor is $q_{\zeta,*}p_\zeta^*$, where $q_{\zeta,*}=q_{\zeta,!}$ since $q_\zeta$ is a closed embedding. When $\zeta$ is not surjective, the map $q_\zeta$ simply inserts the unit object at the legs in $J\setminus \operatorname{im}(\zeta)$. 
\end{remark}

\subsubsection{Constant terms}\label{sssec: constant terms}
Suppose that $G$ is split, and choose a Borus $(B,T)$. Consider the natural maps
\[
\begin{tikzcd}[ampersand replacement=\&]
\Gr_{G,S/(\Div_X^1)^I} \& \Gr_{B,S/(\Div_X^1)^I} \ar[l,"p_G"'] \ar[r,"q_G"] \& \Gr_{T,S/(\Div_X^1)^I}
\end{tikzcd}
\]
and define the functor
\begin{equation}\label{eq: normalized constant term notation}
\CT_B=\rR q_{G!}p_G^*\co D_{\et}(\cHck_{G,S/(\Div_X^1)^I};k)^{\bd}\rightarrow D_{\et}(\Gr_{T,S/(\Div_X^1)^I};k).
\end{equation}
Let $\deg_G\co\Gr_{T,S/(\Div_X^1)^I}\rightarrow\Z$ be the locally constant function of \cite[Proposition VI.7.4]{FS}, obtained by pairing $2\rho_G$, the sum of the roots of $T$ in $B$, with the sum of the relative positions of all modifications. According to \cite[Propositions VI.4.2 and VI.7.4]{FS}, the functor $\CT_B[\deg_G]$ is conservative and perverse t-exact, hence exact and faithful on perverse sheaves. By \cite[Proposition VI.6.4]{FS} a complex $M$ with bounded support is ULA if and only if $\rR\pi_{T*}\CT_B(M)$ is locally constant with perfect fibers.

\begin{example}\label{ex:torus-geom-satake}
  We consider the case of a split torus $T$. For $\boldsymbol\lambda=(\lambda_i)_i\in X_*(T)^I$ we have the section
\begin{align*}
s_{\boldsymbol\lambda} \co (\Div_X^1)^I & \rightarrow \Gr_{T,(\Div_X^1)^I} \\
(D_i)_i & \mapsto  \sum_i\lambda_iD_i.
\end{align*}
Every $M\in\Sat(\Gr_{T,(\Div_X^1)^I};k)$ has a canonical (finite) decomposition
\begin{equation}\label{eq: torus weight sections}
M \cong \bigoplus_{\boldsymbol\lambda\in X_*(T)^I}s_{\boldsymbol\lambda*}V_{\boldsymbol\lambda}, \qquad V_{\boldsymbol\lambda}\in\Loc((\Div_X^1)^I;k).
\end{equation}
Indeed, the functor $(V_{\boldsymbol\lambda})_{\boldsymbol\lambda}\mapsto\bigoplus_{\boldsymbol\lambda}s_{\boldsymbol\lambda*}V_{\boldsymbol\lambda}$ from $X_*(T)^I$-graded local systems is fully faithful: the sections $s_{\boldsymbol\lambda}$ are disjoint over the open subset of $(\Div_X^1)^I$ where the legs are pairwise distinct, and restriction to this open subset is fully faithful on Satake categories and on local systems \cite[Proposition VI.9.3]{FS}. It is essentially surjective because under the Satake equivalence of \cite[Proposition VI.10.3 and Theorem VI.11.1]{FS}, $s_{\boldsymbol\lambda*}V$ corresponds to $V$ placed in weight $\boldsymbol\lambda$.
\end{example}

\subsection{Brauer homomorphisms}\label{ssec: brauer homomorphisms}
Let $K \subset G(E)$ be an open compact subgroup. The \emph{Hecke algebra of $G$ with respect to $K$ with coefficients in $\Lambda$} is
\begin{equation}\label{eq: hecke algebra}
\sH(G,K; \Lambda) := \Fun_c(K \bs G(E) / K, \Lambda),
\end{equation}
the compactly supported functions on $K \bs G(E)/ K$ valued in $\Lambda$. This forms an algebra under convolution, normalized so that the indicator function $\bbm{1}_K$ is the unit. 

For a $\Sigma$-stable subgroup $K \subset G(E)$, we write $\iota^* K := \iota^{-1}(K) \subset H(E)$. We say that a compact open subgroup $K \subset G(E)$ is a \emph{plain subgroup} if the natural map $H(E)/\iota^* K \rightarrow  [G(E)/K]^\sigma$ is a bijection. We may view $\sH(G, K; \Lambda)$ as the ring $\mrm{Fun}^{G(E)}_c([G(E) / K] \times [G(E) / K], \Lambda)$ of $G(E)$-invariant (for the diagonal action) functions (valued in $\Lambda$) on $[G(E)/ K] \times [G(E) / K]$, with compact support modulo $G(E)$, under convolution. 

\subsubsection{The un-normalized Brauer homomorphism}\label{sssec: un-normalized Brauer} Now suppose that $\Lambda = k$ has characteristic $\ell$. In this special situation, Treumann-Venkatesh observed that if $K \subset G(E)$ is a plain subgroup, then the \emph{restriction map}
\begin{align}\label{eq: Br}
\sH(G, K;k)^{\sigma}  &= \mrm{Fun}^{G(E)}_c([G(E) / K] \times [G(E) / K], k)^{\sigma} \\
&  \xrightarrow{\text{restrict}} \mrm{Fun}^{H(E)}_{c}([H(E)/\iota^* K] \times [H(E)/\iota^* K], k)  = \sH(H(E), \iota^* K; k) \nonumber
\end{align}
is an algebra homomorphism (cf. \cite[Lemma 6.5]{Fe23} for a proof). This map was introduced in \cite[\S 4]{TV} and called the \emph{(un-normalized) Brauer homomorphism}. We denote it 
\begin{equation}\label{eq: un-normalized Brauer}
\Br \co \sH(G, K;k)^{\sigma} \rightarrow \sH(H, \iota^* K;k).
\end{equation}

\subsubsection{Frobenius twist of algebras}\label{sssec: frob twist algebra} Let $\Frob$ be the absolute ($\ell$-power) Frobenius of $k$. Given a commutative $k$-algebra $A$, we denote by $A\ellt := A \otimes_{k, \Frob} k$ its Frobenius twist. The map $\phi \co A \rightarrow A\ellt$ sending $a \mapsto a \otimes 1$ is a $\Frob$-semilinear (i.e., $\phi(\lambda a) = \Frob(\lambda) \phi(a)$) isomorphism.

If $A$ is equipped with an $\F_\ell$-structure $\varphi_0 \co A \cong A_0 \otimes_{\F_\ell} k$, then there is a $k$-linear isomorphism $f_{\varphi_0} \co A \xrightarrow{\sim} A\ellt$, characterized by the property that it sends $A_0 \subset A$ to $A_0 \otimes 1 \subset A\ellt$ via $\Id \otimes 1$. We denote by $\Frob_{\varphi_0} = f_{\varphi_0}^{-1} \circ \phi  \co A \xrightarrow{\sim} A$; it is characterized as the unique $\Frob$-semilinear automorphism of $A$ which restricts to the identity on $A_0$. 

Suppose $f \co A \rightarrow B$ is any $\Frob$-semilinear homomorphism (i.e., $f(\lambda a) = \lambda^\ell f(a)$ for $\lambda \in k$). Then $f$ factors uniquely through a $k$-linear homomorphism 
\[
\begin{tikzcd}
A \ar[dr, "f"] \ar[d, "\Frob_{\varphi_0}"']  \\
A \ar[r, dashed] & B
\end{tikzcd}
\]
The $k$-linear homomorphism $A \rightarrow B$ obtained by precomposing $f$ with the inverse of $\Frob_{\varphi_0}$ (i.e., the dashed arrow in the diagram above) will be called the \emph{linearization} of $f$ (with respect to $\varphi_0$).

\begin{defn}[The Tate diagonal]\label{def: tate diagonal}
Let $A$ be a commutative $k$-algebra equipped with an action of $\Sigma$ by $k$-algebra automorphisms, and with an $\F_\ell$-structure $\varphi_0$. Let $\Nm \co A \rightarrow \rT^0(A)$ be the \emph{Tate diagonal} map sending $a$ to the class of $a \cdot ({}^{\sigma} a) \cdot \ldots \cdot ({}^{\sigma^{\ell-1}} a)$. One checks that this is a ring homomorphism, which is evidently $\Frob$-semilinear. Let $\Nm \iellt \co A \rightarrow \rT^0(A)$ be the linearization of $\Nm$ with respect to $\varphi_0$.
\end{defn}

\subsubsection{The normalized Brauer homomorphism}\label{sssec: normalized Brauer} Next suppose that $\sH(G(E),K;k)$ and $\sH(H(E), \iota^* K;k)$ are commutative. The Tate diagonal provides a map (cf. \S \ref{ssec: tate cohomology} for notation on Tate cohomology) 
\[
\Nm \co \sH(G,K;k) \rightarrow \rT^0(\sH(G,K;k)) = \frac{ \sH(G,K;k)^\sigma}{N \cdot \sH(G,K;k)} .
\]
One observes that the Brauer homomorphism $\Br$ factors over $\rT^0(\sH(G,K;k))$, hence we obtain a map 
\[
\Br \circ \Nm  \co \sH(G,K;k) \rightarrow \sH(H, \iota^* K;k).
\]
This map is a $\Frob$-semilinear ring homomorphism, and the \emph{normalized Brauer homomorphism} is its linearization with respect to the canonical $\F_\ell$-structure
\[
\sH(G,K;k) \cong \sH(G,K;\F_{\ell}) \otimes_{\F_\ell} k.
\]
Equivalently, we may write
\begin{equation}\label{eq: normalized Brauer}
\br =  \Br  \circ \Nm \iellt \co \sH(G,K;k) \rightarrow \sH(H, \iota^* K;k).
\end{equation}

We will categorify $\br$ to a Tannakian functor; see \cite[\S 2]{DM82} for a reference on Tannakian categories.

\subsection{Construction of the Brauer functor}\label{ssec: modular Brauer extraction}
We now construct the functors $\wcbr^I$. We remark that in this construction, it is permitted for $\Sigma$ to act trivially on $G$. 

For a sheaf of $k$-modules $M$, let
\begin{equation}\label{eq: coefficient Frobenius sheaves}
M\iellt=M\otimes_{k,\Frob^{-1}}k
\end{equation}
be its inverse Frobenius twist. This is an exact tensor autoequivalence, $\Frob^{-1}$-semilinear on morphisms, which commutes with the sheaf operations and preserves Satake objects and their supports. For $X\in\Sat_G^I(k)$, define
\begin{equation}\label{eq: modular convolution norm}
\Nm(X):=X\star\sigma^*X\star\cdots\star(\sigma^{\ell-1})^*X
\end{equation}
and
\begin{equation}
\Nm\iellt(X) :=\Nm(X\iellt).
\end{equation}
Using the commutativity constraint, cyclic rotation of the factors gives isomorphisms $\sigma^*\Nm(X)\cong\Nm(X)$. Since $\Nm(X)$ lies in the abelian category $\Sat_G^I(k)$, this promotes $\Nm(X)$ into a $\Sigma$-equivariant object. The functor $\Nm \iellt$ is compatible with composition of morphisms, but it is not additive; additivity will be restored after passing to Tate constructions (Proposition \ref{prop: modular Brauer extraction}).

\begin{defn}[The Brauer functor]\label{def:brauer-functor} Let
\begin{equation}\label{eq:delta-H-G}
\delta:=\iota^*\dagger_G-\dagger_H\co\pi_1(H)\rightarrow\Z/2\Z,
\end{equation}
where $\iota^*\dagger_G$ is the restriction of $\dagger_G$ to $\pi_1(H)$, and write $[\delta]$ for the shift by the representative of $\delta$ in $\{0,1\}$ on each connected component. Let $\iota_I\co\cHck_{H,(\Div_X^1)^I}\rightarrow\cHck_{G,(\Div_X^1)^I}$ be the natural map. It is $\Sigma$-equivariant for the trivial action on its source, so $\iota_I^*\Nm\iellt(X)$ is a $k[\Sigma]$-module object, and we set
\begin{equation}\label{eq: normalized Smith coefficient}
A_I(X)=[\delta]\bigl(\iota_I^*\Nm\iellt X\bigr)^{t\Sigma}\in D_{\et}(\cHck_{H,(\Div_X^1)^I};k)^{\bd}.
\end{equation}
Define the \emph{Brauer functor}
\begin{equation}\label{eq: normalized brauer functor}
\wcbr^I(X)={}^p\rH^0(A_I(X)) \in \Sat_H^I(k),
\end{equation}
and abbreviate $\wcbr=\wcbr^{\{1\}}$. We show in Proposition \ref{prop: modular Brauer extraction} that $\wcbr^I(X)$ is a Satake object. 
\end{defn}

\begin{remark}
We indicate why this can be thought of as a categorification of the Brauer homomorphism. By Proposition \ref{prop: BD gr fixed points} and Example \ref{ex: Tate commutation instances}(1), the pullback of $A_I(X)$ to $\Gr_{H,(\Div_X^1)^I}$ is $[\delta]\Real(\Psm(\Nm\iellt X))$, so we can think of $A_I(X)$ as the ``periodic realization'' (up to shift) of $\Psm(\Nm\iellt X)$, the Smith--Treumann localization of $\Nm \iellt X$. The latter object corresponds to the normalized Brauer homomorphism, under the function-sheaf dictionary. 
\end{remark}

\subsection{Constant terms and perversity}\label{ssec: compatibility with constant term}
If $G$ has a $\Sigma$-stable Borus $(B,T)$, then Lemma \ref{lem: fixed borus} implies that $(B_H,T_H)=(B^\sigma,T^\sigma)$ is a Borus of $H$. For $\lambda\in X_*(T)$, we define 
\[
N(\lambda)=\sum_{r=0}^{\ell-1}\sigma^r(\lambda) \in X_*(T)^\sigma=X_*(T_H).
\] 
When $N$ is applied to tuples of coweights, it means apply $N$ entrywise.

\begin{lemma}\label{lem: CT PSm}
Suppose $G$ has a $\Sigma$-stable Borus $(B,T)$. Let $M$ be a bounded $\Sigma$-equivariant complex on $\cHck_{G,S/(\Div_X^1)^I}$ with bounded support, and let $\iota_{T,I}\co\Gr_{T_H,S/(\Div_X^1)^I}\rightarrow\Gr_{T,S/(\Div_X^1)^I}$ be the natural map. There is a natural isomorphism
\begin{equation}\label{eq: CT PSm}
\CT_{B_H}\bigl((\iota_I^*M)^{t\Sigma}\bigr)\cong\bigl(\iota_{T,I}^*\CT_B(M)\bigr)^{t\Sigma},
\end{equation}
compatible with pullback in $S$ and with the cup products \eqref{eq: exterior Tate cup}.
\end{lemma}

\begin{proof}
Define $Y$ to be the fibered product as in the diagram 
\[
\begin{tikzcd}[ampersand replacement=\&]
\& Y \ar[d] \ar[r,"q"] \ar[dl,"p"'] \ar[dr,phantom,"\lrcorner",very near start] \& \Gr_{T_H,S/(\Div_X^1)^I} \ar[d,"\iota_{T,I}"] \\
\Gr_{G,S/(\Div_X^1)^I} \& \Gr_{B,S/(\Div_X^1)^I} \ar[l,"p_G"] \ar[r,"q_G"'] \& \Gr_{T,S/(\Div_X^1)^I}
\end{tikzcd}
\]
Base change gives
\begin{equation}\label{eq: constant term base change}
\iota_{T,I}^*\CT_B(M)\cong\rR q_!p^*M.
\end{equation}
The group $\Sigma$ acts on $Y$ over $\Gr_{T_H,S/(\Div_X^1)^I}$, and $Y^\sigma=\Gr_{B_H,S/(\Div_X^1)^I}$ by Propositions \ref{prop: BD gr fixed points} and \ref{prop: GrB fixed points}. 
The inclusion $i\co\Gr_{B_H,S/(\Div_X^1)^I}=Y^\sigma\inj Y$ of the fixed locus fits into the commutative diagram
\[
\begin{tikzcd}[ampersand replacement=\&]
\Gr_{H,S/(\Div_X^1)^I} \ar[d,"\iota_I"'] \& \Gr_{B_H,S/(\Div_X^1)^I} \ar[l,"p_H"'] \ar[r,"q_H"] \ar[d,hook,"i"] \& \Gr_{T_H,S/(\Div_X^1)^I} \ar[d,equals] \\
\Gr_{G,S/(\Div_X^1)^I} \& Y \ar[l,"p"] \ar[r,"q"'] \& \Gr_{T_H,S/(\Div_X^1)^I}
\end{tikzcd}
\]
so that $i^*p^*M\cong p_H^*\iota_I^*M$ and $q\circ i=q_H$. Since $M$ has bounded support and $p_G$ is bijective on geometric points and a locally closed immersion on each connected component of $\Gr_{B,S/(\Div_X^1)^I}$ \cite[Proposition VI.3.1]{FS}, the complex $p^*M$ is supported on a $\Sigma$-stable closed subspace of $Y$ that is spatial and of finite dimension over $S$, and meets only finitely many connected components. On this subspace Proposition \ref{prop: realized Smith localization} applies to $q$; together with Example \ref{ex: Tate commutation instances}(1), it gives
\[
(\rR q_!p^*M)^{t\Sigma}\cong\rR q_{H!}\bigl((p_H^*\iota_I^*M)^{t\Sigma}\bigr)\cong\rR q_{H!}p_H^*\bigl((\iota_I^*M)^{t\Sigma}\bigr)=\CT_{B_H}\bigl((\iota_I^*M)^{t\Sigma}\bigr).
\]
The compatibilities with pullback and with cup products follow from those of base change, of \eqref{eq: realized Smith localization}, and of Proposition \ref{prop: Tate cup compatibility}.
\end{proof}

\begin{prop}\label{prop: modular Brauer extraction}
The construction \eqref{eq: normalized brauer functor} defines an exact, faithful, $k$-linear functor $\wcbr^I\co\Sat_G^I(k)\rightarrow\Sat_H^I(k)$, whose formation commutes with pullback along any $S\rightarrow(\Div_X^1)^I$. 

Moreover, assume $G$ has a $\Sigma$-stable Borus $(B,T)$. If $X \in \Sat_G^I(k)$ and 
\begin{equation}\label{eq: input Brauer weights}
\CT_B[\deg_G](X)=\bigoplus_{\boldsymbol\lambda}s_{\boldsymbol\lambda*}(V_{\boldsymbol\lambda})
\end{equation}
in the notation of \eqref{eq: torus weight sections}, then there is a natural isomorphism
\begin{equation}\label{eq: Brauer constant term weights}
\CT_{B_H}[\deg_H](\wcbr^I(X))\cong\bigoplus_{\boldsymbol\lambda}s_{N\boldsymbol\lambda*}(V_{\boldsymbol\lambda}).
\end{equation}
\end{prop}

\begin{proof} We first observe that the formation of \eqref{eq: normalized brauer functor} commutes with pullback along any $S\rightarrow(\Div_X^1)^I$: the norm, $\iota_I^*$ and $[\delta]$ commute with such pullbacks, as do the Tate construction (Example \ref{ex: Tate commutation instances}(1)) and ${}^p\rH^0$, since pullback is perverse t-exact \cite[Proposition VI.7.4]{FS}.

The remaining assertions of the first paragraph reduce to the case where $G$ has a $\Sigma$-stable Borus $(B,T)$ with $T$ split. Indeed, since $\sigma$ has finite order, such a Borus exists after replacing $E$ by a finite extension $E'$ \cite[Theorem 7.5]{St68}, which amounts to pulling back along the finite \'etale cover $\pi\co(\Div_{X'}^1)^I\rightarrow(\Div_X^1)^I$. By the previous paragraph, $\pi^*\circ\wcbr^I$ is the Brauer functor for $G_{E'}$ composed with $\pi^*$. The functor $\pi^*$ is exact and faithful on complexes with bounded support. This pullback preserves and reflects perversity, since $\pi^*$ is perverse t-exact \cite[Proposition VI.7.4]{FS} and conservative, and it preserves and reflects the ULA property, since local acyclicity descends along v-covers of the base \cite[Proposition IV.2.5]{FS}. Hence it suffices to show that $\wcbr^I(X)$ is a Satake object and that $\wcbr^I$ is additive, $k$-linear, exact and faithful after applying $\pi^*$, that is, for $G_{E'}$.

Now suppose that $G$ has a $\Sigma$-stable Borus $(B,T)$. We deduce the asserted properties of \eqref{eq: normalized brauer functor} from \eqref{eq: Brauer constant term weights}. Indeed, the explicit formula \eqref{eq: Brauer constant term weights} implies that $\CT_{B_H}[\deg_H]\circ\wcbr^I$ is additive, $k$-linear and exact. As $\CT_{B_H}[\deg_H]$ is exact, faithful and conservative on perverse sheaves, we may conclude that $\wcbr^I$ is additive, $k$-linear, exact and faithful. After undoing the shifts by $\deg_H$, the pushforward of \eqref{eq: Brauer constant term weights} to $(\Div_X^1)^I$ is a finite direct sum of shifted local systems of finite rank, so $\wcbr^I(X)$ is ULA by \cite[Proposition VI.6.4]{FS}, and hence lies in $\Sat_H^I(k)$.

By definition of $\wcbr^I$ and the perverse t-exactness of $\CT_{B_H}[\deg_H]$ \cite[Proposition VI.7.4]{FS}, we have
\begin{equation}\label{eq: CT brauer H0}
\CT_{B_H}[\deg_H](\wcbr^I(X)) = \rH^0\bigl(\CT_{B_H}[\deg_H]A_I(X)\bigr).
\end{equation}
Consider the locally constant function
\[
d:=\delta+\deg_H-\deg_G\circ\iota_{T,I}\co\Gr_{T_H,(\Div_X^1)^I}\rightarrow\Z,
\]
where $\delta$ is the function from Definition \ref{def:brauer-functor} restricted to $T_H$. Lemma \ref{lem: CT PSm} identifies $\CT_{B_H}[\deg_H]A_I(X)$ with
\begin{equation}\label{eq:CT-brauer}
\bigl(\iota_{T, I}^* \CT_{B}[\deg_G] ( \Nm\iellt(X))\bigr)^{t \Sigma}[d],
\end{equation}
which we now compute. 

Since $B$ is $\Sigma$-stable, $\CT_B[\deg_G]$ commutes with $\sigma^*$, and since it is symmetric monoidal it carries $\Nm\iellt(X)$, with its $\Sigma$-equivariant structure, to
\[
\bigotimes_{r=0}^{\ell-1}(\sigma^r)^*\Bigl(\bigoplus_{\boldsymbol\lambda}s_{\boldsymbol\lambda*}V_{\boldsymbol\lambda}\iellt\Bigr)=\bigoplus_{(\boldsymbol\lambda_0,\ldots,\boldsymbol\lambda_{\ell-1})}s_{\boldsymbol\mu*}\bigl(V_{\boldsymbol\lambda_0}\iellt\otimes\cdots\otimes V_{\boldsymbol\lambda_{\ell-1}}\iellt\bigr),\qquad\boldsymbol\mu=\sum_r\sigma^{-r}\boldsymbol\lambda_r,
\]
with $\Sigma$ permuting the words $(\boldsymbol\lambda_0,\ldots,\boldsymbol\lambda_{\ell-1})$ cyclically. 

The orbit of a nonconstant word is free, so the corresponding summand is induced from the trivial subgroup, hence has vanishing Tate construction and does not contribute to \eqref{eq:CT-brauer}. The constant word $\boldsymbol\lambda$ contributes $s_{N\boldsymbol\lambda*}$ of $(V_{\boldsymbol\lambda}\iellt)^{\otimes\ell}$, on which $\sigma$ acts by cyclic permutation.\footnote{In principle there could be a Koszul sign character, but it equals $(-1)^{(\ell-1)\sum_i\tw{2\rho_G,\lambda_i}}=1$.} The section $s_{N\boldsymbol\lambda}$ lands in the connected component of $\Gr_{T_H,(\Div_X^1)^I}$ indexed by $\sum_iN\lambda_i\in X_*(T_H)$, where $d$ takes the value $m_{\boldsymbol\lambda}$ below. So the contribution of $\boldsymbol\lambda$ to \eqref{eq:CT-brauer} is
\begin{equation}\label{eq: constant word Tate coefficient}
s_{N\boldsymbol\lambda*}\Bigl(\bigl((V_{\boldsymbol\lambda}\iellt)^{\otimes\ell}\bigr)^{t\Sigma}[m_{\boldsymbol\lambda}]\Bigr),
\end{equation}
where 
\begin{itemize}
\item $\delta(-)$ is as in Definition \ref{def:brauer-functor} composed with the obvious set map $\Z/2\Z \rightarrow \{0,1\}$. 
\item $m_{\boldsymbol\lambda}=\delta\Bigl(\sum_iN\lambda_i\Bigr)+\sum_i\tw{2\rho_H,N\lambda_i}-\sum_i\tw{2\rho_G,N\lambda_i}$. By \eqref{eq:delta-H-G}, $m_{\boldsymbol\lambda}$ is even. 
\end{itemize}
Let $\beta\in\rT^2(k)$ be the periodicity class of Example \ref{ex: Tate cohomology of trivial coeff}. By Proposition \ref{prop: Tate construction computes Tate cohomology} and Example \ref{ex: Tate cohomology of Nm}, the Tate diagonal followed by multiplication by $\beta^{m_{\boldsymbol\lambda}/2}$ gives a natural isomorphism of local systems
\begin{equation}\label{eq: normalized Tate diagonal local systems}
V_{\boldsymbol\lambda}\xrightarrow{\sim}\rH^0\Bigl(\bigl((V_{\boldsymbol\lambda}\iellt)^{\otimes\ell}\bigr)^{t\Sigma}[m_{\boldsymbol\lambda}]\Bigr).
\end{equation}
The Frobenius twists are arranged to make this identification $k$-linear. Since $s_{N\boldsymbol\lambda*}$ is exact, taking $\rH^0$ of the sum of the contributions \eqref{eq: constant word Tate coefficient} and using \eqref{eq: CT brauer H0} gives the isomorphism \eqref{eq: Brauer constant term weights}. From the definition, we see that it is natural in $X$.

\end{proof}

\begin{prop}\label{prop: Brauer perverse coefficients}
For every $n\in\Z$ there is an isomorphism, natural in $X$ and compatible with pullback in the base,
\begin{equation}\label{eq: Brauer perverse coefficients}
{}^p\rH^n(A_I(X)) \cong\wcbr^I(X).
\end{equation}
For even $n$ it is induced by multiplication by $\beta^{n/2}$, and for odd $n$ by multiplication by $\eta\beta^{(n-1)/2}$, where $\eta\in\rT^1(k)$ and $\beta \in \rT^2(k)$ are as in Example \ref{ex: Tate cohomology of trivial coeff}.
\end{prop}

\begin{proof}
By the lax monoidality of the Tate construction (\S \ref{ssec: Tate cup products}), $A_I(X)$ is a module over $k^{t\Sigma}$, so $\beta$ and $\eta$ act on it by maps $A_I(X)\rightarrow A_I(X)[2]$ and $A_I(X)\rightarrow A_I(X)[1]$. Multiplying by $\beta^{n/2}$ or $\eta \beta^{(n-1)/2}$, depending on whether $n$ is even or odd, gives a natural transformation
\[
\wcbr^I = {}^p\rH^0 \circ A_I \rightarrow {}^p\rH^n \circ A_I.
\] 
To check that this is an isomorphism, we may pull back along a finite \'etale cover of $(\Div_X^1)^I$, with which the formation of $A_I$ and of ${}^p\rH^n$ commutes (see the proof of Proposition \ref{prop: modular Brauer extraction}). Hence we may assume that $G$ has a $\Sigma$-stable Borus $(B,T)$ with $T$ split. To see that this is an isomorphism, we may apply $\CT_{B_H}[\deg_H]$, which is t-exact, conservative, and by Lemma \ref{lem: CT PSm} compatible with these module structures. By \eqref{eq: constant word Tate coefficient}, this reduces to showing that multiplication by $\beta^{n/2}$ and $\eta \beta^{(n-1)/2}$ are both isomorphisms on the Tate cohomology of $(V\iellt)^{\otimes\ell}$ with its cyclic rotation action, which is Example \ref{ex: Tate cohomology of Nm}.  
\end{proof}

\subsection{Tensor structure}\label{ssec: Brauer fusion} Recall that $\Sat_G^I(k)$ has a symmetric monoidal structure coming from fusion \cite[\S VI.9]{FS}. Next we will promote $\wcbr^I$ to a tensor functor, with respect to this structure.

\begin{thm}\label{thm: wtcbr tannakian functor}
For each non-empty finite set $I$, the functors 
\[
\wcbr^I\co\Sat_G^I(k)\rightarrow\Sat_H^I(k)
\]
are exact, faithful and symmetric monoidal, and linear over $\Loc((\Div_X^1)^I;k)$. They commute with fusion and with the ``coalescence maps'' induced by pullback along maps of finite sets $I \rightarrow J$, compatibly with base change along any $S\rightarrow(\Div_X^1)^I$.
\end{thm}

\begin{proof}
The assertions can be checked after pullback along a finite \'etale cover of $(\Div_X^1)^I$, with which the formation of $\wcbr^I$ commutes by Proposition \ref{prop: modular Brauer extraction}. As in the proof of that proposition, we may therefore reduce to the case that $G$ has a $\Sigma$-stable Borus $(B,T)$. 

Fix a partition $I=I_1\sqcup I_2$, and let
\[
j\co\bigl((\Div_X^1)^{I_1}\times(\Div_X^1)^{I_2}\bigr)_{\mrm{disj}}\inj(\Div_X^1)^{I}
\]
be the inclusion of the open subset where every leg indexed by $I_1$ is disjoint from every leg indexed by $I_2$. Over this subset the Hecke stack factors \cite[\S VI.9, before Definition/Proposition VI.9.4]{FS}:
\[
j^*\cHck_{G,(\Div_X^1)^I}\cong j^*\bigl(\cHck_{G,(\Div_X^1)^{I_1}}\times\cHck_{G,(\Div_X^1)^{I_2}}\bigr),
\]
where $j^*$ denotes base change along $j$, and similarly for $H$. We write $P\boxtimes Q$ for the exterior product of complexes $P$ and $Q$ on the two factors, restricted to this open subset. By \cite[Proposition VI.9.3 and Definition/Proposition VI.9.4]{FS}, restriction along $j$ is fully faithful on Satake categories, and for $X\in\Sat_G^{I_1}(k)$ and $Y\in\Sat_G^{I_2}(k)$ the fusion product $X\ast Y\in\Sat_G^I(k)$ is characterized as the unique object of $\Sat_G^I(k)$ satisfying $j^*(X\ast Y)\cong X\boxtimes Y$. We first construct an isomorphism 
\begin{equation}
\wcbr^{I_1}(X)\boxtimes\wcbr^{I_2}(Y)\cong j^*\wcbr^I(X\ast Y).
\end{equation}

\emph{Step 1: perverse truncation and exterior products.} Let $P\in D_{\et}(\cHck_{H,(\Div_X^1)^{I_1}};k)^{\bd}$ and $Q\in D_{\et}(\cHck_{H,(\Div_X^1)^{I_2}};k)^{\bd}$ be such that ${}^p\rH^0(P)$ and ${}^p\rH^0(Q)$ are Satake. It is clear from the definition that the exterior product preserves ${}^pD^{\leq0}$. Moreover, exterior products of Satake objects are Satake \cite[\S VI.9]{FS}. Applying ${}^p\rH^0$ to the truncation triangles and to ${}^p\tau^{\leq0}P\boxtimes{}^p\tau^{\leq0}Q\rightarrow P\boxtimes Q$ therefore gives a natural map
\begin{equation}\label{eq: exterior perverse extraction}
{}^p\rH^0(P)\boxtimes{}^p\rH^0(Q)\cong{}^p\rH^0\bigl({}^p\tau^{\leq0}P\boxtimes{}^p\tau^{\leq0}Q\bigr)\rightarrow{}^p\rH^0(P\boxtimes Q).
\end{equation}
\emph{Step 2: the comparison map.} After restriction along $j$, reordering the factors by the commutativity constraint gives $j^*\Nm\iellt(X\ast Y)\cong\Nm\iellt(X)\boxtimes\Nm\iellt(Y)$, compatibly with the $\Sigma$-equivariant structures, and $\iota_I$ restricts to $\iota_{I_1}\times\iota_{I_2}$. Hence the cup product \eqref{eq: exterior Tate cup} gives a map
\[
\bigl(\iota_{I_1}^*\Nm\iellt X\bigr)^{t\Sigma}\boxtimes\bigl(\iota_{I_2}^*\Nm\iellt Y\bigr)^{t\Sigma}\xrightarrow{\cup}j^*\bigl(\iota_I^*\Nm\iellt(X\ast Y)\bigr)^{t\Sigma}.
\]

Now we need to take care of sign issues. We use the standard convention that $C[a]\otimes D[b]\rightarrow(C\otimes D)[a+b]$ sends $x\otimes y$ to $(-1)^{b|x|}x\otimes y$, where $|x|$ is the degree of $x$ in $C$. On connected components $c$ and $c'$ of the two factors, put
\begin{equation}\label{eq: Brauer shift carry}
a(c,c')=\tfrac12\bigl(\delta(c)+\delta(c')-\delta(c+c')\bigr)\in\{0,1\},\qquad\epsilon(c,c')=(-1)^{\delta(c)\dagger_H(c')}.
\end{equation}
We define a map 
\begin{equation}
A_{I_1}(X)\boxtimes A_{I_2}(Y) \rightarrow j^* A_I(X \ast Y)
\end{equation}
as the composite
\begin{equation}\label{eq: exterior normalized Smith cup}
\begin{tikzcd}[ampersand replacement=\&, column sep=2.4em]
A_{I_1}(X)\boxtimes A_{I_2}(Y)\ar[r,"\sim"] \& \Bigl(\bigl(\iota_{I_1}^*\Nm\iellt X\bigr)^{t\Sigma}\boxtimes\bigl(\iota_{I_2}^*\Nm\iellt Y\bigr)^{t\Sigma}\Bigr)[\delta(c)+\delta(c')] \ar[d,"\cup"] \\
j^*A_I(X\ast Y) \& j^*\bigl(\iota_I^*\Nm\iellt(X\ast Y)\bigr)^{t\Sigma}[\delta(c)+\delta(c')] \ar[l,"\epsilon\beta^{-a}"']
\end{tikzcd}
\end{equation}
where the horizontal map $\epsilon\beta^{-a}$ in the bottom row denotes multiplication by $\epsilon(c,c')\beta^{-a(c,c')}$. Applying ${}^p\rH^0$ and combining with \eqref{eq: exterior perverse extraction}, we obtain a natural transformation
\begin{equation}\label{eq: exterior Brauer comparison}
\begin{tikzcd}[ampersand replacement=\&, column sep=2.4em]
\wcbr^{I_1}(X)\boxtimes\wcbr^{I_2}(Y) \ar[r,"\eqref{eq: exterior perverse extraction}"] \& {}^p\rH^0\bigl(A_{I_1}(X)\boxtimes A_{I_2}(Y)\bigr) \ar[r,"\eqref{eq: exterior normalized Smith cup}"] \& j^*\wcbr^I(X\ast Y).
\end{tikzcd}
\end{equation}

\emph{Step 3: analyzing constant terms.} Recall that we have reduced to the case where $G$ has a $\Sigma$-stable Borus $(B,T)$, with $\Sigma$-fixed points $(B_H,T_H)$. In order to check that \eqref{eq: exterior Brauer comparison} is an isomorphism, we apply $\CT_{B_H}[\deg_H]\boxtimes\CT_{B_H}[\deg_H]$ to it. Write
\[
\CT_B[\deg_G](X)=\bigoplus_{\boldsymbol\lambda}s_{\boldsymbol\lambda*}V_{\boldsymbol\lambda},\qquad\CT_B[\deg_G](Y)=\bigoplus_{\boldsymbol\mu}s_{\boldsymbol\mu*}W_{\boldsymbol\mu}
\]
as in \eqref{eq: torus weight sections}, with $\boldsymbol\lambda\in X_*(T)^{I_1}$ and $\boldsymbol\mu\in X_*(T)^{I_2}$. Since normalized constant terms commute with fusion \cite[Proposition VI.9.6]{FS}, the corresponding decomposition of $\CT_B[\deg_G](X\ast Y)$ has the summands $s_{(\boldsymbol\lambda,\boldsymbol\mu)*}(V_{\boldsymbol\lambda}\boxtimes W_{\boldsymbol\mu})$. Applying \eqref{eq: Brauer constant term weights} to $X$, $Y$ and $X\ast Y$, the source and the target of \eqref{eq: exterior Brauer comparison} become, after $\CT_{B_H}[\deg_H]$,
\[
\bigoplus_{\boldsymbol\lambda,\boldsymbol\mu}s_{N\boldsymbol\lambda*}V_{\boldsymbol\lambda}\boxtimes s_{N\boldsymbol\mu*}W_{\boldsymbol\mu}\qquad\text{and}\qquad j^*\bigoplus_{\boldsymbol\lambda,\boldsymbol\mu}s_{N(\boldsymbol\lambda,\boldsymbol\mu)*}\bigl(V_{\boldsymbol\lambda}\boxtimes W_{\boldsymbol\mu}\bigr),
\]
which agree summand by summand. It therefore suffices to compute the map on the summand indexed by a pair $(\boldsymbol\lambda,\boldsymbol\mu)$; write $V=V_{\boldsymbol\lambda}$ and $W=W_{\boldsymbol\mu}$. Lemma \ref{lem: CT PSm} is compatible with cup products and with the K\"unneth isomorphism \eqref{eq: Kunneth compact supports}, so the map itself is computed by cup products of the classes \eqref{eq: normalized Tate diagonal local systems}. This explicit presentation shows that \eqref{eq: exterior Brauer comparison} is an isomorphism.

\emph{Step 4: coherence.} Normalized constant terms carry the modified commutativity constraint of the Satake categories to the ordinary commutativity constraint of local systems, and the classes \eqref{eq: normalized Tate diagonal local systems} have even degree, so \eqref{eq: exterior Tate cup} introduces no sign under symmetry. Hence, after the faithful functor $\CT_{B_H}[\deg_H]$, the unit, associativity and symmetry diagrams for \eqref{eq: exterior Brauer comparison} become the corresponding diagrams of tensor isomorphisms of local systems, and they commute.

\emph{Step 5: extension across collisions.} Both $\wcbr^I(X\ast Y)$ and $\wcbr^{I_1}(X)\ast\wcbr^{I_2}(Y)$ lie in $\Sat_H^I(k)$, so by the full faithfulness in \cite[Proposition VI.9.3]{FS} the isomorphism \eqref{eq: exterior Brauer comparison} and the attendant coherent diagrams extend uniquely across $(\Div_X^1)^I$. Restricting to diagonals gives the symmetric monoidal structure on $\wcbr^I$ for each $I$ \cite[Definition/Proposition VI.9.4]{FS}, and the compatibility with base change in Proposition \ref{prop: modular Brauer extraction} identifies these structures under all coalescence maps.

The final assertion to establish is linearity over tensoring with $L \in \Loc((\Div_X^1)^I;k)$. This follows from the argument of Example \ref{ex: Tate cohomology of Nm}, which supplies a natural isomorphism
\[
\rH^0\Bigl(\bigl((L\iellt)^{\otimes\ell}\bigr)^{t\Sigma}\Bigr)\cong L.
\]
\end{proof}

\begin{lemma}\label{lem: CT br}
Suppose that $G$ has a $\Sigma$-stable Borus $(B,T)$ with $T$ split, and let $N_*$ be the tensor functor on torus Satake objects which sends a summand $s_{\boldsymbol\lambda*}V$ to $s_{N\boldsymbol\lambda*}V$. There is an isomorphism of tensor functors
\begin{equation}\label{eq: CT br}
\CT_{B_H}[\deg_H]\circ\wcbr^I\cong N_*\circ\CT_B[\deg_G],
\end{equation}
compatible with fusion and coalescence.
\end{lemma}

\begin{proof}
The underlying isomorphism of additive functors is \eqref{eq: Brauer constant term weights}. It is compatible with the tensor structures away from the collision locus by the proof of Theorem \ref{thm: wtcbr tannakian functor}, and hence everywhere by \cite[Proposition VI.9.3]{FS}.
\end{proof}

\begin{prop}\label{prop: br compatible with fiber functor}
After passing to an extension of $E$ over which there is a split $\Sigma$-stable Borus, there is an isomorphism of fiber functors
\begin{equation}\label{eq: Brauer local fiber functor}
\msf F^{H,I}\circ\wcbr^I\cong\msf F^{G,I},
\end{equation}
and in particular on geometric fibers. There are also natural isomorphisms $\wcbr^I\circ\sigma^*\cong\wcbr^I$, compatible with fusion.
\end{prop}

\begin{proof}
As remarked in the proof of \cite[Theorem VI.11.1]{FS}, the fiber functor is identified (as tensor functors) with 
\[
\bigoplus_i H^i \circ \rR \pi_{G*} \cong H^0 \circ \rR\pi_{T*} \CT_B[\deg_G]
\]
In view of this, the identification \eqref{eq: Brauer local fiber functor} follows from Lemma \ref{lem: CT br}. The last assertion follows from observing that cyclic rotation of the factors identifies $\Nm(\sigma^*X)$ with $\Nm(X)$ as $\Sigma$-equivariant objects, compatibly with the commutativity constraints.
\end{proof}

\section{The $\sigma$-dual homomorphism}\label{sec: sigma-dual homomorphism}

In this section, we apply Tannaka duality to the Brauer functor to construct the $\sigma$-dual homomorphism 
\[
\ld \psi \co \ld H \rightarrow \ld G
\]
which lifts the Brauer homomorphism on spherical Hecke algebras. Such a map was constructed in \cite{TV}, when $G$ is simply connected and $H$ is semisimple, unless $G$ has type E$_6$. Our construction has the following advantages: it works in full generality (without any technical assumptions on $G$ or $H$), and it is canonical (while the map in \cite{TV} is not uniquely characterized).

\subsection{The $c$-group}
Actually, the most canonical object uses a different incarnation of the $L$-group. There are (at least) two normalizations of the dual group of $G/E$. 
\begin{itemize}
  \item We write $\ld G \cong \wh{G} \rtimes W_E$ for the \emph{$L$-group} in the sense of Langlands. 
  \item There is also the \emph{$c$-group} $\lc G$ as in \cite[\S 7.8]{TVarxiv}. 
\end{itemize}
Over a $p$-adic local field, a choice of square root of the cyclotomic character (which always exists, for example by choosing a square root of the residue cardinality) induces an isomorphism $\ld G \cong \lc G$ by \cite[Lemma 5.5.7]{ZhuIntro17}. However, because of the non-canonicity of such choices, it is necessary to distinguish these groups for certain purposes. While the $L$-group is the incarnation more typically treated in the literature, the $c$-group is the incarnation that naturally shows up in the Geometric Satake equivalence (see \cite[VI.11]{FS} and \cite[Remark 5.5.11]{ZhuIntro17}):
\[
\Rep_{\Loc(\Div_X^1; \Lambda)}(\chG) \cong \Rep_{\Rep_\Lambda(W_E)}(\chG) \cong \Rep_\Lambda(\lc G).
\]

\begin{defn}
For homomorphisms $\phi_1, \phi_2 \co \ld H \rightarrow \ld G$ over $W_E$, we write $\phi_1 \sim \phi_2$ if they are conjugate by an element of $\wh{G}(k)$, and similarly for homomorphisms $\lc H \rightarrow \lc G$ over $W_E$, with $\chG(k)$ in place of $\wh{G}(k)$. A diagram of such homomorphisms \emph{commutes up to $\sim$} if its two composites are related by $\sim$. 
\end{defn}

\subsection{Construction of the $\sigma$-dual homomorphism}\label{ssec: sigma-dual homomorphism}

\begin{thm}\label{thm: 1-leg brauer functor} The functor
\[
\wcbr \co \Sat(\Gr_{G, \Div_X^1};k) \rightarrow \Sat(\Gr_{H, \Div_X^1};k)
\]
corresponds under the Geometric Satake equivalence to restriction along a continuous homomorphism $\lc\psi\co\lc H\rightarrow\lc G$ over $W_E$, which is algebraic on $\chH$ and well defined up to $\chG(k)$-conjugacy.
\end{thm}

\begin{proof}
Choose a tensor identification of the geometric fiber functors supplied by Proposition \ref{prop: br compatible with fiber functor}. The assertion follows from Theorem \ref{thm: wtcbr tannakian functor} by relative Tannakian duality over $\Loc(\Div_X^1;k)$, as in \cite[\S VI.10]{FS}; changing the identification conjugates the resulting homomorphism by $\chG(k)$.
\end{proof}

We fix a representative $\lc\psi$, and write $\chpsi\co\chH\rightarrow\chG$ for its restriction to $\chH$.

\begin{thm}\label{thm: relative functoriality} The Tannakian functors 
\[
\wcbr^I \co \Sat_G^I(k) \rightarrow \Sat_H^I(k)
\]
from \eqref{eq: normalized brauer functor} satisfy the following properties.

(1) Under the Geometric Satake equivalence \eqref{eq: geom sat Div} of Fargues-Scholze, $\wcbr^I$ corresponds to the restriction $\lc \psi^* \co \Rep_k(\lc G)^{\otimes I} \rightarrow \Rep_k(\lc H)^{\otimes I}$ induced by the $\sigma$-dual homomorphism $\lc \psi  \co \lc H \rightarrow \lc G$ from Theorem \ref{thm: 1-leg brauer functor}. (In particular, $\wcbr^{\{1\}}$ agrees with the $\wcbr$ from \eqref{eq: normalized brauer functor}.)

(2) (Naturality in $I$) For any map of non-empty finite sets $\zeta \co I \rightarrow J$, inducing the fusion product $\zeta \co \Sat_G^I(k) \rightarrow \Sat_G^J(k)$ and similarly for $H$, there is a commutative square
\[
\begin{tikzcd}
 \Sat_G^I(k) \ar[r, "\wcbr^I"] \ar[d, "\zeta"] &  \Sat_H^I(k) \ar[d, "\zeta"] \\
 \Sat_G^J(k) \ar[r, "\wcbr^J"]  & \Sat_H^J(k)
\end{tikzcd}
\]
such that the implicit natural isomorphisms are compatible with compositions $I \rightarrow J \rightarrow K$. 
\end{thm}

\begin{proof}
Write $I = \bigsqcup_{i \in I} \{i\}$. We bootstrap from the case $|I|=1$ using the convolution Hecke stack
\[
\wt{\cHck}_G^I  := \cHck_G^{I; \{i\}_{i \in I} }\rightarrow (\Div^1_X)^I
\]
defined in \cite[p. 226]{FS}. We have projection maps $p_i \co \wt{\cHck}_G^I  \rightarrow \wt{\cHck}_{G, \Div_X^1}
^{\{i\}}$ for each $i \in I$, as well as a convolution map $m \co 
\wt{\cHck}_G^I \rightarrow 
\cHck_{G, (\Div_X^1)^I}$. The map 
\[
(\cK_i)_{i \in I} \mapsto Rm_! (\otimes_{i \in I} \ p_i^* \cK_i) \in \Sat_G^I(k)
\]
induces an equivalence 
\begin{equation}\label{eq: multileg convolution equivalence}
\conv \co \Sat(\Gr_{G, \Div^1_X})^{\otimes I} \xrightarrow{\sim} \Sat_G^I(k) .
\end{equation}
Here $(-)^{\otimes I}$ denotes Deligne's tensor product over $k$, and $\conv$ is an equivalence by \cite[Proposition VI.10.3]{FS}.

The same considerations apply to $H$. By the compatibility with fusion in Theorem \ref{thm: wtcbr tannakian functor}, $\wcbr^I$ makes the following diagram commute:
\begin{equation}\label{eq: multileg brauer tensor powers}
\begin{tikzcd}
\Sat_G(k)^{\otimes I}  \ar[d, "\sim"', "\conv"] \ar[r, "\wt{\cbr}^{\otimes I}"]    &  \Sat_H(k)^{\otimes I}  \ar[d, "\sim"', "\conv"]    \\
 \Sat_G^I(k) \ar[r, dashed, "\wt{\cbr}^I"]  &   \Sat_H^I(k) 
\end{tikzcd}
\end{equation}
Then property (1) follows from the $|I|=1$ case, which was arranged in Theorem \ref{thm: 1-leg brauer functor}, and Tannakian reconstruction \cite[\S VI.10]{FS}. 


Property (2), including the compatibility with compositions, is the compatibility with fusion and coalescence in Theorem \ref{thm: wtcbr tannakian functor}.

\end{proof}

On maximal tori, Lemma \ref{lem: CT br} gives the following.

\begin{cor}\label{cor: dual torus}
Let $\chT_H, \chT$ be the canonical maximal tori in $\chH, \chG$, respectively. Some $\chG(k)$-conjugate of $\chpsi$ carries $\chT_H$ into $\chT$, and its restriction $\chT_H\rightarrow\chT$ corresponds to the map $X^*(\chT) \rightarrow  X^*(\chT_H) = X^*(\chT)^\sigma$ given by applying $N = (1+ \sigma + \ldots + \sigma^{\ell-1})$. 
\end{cor}

Given choices of isomorphisms $\ld G \xrightarrow{\sim} \lc G$ and $\ld H \xrightarrow{\sim} \lc H$, we can transport $\lc \psi$ to a map $\ld \psi \co \ld H \rightarrow \ld G$ characterized by a commutative diagram 
\[
\begin{tikzcd}
\lc H \ar[r, "\lc \psi"] \ar[d, "\wr"] & \lc G \ar[d, "\wr"] \\
\ld H \ar[r, "\ld \psi"] & \ld G
\end{tikzcd}
\]
However, the choices are not canonical, and this can create genuine incompatibilities when comparing $\sigma$-dual homomorphisms with the same domain and codomain that arise in different ways, as will be seen in the example of base change. Also, in the context of global number fields Treumann-Venkatesh pointed out \cite[\S 7.8]{TVarxiv} that $\ld \psi$ may not exist. Without the simply connected hypotheses, a $\sigma$-dual homomorphism need not exist with the ``usual'' $L$-group defined by Langlands, and they predicted that using instead the ``$c$-group'' might fix this issue.

\begin{remark}
The second main theorem of Treumann-Venkatesh (see \cite[\S 1.3]{TV}) is the construction of a $\sigma$-dual homomorphism when $G$ is simply connected and $H$ is semisimple, with three possible exceptions when $G$ has type $\rE_6$. Their proof is based on classification of all possible examples, and then case-by-case analysis of each. By contrast, our construction is completely uniform. 
\end{remark}

\subsection{Compatibilities}
We record some compatibilities of $\lc\psi$.

\subsubsection{Base change}\label{sssec:base-change}

Let $E'/E$ be a finite extension. Then $H_{E'} = G_{E'}^\sigma$, and Theorem \ref{thm: 1-leg brauer functor} gives a map  
\begin{equation}\label{eq:lcpsi_bc}
\lc (\psi_{E'}) \co \lc(H_{E'}) \rightarrow \lc (G_{E'}).
\end{equation}
On the other hand, let $(\bullet)|_{E'}$ denote restriction along the inclusion $W_{E'} \inj W_E$. Applying this to $\lc \psi \co \lc H \rightarrow \lc G$, we obtain a map  
\begin{equation}\label{eq:ldpsi_restrict}
(\lc \psi)|_{E'} \co (\lc H)|_{E'} \rightarrow (\lc G)|_{E'}
\end{equation}
Note that we have natural isomorphisms $(\lc H)|_{E'} \cong \lc (H_{E'})$ and $(\lc G)|_{E'} \cong \lc (G_{E'})$. The following Lemma expresses compatibility of the formation of $\lc\psi$ with respect to the base change along $E \rightarrow E'$. 

\begin{lemma}\label{lem:psi-bc}
Under the identifications $(\lc H)|_{E'}\cong\lc(H_{E'})$ and $(\lc G)|_{E'}\cong\lc(G_{E'})$, we have $(\lc\psi)|_{E'}\sim\lc(\psi_{E'})$; that is, the diagram
\[
\begin{tikzcd}
{(\lc H)|_{E'}} \ar[r, "{(\lc \psi)|_{E'}}"] \ar[d, "\wr"] &  (\lc G)|_{E'} \ar[d, "\wr"] \\
\lc (H_{E'}) \ar[r, "\lc (\psi_{E'})"] &  \lc (G_{E'})
\end{tikzcd}
\]
commutes (up to $\chG(k)$-conjugacy).
\end{lemma}

\begin{proof}
Let $\Div_{X'}^1\rightarrow\Div_X^1$ be the finite \'etale cover corresponding to $E'$. The construction \eqref{eq: normalized Smith coefficient} commutes with pullback along this map by Example \ref{ex: Tate commutation instances}(1), and so does perverse truncation. Hence the pullback of $\wcbr$ to $\Div_{X'}^1$ is isomorphic, as a tensor functor, to the Brauer functor for $G_{E'}$ (Theorem \ref{thm: wtcbr tannakian functor}), and the assertion follows by Tannaka duality.
\end{proof}

\begin{remark}\label{rem:ldpsi-base-change}
If the square root of the cyclotomic character used over $E'$ is the restriction of the one used over $E$, then the same compatibility holds for $\ld\psi\co\ld H\rightarrow\ld G$ on all of $W_{E'}$. For independently chosen unramified square roots, it holds at least after restriction to the subgroup of $W_{E'}$ mapping to $2\Z$ under the valuation.
\end{remark}

\subsubsection{Products}\label{ssec: sigma-dual products} Suppose we have factorizations $G \cong G_1 \times G_2$, $\sigma \cong \sigma_1 \times \sigma_2$, hence $H = H_1 \times H_2$ where $H_i=G_i^{\sigma_i}$. Theorem \ref{thm: 1-leg brauer functor}, applied to each factor, gives
\[
\lc \psi_1 \co \lc H_1 \rightarrow \lc G_1, \quad \text{and} \quad \lc \psi_2 \co \lc H_2 \rightarrow \lc G_2.
\]

\begin{lemma}\label{lem:psi-factorization}
Under the isomorphisms $\lc H\cong\lc H_1\times_{W_E}\lc H_2$ and $\lc G\cong\lc G_1\times_{W_E}\lc G_2$, we have 
\[
\lc\psi\sim\lc\psi_1\times_{W_E}\lc\psi_2.
\]
\end{lemma}


\begin{proof}
We have $\Gr_{G,\Div_X^1}=\Gr_{G_1,\Div_X^1}\times_{\Div_X^1}\Gr_{G_2,\Div_X^1}$, and similarly for $H$. For Satake objects $X_i$ for $G_i$, the norm of $X_1\boxtimes X_2$ is $\Nm(X_1)\boxtimes\Nm(X_2)$, so the cup product \eqref{eq: exterior Tate cup}, normalized as in \eqref{eq: exterior normalized Smith cup}, gives a natural isomorphism
\begin{equation}\label{eq: Brauer product comparison}
\wcbr(X_1\boxtimes X_2)\cong\wcbr_1(X_1)\boxtimes\wcbr_2(X_2).
\end{equation}
As in the proof of Theorem \ref{thm: wtcbr tannakian functor}, after normalized constant terms this is the identity on weight spaces; hence it is compatible with the tensor structures. Applying Tannaka duality gives the claim.
\end{proof}

\subsection{Maps inducing isomorphisms on adjoint groups}\label{ssec: sigma-dual central isogenies} Let $G\rightarrow G'$ be a homomorphism of reductive groups over $E$ which induces an isomorphism $G_{\ad}\xrightarrow{\sim}G'_{\ad}$ (e.g., a central isogeny). Then its kernel is central, it identifies the roots of $G$ with those of $G'$, and it restricts to an isomorphism $G_{\mrm{sc}}\xrightarrow{\sim}G'_{\mrm{sc}}$ of the simply connected covers of the derived groups. The induced map $g\co\Gr_{G,\Div_X^1}\rightarrow\Gr_{G',\Div_X^1}$ maps each geometric connected component isomorphically onto a connected component. Under Geometric Satake, the restriction along the dual homomorphism $\lc G'\rightarrow\lc G$ corresponds to the pushforward 
\begin{equation}\label{eq: central isogeny} 
\Sat(\Gr_{G,\Div_X^1};k) \xrightarrow{\rR g_*} \Sat(\Gr_{G',\Div_X^1};k).
\end{equation}

Now suppose that $\sigma$ acts compatibly on $G,G'$ with fixed subgroups $H, H'$, and let $f \co \Gr_{H,\Div_X^1} \rightarrow \Gr_{H',\Div_X^1}$ be the induced map. The map $H\rightarrow H'$ again induces an isomorphism on adjoint groups (recall that $H'=G'^\sigma$ is assumed connected). 

\begin{lemma}\label{lem:psi-central-isogeny}
Let $\lc\psi\co\lc H\rightarrow\lc G$ and $\lc\psi'\co\lc H'\rightarrow\lc G'$ be the $\sigma$-dual homomorphisms. Then the diagram
\begin{equation}
\begin{tikzcd}
\lc H' \ar[r]  \ar[d, "\lc \psi'"] & \lc H  \ar[d, "\lc \psi"] \\
\lc G' \ar[r] & \lc G  
\end{tikzcd}
\end{equation}
commutes (up to $\chG(k)$-conjugacy).
\end{lemma}

\begin{proof}
By Tannaka duality it suffices to give a tensor isomorphism $\rR f_*\circ\wcbr\cong\wcbr'\circ\rR g_*$, compatible with fusion. Only finitely many connected components meet any bounded support, so on bounded supports $f$ and $g$ are proper, and $\rR f_*$, $\rR g_*$ are perverse t-exact because $f$ and $g$ map components isomorphically onto components. Since the maps $G\rightarrow G'$ and $H\rightarrow H'$ identify roots, $\dagger_G$ and $\dagger_H$ are pulled back from $\dagger_{G'}$ and $\dagger_{H'}$, so $\delta$ is the pullback of $\delta'$. The tensor functor $\rR g_*$ commutes with the norms, and $f$ is the map induced by $g$ on fixed points (Proposition \ref{prop: BD gr fixed points}), so Proposition \ref{prop: realized Smith localization} gives
\begin{equation}\label{eq: Brauer central comparison}
\rR f_*A_1(X)\cong A'_1(\rR g_*X).
\end{equation}
Taking ${}^p\rH^0$ implies that $\rR f_*\circ\wcbr\cong\wcbr'\circ\rR g_*$. This isomorphism is compatible with exterior products and fusion, which proves the lemma.
\end{proof}

\subsection{Examples of the $\sigma$-dual homomorphism}\label{ssec: examples sigma-dual}

\begin{prop}\label{prop: sigma-dual tori}
Suppose that $G=T$ is a torus and that $H=T^\sigma$ is connected, and let $N\co X_*(T)\rightarrow X_*(H)$ be the norm. Then the $\sigma$-dual homomorphism may be taken to be $\chpsi\rtimes\Id_{W_E}$, where $\chpsi\co\chH\rightarrow\chT$ is dual to $N$. In particular, if $\Sigma$ acts trivially on $T$, then $\chpsi$ is the $\ell$-power map.
\end{prop}

\begin{proof}
This follows from Proposition \ref{prop: modular Brauer extraction}.
\end{proof}

\begin{prop}[Cyclic base change]\label{prop: sigma-dual cyclic base change}
Let $E'/E$ be cyclic of degree $\ell$, let $G=\Res_{E'/E}(H_{E'})$, and let $\Sigma=\Gal(E'/E)$ act in the natural way on $G$, so that $G^\sigma=H$. Setting $\Omega:=\Hom_E(E',E^s)$, identify $\chG=\prod_{\tau\in\Omega}\chH$, with $W_E$-action
\begin{equation}\label{eq: base change pinned Weil action}
w \cdot ((g_\tau)_\tau)=( w \cdot (g_{w^{-1}\tau}))_\tau.
\end{equation}
Let $\Delta\co\chH\rightarrow\chG$ be the diagonal. Then $\lc\psi\sim\lc\psi_{\mathrm{BC}}$, where
\begin{equation}\label{eq: base change diagonal c-map}
\lc\psi_{\mathrm{BC}}(h,w)=(\Delta(h),w).
\end{equation}
\end{prop}

\begin{proof}
Choose a Borus $(B,T)$ of $H_{E^s}$ and use the product Borus of $G_{E^s}=\prod_\Omega H_{E^s}$, which is $\Sigma$-stable. Then Proposition \ref{prop: br compatible with fiber functor} gives a comparison
\begin{equation}\label{eq: base change fiber comparison}
\xi_{B,T}\co\msf F^H\circ\wcbr\xrightarrow{\sim}\msf F^G.
\end{equation}
We claim that $\xi_{B,T}$ does not depend on $(B,T)$. As $H(E^s)$ acts transitively on its Borus pairs, it suffices to see that $\xi_{B,T}$ is unchanged by simultaneous conjugation by $u\in H(E^s)$, acting through $\Delta(u)$ on $\Gr_G$. Every map entering the construction of $\xi_{B,T}$ commutes with this conjugation, so it suffices that $u$ and $\Delta(u)$ act trivially on the fiber functors. This holds because the action of the connected group $L^+H$ on the cohomology of an equivariant complex with bounded support is constant, and hence trivial \cite[Propositions VI.1.10, VI.2.3 and VI.2.8]{FS}. By transport of structure,
\[
w \cdot \xi_{B,T} = \xi_{wB, wT} \quad \text{ for all } w \in W_E,
\] 
and combined with the claim, this implies that $\xi:=\xi_{B,T}$ is $W_E$-equivariant.
This equivariance implies that with the choice of fiber functors in \eqref{eq: base change fiber comparison}, $\lc\psi$ has the form $(h,w)\mapsto(\phi(h),w)$ for some $W_E$-equivariant map $\phi \co \chH \rightarrow \chG$. With respect to the factorization $\chG \cong \prod_{\tau \in \Omega} \chH$, write  
\[
\phi = (\phi_{\tau})_{\tau \in \Omega}.
\]
On the cocharacter lattice $X_*(T)^\Omega$ of the product torus, the norm map $N\co X_*(T)^\Omega\rightarrow X_*(T)$ is the sum of the coordinates, so by Lemma \ref{lem: CT br} the restriction of $\phi$ to $\chT_H$ is the diagonal. Thus each coordinate homomorphism $\phi_\tau\co\chH\rightarrow\chH$ restricts to the identity on $\chT_H$, and is therefore conjugation by an element of $\chT_H(k)$. 

We have established that $\phi=\Ad(t)\circ\Delta$ for some $t\in\chT_G(k)$. Since $\phi$ and $\Delta$ are $W_E$-equivariant and the centralizer of $\Delta(\chH)$ is $Z(\chG)$, the image $\bar t$ of $t$ in $\chT_G/Z(\chG)$ is fixed by $W_E$. Let $Q\subset X=X^*(\chT_G)$ be the root lattice and $\Gamma$ the finite image of $W_E$ in $\Aut(X)$. As $\Gamma$ permutes the simple roots, the coinvariant group $Q_\Gamma$ is free, so $Q_\Gamma\rightarrow X_\Gamma$ is injective (being injective rationally). Viewing $\bar t$ as a character of $Q_\Gamma$ and using divisibility of $k^\times$, we extend it to $X_\Gamma$ and obtain a $W_E$-fixed $g\in\chT_G(k)$ with image $\bar t$. Conjugating by $g^{-1}$ gives \eqref{eq: base change diagonal c-map}. 

\end{proof}

\section{Tate cohomology of moduli of local shtukas}\label{sec: local shtukas} 

\subsection{Moduli spaces of local shtukas}\label{ssec: shtukas}  The definitions of local shtukas in $p$-adic geometry are developed in Scholze's Berkeley Lectures on $p$-adic geometry, especially 	\cite[Lecture XXIII]{SW20}. Properties of the cohomology of their moduli spaces are established in \cite[\S IX.3]{FS}. We recall some relevant aspects here.  

\subsubsection{Setup} Recall that $G$ is a reductive group over $E$. For each finite set $I$, $b \in B(G)$, and compact open subgroup $K < G(E)$, we have a moduli space of local shtukas $\Sht_{(G,b,I),K}$. By \cite[Theorem 23.1.4]{SW20}, it is an inductive limit of locally spatial diamonds with finite cohomological dimension along a countable index set. Furthermore, there is an action of $G_b(E)$ on $\Sht_{(G,b,I),K}$. 

The space $\Sht_{(G,b,I),K}$ comes equipped with ``leg'' maps 
\[
f_K \co \Sht_{(G,b,I),K} \rightarrow \prod_{i \in I} \Spd \brE
\]
which are partially proper and ind-shriekable (cf. \cite[\S IX.3]{FS}), and a Grothendieck-Messing period map 
\[
\pi_K \co \Sht_{(G,b,I),K}  \rightarrow \Gr^{\twi}_{G, \prod_{i \in I} \Spd \brE},
\]
which is \'{e}tale \cite[Proposition 23.3.3, Corollary 23.4.4]{SW20}. 

\subsubsection{Satake sheaves}\label{sssec: shtuka Satake sheaves} As in \cite[\S IX]{FS}, fix an unramified square root $\chi_{\mrm{cyc}}^{1/2}\co W_E\rightarrow k^\times$ of the cyclotomic character. It gives an isomorphism $\alpha_G\co\ld G\xrightarrow{\sim}\lc G$ over $W_E$, identifying $\wh G$ with $\chG$ and changing the Weil splitting to account for the cyclotomic twist in Geometric Satake. We use the same square root for $H$.

Let $Q$ be a finite quotient of $W_E$ through which both the action on $\wh{G}$ and $\chi_{\mrm{cyc}}^{1/2}$ factor. For any finite set $I$ and $W \in  \Rep_k((\wh{G} \rtimes Q)^I)$, the Geometric Satake equivalence gives a relative perverse sheaf $\cS_W$ on $\Gr^{\twi}_{G, \prod_{i \in I} \Spd \brE}$ (cf. \cite[\S I.9, \S IX.3]{FS}) which we pull back via $\pi_K^*$ to $\Sht_{(G,b,I),K}$, and we also denote the resulting complex by $\cS_W \in D_{\et}^b(\Sht_{(G,b,I),K};k)$.

\subsubsection{Cohomology of moduli spaces of local shtukas}\label{sssec: coh of shtuka} By \cite[Corollary I.7.3, Proposition IX.3.2]{FS}, we may regard
\begin{equation}\label{eq: cohomology Weil action}
\rR f_{K!} \cS_W \in  D(\Rep^{\mrm{sm}}_k G_b(E))^{B\prod_{i \in I}W_E}
\end{equation}
as a (derived) representation of $G_b(E)$ with a commuting action of $\prod_{i \in I} W_{E}$. To demystify this a bit: the $G_b(E)$-action on cohomology is induced by the $G_b(E)$-action on the space $\Sht_{(G,b,I),K}$, and the $\prod_{i \in I} W_{E}$-action comes from a natural descent (using an interpretation via Hecke operators) of $\rR f_{K!} \cS_W $ from $\prod_{i \in I} \Spd \brE$ to $\prod_{i\in I} [\Spd \brE / \varphi] \cong \prod_{i \in I} [\Spd C/ W_E]$ where $C$ is a completed algebraic closure of $E$. The isomorphism $[\Spd \brE] / \varphi \cong [\Spd C/W_E]$ is on \cite[p.164]{FS}.

\begin{example}\label{ex: cohomology of trivial shtuka}For $W=\bbm{1}$, the sheaf $\cS_{\bbm{1}}$ is supported on the trivial-modification locus, which is non-empty only if $b=1_G$. In that case, $\rR f_{K!} \cS_{\bbm{1}}$ is the constant local system on $\prod_{i \in I} \Spd \brE$ with fiber $\cInd_K^{G(E)} (k)$, equipped with the obvious $G(E)$-action and the trivial $W_E$-action. 
\end{example}

The functor
\[ 
\Rep_k((\wh{G} \rtimes Q)^I) \rightarrow  D(\Rep^{\mrm{sm}}_k G_b(E))^{B\prod_{i \in I}W_E}
\]
sending $W \mapsto \rR f_{K!} \cS_W$ satisfies the following \emph{fusion compatibility}. Any map of finite non-empty sets $\zeta \co I \rightarrow J$ induces a map $(\ld G)^J \rightarrow (\ld G)^I$. Let $\Res_\zeta \co \Rep_k((\ld G)^I)  \rightarrow \Rep_k((\ld G)^J)$ be the restriction functor along this map. By construction, Geometric Satake is arranged so that the corresponding functor $\Sat_G^I(k) \rightarrow \Sat_G^J(k)$ is the fusion product. For surjective $\zeta$, we have a commutative diagram
\[
\begin{tikzcd}
\Sht_{(G,b,J),K} \ar[d, "\pi_K"] \ar[r]  & \Sht_{(G,b,I),K} \ar[d, "\pi_K"] \\
\Gr^{\twi}_{G, \prod_{j \in J} \Spd \brE} \ar[r] \ar[d] & \Gr^{\twi}_{G, \prod_{i \in I} \Spd \brE}  \ar[d]  \\
\prod_{j \in J} \Spd \brE \ar[r] & \prod_{i \in I} \Spd \brE
\end{tikzcd} 
\]
The Geometric Satake equivalence is constructed in \cite[\S VI]{FS} so that one has $\cS_W|_{\Gr^{\twi}_{G, \prod_{j \in J} \Spd \brE}} \cong \cS_{\Res_\zeta W}$, naturally in $W \in \Rep_k((\wh{G} \rtimes Q)^I)$ and compatibly under compositions of finite sets. This induces a natural isomorphism  
\begin{equation}\label{eq: fusion}
\left(\rR f_{K!} \cS_W\right)\big|_{\prod_{j \in J} \Spd \brE} \cong \rR f_{K!} \cS_{\Res_\zeta( W)} 
\end{equation}
of functors $\Rep_k((\wh{G} \rtimes Q)^I) \rightarrow  D(\Rep^{\mrm{sm}}_k G_b(E))^{B\prod_{j \in J}W_E}$. For arbitrary $\zeta$, factor it as $I\twoheadrightarrow\zeta(I)\hookrightarrow J$ and use the pull-push construction of Remark \ref{rem:fusion}, inserting trivial modifications at the unused legs. This gives \eqref{eq: fusion} for all $\zeta$. Moreover, the natural isomorphism \eqref{eq: fusion} is compatible under compositions of maps of finite sets.

\subsection{Excursion algebra}\label{ssec: excursion algebra}

An \emph{excursion datum} for $\wh{G}$ (over $k$) \cite[Definition VIII.4.2]{FS} is a tuple $\cD = (I, V, \alpha, \beta, (\gamma_i)_{i \in I})$ where: 
\begin{itemize}
\item $I$ is a finite set and $\gamma_i \in W_E$ for each $i \in I$,
\item $V \in \Rep_k((\wh{G} \rtimes Q)^I)$ for varying $Q$, and $ \bbm{1} \xrightarrow{\alpha} V|_{\wh{G}}$ and $  V|_{\wh{G}} \xrightarrow{\beta} \bbm{1} $ are maps of $\wh{G}$-representations. (Here $\bbm{1}$ is the trivial representation.)  
\end{itemize}
The \emph{excursion algebra} $\Exc_k(W_E, \wh{G})$ is the $k$-algebra on generators $S_{\cD}$ indexed by excursion data $\cD$, with relations as in \cite[\S 2.4]{Fe23}. Another definition appears in \cite[Definition VIII.3.4]{FS}. 

For an excursion datum $\cD$, let
\[
f_{\cD}((g_i)_i)=\beta\circ\rho_V((g_i)_i)\circ\alpha(1).
\]
This function is invariant under left and right translation by the diagonal copy of $\wh G$. In the presentation of \cite[Definition VIII.3.4 and \S VIII.4]{FS} by free groups, $S_{\cD}$ is determined by the restriction of $f_{\cD}$ to the component $\wh G^I\cdot(\bar\gamma_i)_i$, where $\bar\gamma_i$ is the image of $\gamma_i$ in $Q$.

An \emph{$L$-parameter} (with coefficients in $k$) is a class in $\rH^1(W_E; \wh{G}(k))$, or equivalently a section $W_E \rightarrow \wh{G}(k) \rtimes W_E$ up to $\wh{G}(k)$-conjugation. An $L$-parameter is \emph{semisimple} if whenever it factors through a parabolic $\ld P(k) \subset \ld G(k)$, it also factors through a Levi $\ld M(k) \subset \ld P(k)$ \cite[Definition VIII.3.1]{FS}. 

Combining \cite[Proposition VIII.3.8]{FS} and the statement above \cite[Definition VIII.3.4]{FS} that the natural map from $\Exc(W_E, \wh{G})$ to the \emph{spectral Bernstein center} $\cO(Z^1(W_E, \wh{G}))^{\wh{G}}$ is a universal homeomorphism, we obtain a canonical bijection between 
\[
\{\text{characters $\Exc_k(W_E, \wh{G}) \rightarrow k$}\} \longleftrightarrow \{\text{semisimple $L$-parameters $\rho \in \rH^1( W_E; \wh{G}(k))$}\}. 
\]


\begin{lemma}\label{lem: excursion pullback}
Let $\ld\psi\co\ld H\rightarrow\ld G$ be an $L$-homomorphism, and write $\ld\psi(h,w)=(\phi(h)z(w),w)$. Composition with $\ld\psi$ induces a $k$-algebra homomorphism
\begin{equation}\label{eq: excursion algebra functoriality}
\ld\psi^*\co\Exc_k(W_E,\wh G)\rightarrow\Exc_k(W_E,\wh H),
\end{equation}
which depends only on the $\wh G(k)$-conjugacy class of $\ld\psi$. For an excursion datum $\cD$, the element $\ld\psi^*S_{\cD}$ is represented on the component selected by $(\gamma_i)_i$ by the function
\[
(h_i)_i\mapsto f_{\cD}\bigl((\phi(h_i)z(\gamma_i),\bar\gamma_i)_i\bigr),
\]
and if $\ld\psi^*V$ has finite Weil image then $\ld\psi^*S_{\cD}=S_{\ld\psi^*\cD}$, where $\ld\psi^*\cD=(I,\ld\psi^*V,\ld\psi^*\alpha,\ld\psi^*\beta,(\gamma_i)_i)$. On $k$-points, $\ld\psi^*$ sends (the point corresponding to) a semisimple $L$-parameter $\rho \in \rH^1(W_E; \wh{H}(k))$ to (the point corresponding to) the semisimple $L$-parameter $(\ld \psi \circ \rho)^{\mrm{ss}} \in \rH^1(W_E;\wh{G}(k))$.
\end{lemma}

\begin{proof}
For each map $F_n\rightarrow W_E$ from a free group, composition with $\ld\psi$ defines a $\phi$-equivariant morphism $Z^1(F_n,\wh H)\rightarrow Z^1(F_n,\wh G)$, compatible with maps of free groups and with passage to the quotients by open subgroups of wild inertia as in \cite[\S VIII.3.2]{FS}. Pulling back invariant functions gives \eqref{eq: excursion algebra functoriality} by \cite[Definition VIII.3.4]{FS}. Conjugating $\ld\psi$ by $\wh G(k)$ does not change the pullback of invariant functions. The remaining assertions follow by evaluating on cocycles.
\end{proof}

\begin{remark}\label{rem: excursion continuous Weil}
Strictly speaking, the excursion operators of \cite{FS} are only defined for data where the $W_E$-action on $V$ factors through a finite quotient, while $\ld\psi^*V$ need not have this property, since the Weil component $z$ of $\ld\psi$ may have infinite image. We explain that the same formula in Lemma \ref{lem: excursion pullback} computes the excursion operators attached to $\ld\psi^*\cD$ when the $W_E^I$-action on $\ld\psi^*V$ is only continuous. Let $U$ be a finite-dimensional continuous representation of $(\ld H)^I$. Inertia acts on $U$ through a finite quotient, so there are powers of Frobenius in each factor which centralize the image of inertia and act trivially on $\wh H^I$. Their simultaneous generalized eigenspaces are subrepresentations, and on each of them an unramified twist makes these powers unipotent, hence of finite order since $k$ has characteristic $\ell$. Thus $U$ is a finite direct sum of unramified twists of representations with finite Weil image. Twisting by unramified characters $(\chi_i)_i$ multiplies both the excursion operator and the function $f_{\cD}$ by $\prod_i\chi_i(\gamma_i)$, so by additivity the formula reduces to the case where $W_E$ has finite image. In particular, the excursion operator attached to $\ld\psi^*\cD$ is $\FS_H(\ld\psi^*S_{\cD})$.
\end{remark}
 
\subsection{The Bernstein center}\label{ssec: Bernstein center}
Recall that in \eqref{eq: hecke algebra} we defined the Hecke algebra $\sH(G, K;\Lambda)$ for a compact open subgroup $K \subset G(E)$. We let $\mf{Z}(G,K; \Lambda) :=  Z(\sH(G,K; \Lambda))$ be the center of $\sH(G,K; \Lambda)$. 

If $K \subset K'$ have prime-to-$\ell$ pro-order (e.g., this will be true as long as they are sufficiently small), then convolution with the idempotent $[K':K]^{-1}\bbm{1}_{K'}$ gives a homomorphism $\mf{Z}(G,K; \Lambda) \rightarrow \mf{Z}(G,K'; \Lambda)$. The \emph{Bernstein center of $G$ (with coefficients in $\Lambda$)} is 
\[
\mf{Z}(G; \Lambda) := \varprojlim_{K} \mf{Z}(G,K; \Lambda),
\]
where the transition maps are as above, and the inverse limit is taken over $K$ with prime-to-$\ell$ pro-order.

If $\Lambda = k$, we abbreviate $\sH(G,K) := \sH(G,K; k)$, $\mf{Z}(G,K) := \mf{Z}(G,K; k)$, and $\mf{Z}(G) := \mf{Z}(G; k)$.

The Bernstein center $\mf{Z}(G)$ may also be identified with the ring of endomorphisms of the identity functor of the category $\Rep_k^{\mrm{sm}}(G(E))$ (cf. \cite[Definition I.9.2]{FS}). In particular, any irreducible smooth representation $\Pi$ of $G(E)$ over $k$ induces a character of $\mf{Z}(G)$.

\subsection{Review of Fargues-Scholze correspondence}\label{ssec: review FS}

Fargues-Scholze construct in \cite{FS} a $k$-algebra homomorphism 
\begin{equation}\label{eq: FS_G}
\FS_G \co \Exc_k(W_E, \wh{G}) \rightarrow \mf{Z}(G; k).
\end{equation}
Via \eqref{eq: FS_G}, $\Exc_k(W_E, \wh{G})$ acts through a character on any irreducible admissible representation $\Pi \in \Rep_k^{\mrm{sm}} (G(E))$, and the corresponding semisimple $L$-parameter is denoted $\rho_\Pi \in \rH^1(W_E; \wh{G}(k))$.

We present the construction of \eqref{eq: FS_G} on \cite[p.36]{FS}. Fix $b :=1_G \in B(G)$. Let
\[
\ol{x} \co \Spd \wh{\ol{E}} \rightarrow \prod_{i \in I} \Spd \brE
\]
be the geometric diagonal, and $f_K^\Delta \co \Sht_{(G,1_G, I),K}^\Delta \rightarrow \Spd \wh{\ol{E}}$ be the base change of $f_K$ along $\ol{x}$. 

Let $K \subset G(E)$ be a compact open subgroup. Recall from Example \ref{ex: cohomology of trivial shtuka} that the underlying $G(E)$-representation of $\rR f_{K!}^\Delta \cS_{\bbm{1}}$ is $\cInd_K^{G(E)} k$. Hence for each excursion datum $\cD = (I, V, \alpha, \beta, (\gamma_i)_{i \in I})$ we have a composition 
\begin{equation}\label{eq: excursion action}
\adjustbox{scale=0.9,center}{\begin{tikzcd}
\cInd_K^{G(E)} k \ar[rrrrr, dashed, "\FS_G(S_{\cD}) "] \ar[d, equals]  & & & & & \cInd_K^{G(E)} k \ar[d, equals] \\
\rR f_{K!}^\Delta \cS_{\bbm{1}} \ar[r, "{\alpha}"] &  \rR f_{K!}^\Delta \cS_{V|_{\wh{G}}} \ar[r, equals] &  (\rR f_{K!} \cS_V)_{\ol{x}} \ar[r, "{(\gamma_i)_{i \in I}}"] &  (\rR f_{K!} \cS_V)_{\ol{x}} \ar[r, equals] &  \rR f_{K!}^\Delta \cS_{V|_{\wh{G}}} \ar[r, "{\beta}"] & 
\rR f_{K!}^\Delta \cS_{\bbm{1}}  
\end{tikzcd}}
\end{equation}
(in the middle step we used \eqref{eq: cohomology Weil action} to obtain an action of $\prod_{i \in I} W_E$ on $(\rR f_{K!} \cS_V)_{\ol{x}}$) which defines an element
\[
\FS_G(S_{\cD})\in\End_{G(E)}(\cInd_K^{G(E)}k)\cong\sH(G,K).
\]
By \cite[Proposition I.9.3]{FS}, excursion operators commute with Hecke operators, so this element lies in $\mf Z(G,K)$. As $K$ varies, these elements are compatible under convolution, hence assemble to an element of $\mf{Z}(G)$. The map \eqref{eq: FS_G} sends $S_{\cD}$ to this element of $\mf{Z}(G)$.

\subsection{Excursion operators on Tate cohomology}\label{ssec: excursion on Tate} Recall the notion of \emph{Tate cohomology} from \S \ref{ssec: tate cohomology}. We will now study excursion operators on the Tate cohomology of $\Sht_{(G,1_G,I),K}$. We will distinguish two actions of excursion operators on Tate cohomology. The first is the native action, obtained by applying Tate cohomology to the usual excursion action. The second is the normed action, obtained by first passing from the excursion algebra to its Tate cohomology via the Tate diagonal. The comparison between these two actions is controlled by Frobenius linearization, and this will be the source of the Frobenius twists appearing later.

\subsubsection{Tate cohomology of smooth representations}\label{sssec: tate other settings} Let $\Sigma$ act on $G(E)$ through its action on $G$, so that $G(E)^\sigma=H(E)$. The formula \eqref{eq:tate-cohomology-of-complex} applies verbatim to $\cK \in D^b(\Rep^{\sm}_\Lambda G(E))^{B \Sigma} \cong D^b(\Rep^{\sm}_\Lambda(G(E)\rtimes\Sigma))$:
\begin{equation}\label{eq:tate-cohomology-G}
\rT^n(\cK)=\colim_{j\to\infty}\Hom_{D^+(\Lambda[\Sigma])}\bigl(\Lambda,\cK[n+2j]\bigr),
\end{equation}
where on the right $\cK$ denotes the underlying complex of $\Lambda[\Sigma]$-modules. Computing with the explicit resolutions of Example \ref{ex:degree-0-tate-construction}, each term of the resulting complex is a finite direct sum of terms of $\cK$, and the differentials and transition maps are built from $1-\sigma$ and $N$. These commute with $H(E)$, so \eqref{eq:tate-cohomology-G} is naturally a smooth representation of $H(E)$, and we obtain functors
\[
\rT^n \co D^b(\Rep^{\sm}_\Lambda G(E))^{B \Sigma} \rightarrow \Rep^{\sm}_{\Lambda} H(E).
\]
The same description shows that $\rT^n$ commutes with direct sums, and with filtered colimits, of complexes concentrated in a fixed finite range of degrees.

\subsubsection{Tate cohomology of moduli of shtukas} 
In this subsection we take $b=1_G$ and assume that $K\subset G(E)$ is $\Sigma$-invariant. The canonical $\Sigma$-equivariant structure on the trivial $G$-bundle induces a $\Sigma$-action on $\Sht_{(G,1_G,I),K}$.

The given action of $\Sigma$ on $G$ induces an action $V \mapsto {}^{\sigma}V$ of $\Sigma$ on $\Rep_k((\wh{G} \rtimes Q)^I)$. Suppose we have a $\Sigma$-equivariant representation $W \in \Rep_k( (\wh{G} \rtimes Q)^I)^{B\Sigma}$. Then $\cS_W$ has the structure of a $\Sigma$-equivariant sheaf on $\Sht_{(G,1_G,I),K}$. This equips its cohomology with a $\Sigma$-equivariant structure, so that we may regard, using \eqref{eq: cohomology Weil action}, 
\begin{equation}\label{eq:equivariant-sheaf-semidirect}
\rR f_{K!}\cS_W \in  D(\Rep^{\mrm{sm}}_k G(E))^{B(\prod_{i \in I}W_E \times \Sigma)}.
\end{equation}

\begin{remark}
This comes from the following more general principle: let $\Gamma$ and $\Gamma'$ be groups, and let \(\mathcal C\) be an \(\infty\)-category with a \(\Gamma\)-action, classified by a functor
\[
\cF \colon B  \Gamma \to \msf{Cat}_\infty.
\]
Suppose \(\cF\) has a \(\Gamma'\)-equivariant structure; more informally, \(\Gamma'\) acts on \(\Gamma\) and on \(\mathcal C\) in a compatible way. Equivalently, the \(\Gamma\)-action and the \(\Gamma'\)-action combine to an action of the semidirect product \(\Gamma \rtimes \Gamma'\) on \(\mathcal C\), i.e., a functor 
\[
\wt \cF \co B (\Gamma \rtimes \Gamma') \rightarrow  \msf{Cat}_\infty.
\]
Then there is a canonical equivalence
\[
(\mathcal C^{B\Gamma})^{B\Gamma'} \cong \limit_{B \Gamma'} (\limit_{B \Gamma} \cF) \cong  \limit_{B  (\Gamma \rtimes \Gamma')} \wt \cF.
\]
\end{remark}

Thanks to \eqref{eq:equivariant-sheaf-semidirect}, we can form the $j$th Tate cohomology of $\rR f_{K!}\cS_W$ in the sense of \eqref{eq:tate-cohomology-G} (for $j \in \Z/2\Z$), which we denote
\begin{equation}\label{eq: tate cohomology of shtukas}
\rT^j (\Sht_{(G,1_G,I),K}; W) := \rT^j(\rR f_{K!}\cS_W )  \in  (\Rep^{\mrm{sm}}_k H(E))^{B\prod_{i \in I}W_E},
\end{equation}
as we will want to record the dependence on $G,I$. 

\begin{example}\label{ex: tate cohomology of trivial shtuka}
For the trivial representation $W=\bbm{1}$, Example \ref{ex: cohomology of trivial shtuka} gives
\[
\rT^j(\Sht_{(G, 1_G, I),K}; \bbm{1}) \cong \rT^j(\cInd_K^{G(E)} k) \in \Rep^{\mrm{sm}}_k(H(E))
\]
with the trivial $\prod_{i \in I} W_E$-action. 
\end{example}

A similar story applies to $H$, with $b=1_H$ and $K\subset H(E)$. Note that $\Sigma$ acts trivially on $\Sht_{(H,1_H,I),K}$, so if $\Sigma$ also acts trivially on $W \in \Rep_k((\wh{H} \rtimes Q)^I)$, then by Example \ref{ex: tate coh of trivial} and Example \ref{ex: Tate cohomology of trivial coeff}, we have for each $j \in \Z/2\Z$ a natural isomorphism
\begin{equation}\label{eq: comparing cohomology and Tate cohomology}
\rT^j(\Sht_{(H,1_H,I),K}; W) \cong \bigoplus_{n \in \Z}  \rH^n(\rR f_{K!} \cS_W)  = \rH^*( \rR f_{K!} \cS_W).
\end{equation}

\subsubsection{Excursion action}\label{sssec: tate excursion action} The $\Sigma$ action on $G$ induces a $\Sigma$-action on $\Exc_k(W_E,  \wh{G})$ by transport of structure. Concretely, $\Sigma$ acts on excursion data by 
\[
\sigma \cdot (I, V, \alpha, \beta, (\gamma_i)_{i \in I}) = (I, {}^\sigma V, \sigma(\alpha), \sigma(\beta), (\gamma_i)_{i \in I})
\]
and then $\sigma \cdot S_{\cD} = S_{\sigma \cdot \cD} \in \Exc_k(W_E,  \wh{G})$. 

In general, given a $k[\Sigma]$-algebra $A$ and a $\Sigma$-equivariant $A$-module $M$ with compatible actions, there is a natural $\rT^0(A) = A^\sigma/(N \cdot A)$-module structure on $\rT^*(M)$. In particular, this equips $\rT^j (\Sht_{(G,1_G,I),K}; W)$ with a natural $\rT^0(\Exc_k(W_E, \wh{G}))$-action. We are most interested in the special case where $W = \bbm{1}$, in which case we have $\rT^j (\Sht_{(G,1_G,I),K}; \bbm{1}) = \rT^j(\cInd_K^{G(E)}k)$ at the level of underlying $H(E)$-representations (cf. Example \ref{ex: tate cohomology of trivial shtuka}). If an excursion datum $\cD = (I, V, \alpha, \beta, (\gamma_i)_{i \in I})$ is $\Sigma$-invariant, then $S_{\cD} \in \Exc_k(W_E, \wh{G})^{\sigma}$ and its action on $\rT^j(\cInd_K^{G(E)} k)$ can be described more concretely using \eqref{eq: excursion action}: it is given by composition
\begin{equation}
\begin{tikzcd}
\rT^j(\cInd_K^{G(E)} k)  \ar[d, equals, "\text{Ex. \ref{ex: tate cohomology of trivial shtuka}}"]  \ar[rrr, dashed, "S_{\cD}"]  & & &   \rT^j(\cInd_K^{G(E)} k)\\
\rT^j(\Sht_{(G,1_G,I),K}; \bbm{1}) \ar[r, "{\alpha}"] &  \rT^j(\Sht_{(G,1_G,I),K}; V)  \ar[r, "{(\gamma_i)_{i \in I}}"] &  \rT^j(\Sht_{(G,1_G,I),K};V) \ar[r, "{\beta}"] & \rT^j(\Sht_{(G,1_G,I),K}; \bbm{1}) \ar[u, equals] 
\end{tikzcd}
\end{equation}

\subsubsection{Interaction with Frobenius}\label{sssec: excursion Frobenius}

Equip $\Exc_k(W_E,\wh G)$ with the $\F_\ell$-structure transported from the $c$-group via the isomorphism $\alpha_G$ of \S \ref{sssec: shtuka Satake sheaves}. Let
\begin{equation}\label{eq: transported excursion Frobenius}
\Phi_G:=\alpha_G^{-1}\circ\Fr_\ell\circ\alpha_G\co\ld G\longrightarrow\ld G,
\qquad F_G:=\Phi_G^*\co\Exc_k(W_E,\wh G)\longrightarrow\Exc_k(W_E,\wh G),
\end{equation}
where $\Fr_\ell\co\lc G\rightarrow\lc G$ is Frobenius on $\chG$ and the identity on $W_E$. Thus $F_G$ is the linearized Frobenius for this transported $\F_\ell$-structure (\S \ref{sssec: frob twist algebra}). 

\begin{lemma}\label{lem:frobenius-twist-FS-parameter}
Let $\Pi$ be an irreducible smooth representation of $G(E)$, and let 
\[
\chi_{\Pi} \co \Exc_k(W_E, \wh{G}) \rightarrow k
\]
be the corresponding character. Then we have $\chi_{\Pi^{(\ell)}} = \chi_\Pi \circ F_G$. 
\end{lemma}

\begin{proof}
More generally, let $A_0$ be an $\F_\ell$-algebra and $A := A_0 \otimes_{\F_\ell} k$. Then the Frobenius on $A$, given by $\Frob_A(a) = a^\ell$, factors as the composition of the linearized Frobenius $F_A = (\Frob_\ell)_{A_0} \otimes \Id$ and $\varphi_A = \Id_{A_0} \otimes (\Frob_\ell)_k$.

For an $A$-module $M$, its Frobenius twist $M^{(\ell)} := M \otimes_{k, \Frob} k$ has the natural $A$-module structure $M^{(\ell)}\cong (\varphi_A^{-1})_*M$. Since $\Frob_A$ does not change supports, we have 
\[
\supp_A(M^{(\ell)}) = \supp_A((\varphi_A^{-1})_*M) = \supp_A((\Frob_A \circ \varphi_A^{-1})_*M)  = \supp_A(F_{A*} M)
\]
In particular, if $M \cong k$ corresponds to the character $\chi \co A \rightarrow k$, then $M^{(\ell)}$ corresponds to the character
\begin{equation}\label{eq: coefficient Frobenius character}
a \mapsto \Frob_\ell\bigl(\chi_M(\varphi_A^{-1}(a))\bigr) =\chi(F_A(a)).
\end{equation}
The proof is completed by applying this to $A = \Exc_k(W_E, \wh{G})$ and $\chi = \chi_{\Pi}$. 
\end{proof}

\begin{remark}\label{rem: pinned Frobenius correction}
In the usual pinned $L$-group coordinates, \eqref{eq: transported excursion Frobenius} is
\begin{equation}\label{eq: pinned Frobenius correction}
\Phi_G(g,w)=\bigl(z_G(w)\Fr_\ell(g),w\bigr),\qquad
z_G(w)=(2\rho_{\wh G}^{\vee})\left(\frac{\chi_{\mrm{cyc}}^{1/2}(w)^\ell}{\chi_{\mrm{cyc}}^{1/2}(w)}\right).
\end{equation}
Here $\Fr_\ell$ on $\wh G$ uses its pinned $\F_\ell$-structure, $\rho_{\wh G}^{\vee}$ is the half-sum of positive coroots, and $z_G$ is an unramified cocycle valued in $Z(\wh G_{\der})(k)^{W_E}$. Thus Lemma \ref{lem:frobenius-twist-FS-parameter} says $\rho_{\Pi^{(\ell)}}\sim z_G\cdot\Fr_\ell\circ\rho_\Pi$ in these coordinates.
\end{remark}

By construction, the Brauer functor $\wcbr^I$ is defined with $\F_\ell$-coefficients. Consequently $\ld\psi^*$ respects the $\F_\ell$-structures transported from the $c$-groups. Defining $F_H$ as in \eqref{eq: transported excursion Frobenius}, we obtain a commutative diagram 
\begin{equation}\label{eq: sigma dual Frobenius compatibility}
\begin{tikzcd}
\Exc_k(W_E, \wh{G}) \ar[r, "F_G"] \ar[d, "\ld\psi^*"] & \Exc_k(W_E, \wh{G}) \ar[d, "\ld\psi^*"] \\
\Exc_k(W_E, \wh{H})  \ar[r, "F_H"] &  \Exc_k(W_E, \wh{H}) 
\end{tikzcd}
\end{equation}

\subsubsection{Normed excursion action}\label{sssec: normed excursion action} 
Recall from Definition \ref{def: tate diagonal} that for any commutative $k[\Sigma]$-algebra $A$, there is the \emph{Tate diagonal} morphism 
\[
A \xrightarrow{ \Nm} \rT^0(A)
\]
sending $a$ to the class of $\Nm(a) = a \cdot \sigma(a)  \cdots  \sigma^{\ell-1}(a)$. This is $\Frob$-semilinear, but an $\F_\ell[\Sigma]$-structure on $A$ induces a linearization $\Nm \iellt \co A \rightarrow \rT^0(A)$, which is a $k$-algebra homomorphism. 

We apply this to $A := \Exc_k(W_E, \wh{G})$, with the transported $\F_\ell$-structure of \S \ref{sssec: excursion Frobenius}. In \S \ref{sssec: tate excursion action} we saw an action of $\rT^0(\Exc_k(W_E, \wh{G}))$ on $\rT^j(\cInd_K^{G(E)} k)$. Inflating this action through $\Nm \iellt$ gives an action of $\Exc_k(W_E, \wh{G})$ on $\rT^j(\cInd_K^{G(E)} k)$, which we call the \emph{normed excursion action}.

\subsubsection{Native excursion action}\label{sssec: native excursion action} 
The $\Sigma$-invariant subalgebra $\Exc_k(W_E, \wh{G})^\sigma$ acts naturally on $\rT^j(\Sht_{(G,1_G,I),K}; V)$ via the quotient $\Exc_k(W_E, \wh{G})^\sigma \surj \rT^0(\Exc_k(W_E, \wh{G}))$ and then the mechanism of \S \ref{sssec: tate excursion action}. We refer to this as the \emph{native excursion action} of $\Exc_k(W_E, \wh{G})^\sigma$ on $\rT^j(\Sht_{(G,1_G,I),K}; V)$. This action is \emph{not} the same as the restriction of the normed excursion action to the subalgebra $\Exc_k(W_E, \wh{G})^\sigma \subset \Exc_k(W_E, \wh{G})$; this discrepancy will be responsible for Frobenius twists which show up later. 


In general, for a commutative $k$-algebra $A$ the absolute Frobenius $A \rightarrow A$ is $\Frob$-semilinear. If $A$ has an $\F_\ell$-structure $A \cong A_0 \otimes_{\F_\ell} k$, then the absolute Frobenius may be linearized to a map $F \co A \rightarrow A$, as explained in \S \ref{sssec: frob twist algebra}. The map $F$ is characterized as the unique $k$-linear homomorphism that sends $a_0 \mapsto a_0^\ell$ for all $a_0 \in A_0 \subset A$. 

The preceding discussion applies to $A := \Exc_k(W_E, \wh{G})^\sigma$ (with the $\F_\ell$-structure of \S \ref{sssec: excursion Frobenius}). In these terms, the normed excursion action of $\Exc_k(W_E, \wh{G})^\sigma$ is the native excursion action composed with the map $F \co  \Exc_k(W_E, \wh{G})^\sigma \rightarrow \Exc_k(W_E, \wh{G})^\sigma$ which is the linearization of the absolute Frobenius.

\subsubsection{Norm of excursion data}\label{sssec: excursion data} For a finite quotient $Q$ of $W_E$ through which both the action on $\wh{G}$ and $\chi_{\mrm{cyc}}^{1/2}$ factor, we define a functor 
\[
\Nm \co \Rep_k((\wh{G} \rtimes Q)^I) \rightarrow \Rep_k((\wh{G} \rtimes Q)^I)^{B\Sigma}
\]
as follows:
\begin{itemize}
\item For a representation $V \in \Rep_k((\wh{G} \rtimes Q)^I)$, we set 
\[
\Nm(V) := V \otimes_k \left( {}^{\sigma} V \right) \otimes_k \ldots \otimes_k \left({}^{\sigma^{\ell-1}}V \right) \in \Rep_k((\wh{G} \rtimes Q)^I)^{B\Sigma}
\]
with the obvious $\Sigma$-equivariant structure. Under Geometric Satake this corresponds to the norm \eqref{eq: modular convolution norm}. 
\item Given $h \co  V \rightarrow V' \in \Rep_k((\wh{G} \rtimes Q)^I)$, we set 
\[
\Nm(h) := h \otimes {}^{\sigma} h \otimes \ldots \otimes {}^{\sigma^{\ell-1}} h  \co \Nm(V) \rightarrow \Nm(V').
\]
\end{itemize}
Note that $\Nm$ is \emph{not} an additive functor, nor is it even $k$-linear.

\begin{defn}
The \emph{linearized norm} is $\Nm\iellt(V):=\Nm(V\iellt)$, where $(-)\iellt$ is the inverse coefficient twist on $c$-group representations, transported to $L$-group representations via $\alpha_G$. Thus $\cS_{V\iellt}\cong(\cS_V)\iellt$, with the right-hand side defined by \eqref{eq: coefficient Frobenius sheaves}. It satisfies $\Nm\iellt(ah)=a\,\Nm\iellt(h)$ for $a\in k$ and morphisms $h$, but it is still not additive; compare Proposition \ref{prop: modular Brauer extraction}.
\end{defn}

\begin{defn}For $V \in \Rep_k((\wh{G} \rtimes Q)^I)$, we denote by 
\[
N \cdot V \in \Rep_k((\wh{G} \rtimes Q)^I)^{B\Sigma}
\]
the $\sigma$-equivariant representation $V \oplus {}^{\sigma} V \oplus \ldots  \oplus {}^{\sigma^{\ell-1}}V$, with the obvious $\Sigma$-equivariant structure. 

For $h \co  V \rightarrow V' \in \Rep_k((\wh{G} \rtimes Q)^I)$, we denote by 
\[
N \cdot h \co N \cdot V \rightarrow N \cdot V'
\]
the $\sigma$-equivariant map $h \oplus {}^{\sigma} h \oplus \ldots \oplus {}^{\sigma^{\ell-1}} h$. Let $\Delta_\ell \co \bbm{1} \rightarrow \bbm{1}^{\oplus \ell}$ denote the diagonal map and $\nabla_\ell \co \bbm{1}^{\oplus \ell} \rightarrow \bbm{1}$ denote the sum map. 
\end{defn}

\begin{defn}
Let $\cD = (I,V, \alpha, \beta, (\gamma_i)_{i \in I})$ be an excursion datum for $\wh{G}$. Define the excursion data 
\[
\Nm \iellt \cD := (I, \Nm \iellt V, \Nm \iellt \alpha, \Nm \iellt \beta, (\gamma_i)_{i \in I})
\]
and
\[
N \cdot \cD := (I, N \cdot V, (N \cdot \alpha) \circ \Delta_\ell, \nabla_\ell \circ (N \cdot \beta), (\gamma_i)_{i \in I}),
\]
which are excursion data for $\wh{G}$. 

\end{defn}

Straightforward calculations yield: 

\begin{lemma}[{\cite[Lemma 5.16]{Fe23}}]\label{lem: norm extension} For all excursion data $\cD$, we have 
\[
\Nm \iellt(S_{\cD}) = [S_{\Nm \iellt (\cD)}]\in \rT^0\bigl(\Exc_k(W_E,\wh{G})\bigr)
\]
and
\[
N \cdot S_{\cD} = S_{N \cdot \cD} \in \Exc_k(W_E, \wh{G})^\sigma. 
\]
\end{lemma}

\begin{proof}The second identity appears in {\cite[Lemma 5.16]{Fe23}}, which also proves that
\[
\Nm (S_{\cD}) = [S_{\Nm (\cD)}]\in \rT^0\bigl(\Exc_k(W_E,\wh{G})\bigr)
\]
from which the first identity follows. 
\end{proof}

\subsection{Equivariant localization for local shtukas}\label{ssec: shtuka localization}
We use the square root of the cyclotomic character fixed in \S \ref{ssec: shtukas}, for both $G$ and $H$, to transport $\lc\psi$ to $\ld\psi$. We have a commutative diagram 
\begin{equation}\label{eq: shtuka leg maps}
\begin{tikzcd}[ampersand replacement=\&]
\Sht_{(H,1_H,I),\iota^*K} \ar[r,"\iota"] \ar[d,"f^H_{\iota^*K}"'] \& \Sht_{(G,1_G,I),K} \ar[d,"f^G_K"] \\
\prod_{i\in I}\Spd\brE \ar[r,equals] \& \prod_{i\in I}\Spd\brE
\end{tikzcd}
\end{equation}
of shtukas, where $\iota$ is the embedding of Proposition \ref{prop: fixed points of sht}. The period maps fit into the commutative diagram
\begin{equation}\label{eq: shtuka period maps}
\begin{tikzcd}[ampersand replacement=\&]
\Sht_{(H,1_H,I),\iota^*K} \ar[r,"\iota"] \ar[d,"\pi^H_{\iota^*K}"'] \& \Sht_{(G,1_G,I),K} \ar[d,"\pi^G_K"] \\
\Gr^{\twi}_{H,\prod_{i\in I}\Spd\brE} \ar[r,"\iota"] \& \Gr^{\twi}_{G,\prod_{i\in I}\Spd\brE}
\end{tikzcd}
\end{equation}
All sheaves have coefficients in $k$.

\begin{prop}\label{prop: equiv localization for shtukas}
Let $I$ be a nonempty finite set, $K\subset G(E)$ a $\Sigma$-stable compact open subgroup of prime-to-$\ell$ pro-order, and $W\in\Rep_k((\wh G\rtimes Q)^I)^{B\Sigma}$. There is a natural $H(E)$-equivariant isomorphism
\begin{equation}\label{eq: global realized localization}
\bigl((\rR f^G_{K!}\cS_W)|_{H(E)}\bigr)^{t\Sigma}\xrightarrow{\sim}\rR f^H_{\iota^*K!}(\pi^H_{\iota^*K})^*\bigl((\iota^*\cS_W)^{t\Sigma}\bigr),
\end{equation}
compatible with base change in $\prod_{i\in I}\Spd\brE$, with change of level, and with the fusion isomorphisms \eqref{eq: fusion}.
\end{prop}

\begin{proof}
Choose a $\Sigma$-stable closed subspace $Z\subset\Gr^{\twi}_{G,\prod_{i\in I}\Spd\brE}$, a finite union of Schubert varieties, containing the support of $\cS_W$, and put
\[
\Sht^G_Z:=(\pi^G_K)^{-1}(Z)\subset\Sht_{(G,1_G,I),K},\qquad\Sht^H_Z:=(\pi^H_{\iota^*K})^{-1}(\iota^{-1}Z)\subset\Sht_{(H,1_H,I),\iota^*K}.
\]
The restriction of the leg map $f^G_K$ to $\Sht^G_Z$ is separated and shriekable with finite $\dimg$, and $\rR f^G_{K!}\cS_W$ is computed on $\Sht^G_Z$. By Proposition \ref{prop: fixed points of sht}, $\iota$ identifies $\Sht_{(H,1_H,I),\iota^*K}$ with the $\Sigma$-fixed locus of $\Sht_{(G,1_G,I),K}$, so by the commutative square \eqref{eq: shtuka period maps} it identifies $\Sht^H_Z$ with the $\Sigma$-fixed locus of $\Sht^G_Z$. By the same commutative square again, the restriction of $\cS_W=(\pi^G_K)^*\cS_W$ to this fixed locus is $(\pi^H_{\iota^*K})^*\iota^*\cS_W$. Now the desired comparison \eqref{eq: global realized localization} follows from Proposition \ref{prop: realized Smith localization} (equivariant localization for Tate cohomology), together with Example \ref{ex: Tate commutation instances}(1) for $(\pi^H_{\iota^*K})^*$ and Remark \ref{rem: equivariance of localization} for the $H(E)$-action. The compatibilities follow from base change for $\rR f_!$.
\end{proof}

The isomorphism \eqref{eq: global realized localization} is only asserted to be $H(E)$-equivariant. We also want to compare the $W_E^I$-actions. The LHS has a natural $W_E^I$-action, thanks to \eqref{eq: cohomology Weil action}. By transport, this induces a $W_E^I$-action on the RHS which we call the \emph{transported action}.

The RHS of \eqref{eq: global realized localization} is the cohomology of $\Sht_{(H,1_H,I),\iota^*K}$ with coefficients $(\iota^*\cS_W)^{t\Sigma}$. This is not a Satake sheaf, so \eqref{eq: cohomology Weil action} does not give it an a priori action of $W_E^I$. However, if $W=\Nm\iellt V$ for some $V\in\Rep_k((\wh G\rtimes Q)^I)$, then we will equip said RHS with a ``perverse filtration'' whose associated graded has a natural $W_E^I$-action, which we will call the \emph{native action}. On $\Sht_{(H,1_H,I),\iota^*K}$ the coefficient complex $(\iota^*\cS_W)^{t\Sigma}$ agrees with $A_I(\cS_V)$.\footnote{The shift $[\delta]$ vanishes (see the proof of Proposition \ref{prop: filtered Weil compatibility}).} By Proposition \ref{prop: Brauer perverse coefficients} and Theorem \ref{thm: relative functoriality}(1), 
\[
{}^p\rH^b(A_I(\cS_V)) \cong \wcbr^I(\cS_V)\cong\cS_{\ld\psi^*V} \quad \text{ for all } b \in \Z.
\]
So the perverse truncations of $A_I(\cS_V)$ induce a filtration of the right side of \eqref{eq: global realized localization} whose graded pieces $\rR f^H_{\iota^*K!}\cS_{\ld\psi^*V}[-b]$ carry an action of $W_E^I$ from \eqref{eq: cohomology Weil action} for $H$. 

The following Proposition establishes the compatibility of these two $W_E^I$-actions.

\begin{prop}\label{prop: filtered Weil compatibility}
Let $K$ be as in Proposition \ref{prop: equiv localization for shtukas}, let $V\in\Rep_k((\wh G\rtimes Q)^I)$, and let $W:=\Nm\iellt V$. The filtration of the RHS of \eqref{eq: global realized localization} induced by the perverse truncations ${}^p\tau^{\leq b}A_I(\cS_V)$ is stable under the transported action of $W_E^I$, and on its graded pieces the transported action agrees with the native action on $\rR f^H_{\iota^*K!}\cS_{\ld\psi^*V}[-b]$. The filtration is functorial in $V$ and compatible with base change and with coalescence of legs.
\end{prop}

The rest of this subsection proves Proposition \ref{prop: filtered Weil compatibility}. We use the following abstract observation.

\begin{lemma}\label{lem: filtered descent}
Let $F\co\cC\rightarrow\cD$ and $J\co\cE\rightarrow\cD$ be exact, colimit-preserving functors of stable $\infty$-categories, with $J$ fully faithful and $\cE$ admitting sequential colimits. Suppose $\cC$ carries a t-structure such that $A=\colim_n\tau^{\leq n}A$ for every bounded-below $A \in \cC$.\footnote{This holds, for example, if the t-structure on $\cC$ is right-complete.} If $F(A)$ and $F(\rH^nA)$ lie in the essential image of $J$ for all $n \in \Z$, then so do the $F(\tau^{\leq b}A)$ for all $b \in \Z$, and the maps 
\[
F(\tau^{\leq b}A)\rightarrow F(\tau^{\leq b+1}A), \quad F(\tau^{\leq b}A)\rightarrow F(A),  \quad F(\tau^{\leq b}A)\rightarrow F(\rH^bA[-b])
\]
descend uniquely through $J$.
\end{lemma}

\begin{proof}
Each canonical truncation $\tau^{\leq n} \tau^{\geq b} A$ is a finite extension of shifts of the $\rH^r(A)$, so $F(\tau^{\leq n} \tau^{\geq b} A)$ lies in the essential image of $J$, which is a stable subcategory closed under colimits. Hence $F(\tau^{\geq b}A)=\colim_n F(\tau^{\leq n}\tau^{\geq b}A)$ lies in the essential image of $J$ for every $b$, and therefore so does $F(\tau^{\leq b}A)=\fib(F(A)\rightarrow F(\tau^{\geq b+1}A))$. The assertions about the maps follow from full faithfulness of $J$.
\end{proof}

To apply Lemma \ref{lem: filtered descent}, we use the description of cohomology of local shtukas in terms of Hecke operators on $\Bun_G$, from the proof of \cite[Proposition IX.3.2]{FS}:
\begin{equation}\label{eq: shtuka Hecke identification}
\rR f^G_{K!}\cS_W\cong\left.j_G^*T_W\bigl(j_{G!}\cInd_K^{G(E)}k\bigr)\right|_{\prod_{i\in I}\Spd\brE},
\end{equation}
where $j_G\co\Bun_G^1\cong[*/G(E)]\inj\Bun_G$ is the inclusion of the stratum of the trivial bundle, so that $j_{G!}\cInd_K^{G(E)}k\in D_{\et}(\Bun_G;k)$ via the equivalence $D(\Rep^{\sm}_kG(E))\cong D_{\et}([*/G(E)];k)$ \cite[Theorem V.1.1]{FS}, and $T_W$ is the Hecke operator of \cite[\S IX.2]{FS}. In terms of the Hecke correspondence
\begin{equation}\label{eq: global Hecke correspondence}
\begin{tikzcd}[ampersand replacement=\&]
\Bun_G \& \Hck_G^I \ar[l,"h_1"'] \ar[r,"h_2"] \& \Bun_G\times(\Div_X^1)^I,
\end{tikzcd}
\end{equation}
it is the functor
\begin{align*}
T_W  \co D_{\et}(\Bun_G;k) & \rightarrow D_{\et}(\Bun_G\times(\Div_X^1)^I;k), \\
A  &  \mapsto \rR h_{2!}\bigl(h_1^*A\dotimes_k\cS_W|_{\Hck_G^I}\bigr),
\end{align*}
where $\cS_W|_{\Hck_G^I}$ is the pullback of $\cS_W$ from the local Hecke stack; $h_2$ is proper on its support. By \cite[Corollary IX.2.3]{FS}, $T_W(A)$ is the pullback of an object of $D_{\et}(\Bun_G;k)^{BW_E^I}$. For $A :=j_{G!}(\cInd_K^{G(E)}k)$, restriction to $\Bun_G^1$ gives the $W_E^I$-action of \eqref{eq: cohomology Weil action}. The analogous statements hold for $H$.

Let
\begin{align*}
F\co D_{\et}(\cHck_{H,(\Div_X^1)^I};k)^{\bd} & \rightarrow D_{\et}([*/H(E)]\times(\Div_X^1)^I;k), \\
B & \mapsto  j_H^*\rR h_{2!}\bigl(h_1^*j_{H!}(\cInd_{\iota^*K}^{H(E)}k)\dotimes_kB|_{\Hck_H^I}\bigr),
\end{align*}
where $h_1$ and $h_2$ are analogous maps to \eqref{eq: global Hecke correspondence} but for $H$. Here $j_H^*$ is restriction to $\Bun_H^1\times(\Div_X^1)^I\cong[*/H(E)]\times(\Div_X^1)^I$. Thus $F(\cS_U)=j_H^*T_U(j_{H!}\cInd_{\iota^*K}^{H(E)}k)$, and by \eqref{eq: shtuka Hecke identification} its pullback to $[*/H(E)]\times\prod_{i\in I}\Spd\brE$ is $\rR f^H_{\iota^*K!}\cS_U$, with the $W_E^I$-action of \eqref{eq: cohomology Weil action}. We give the domain of $F$ the relative perverse t-structure; the colimit condition of Lemma \ref{lem: filtered descent} holds, as one checks after applying normalized constant terms, which on bounded support are conservative, t-exact and colimit-preserving. Let $J$ be the pullback functor
\begin{equation}\label{eq: Weil descent category}
J\co D(\Rep_k^{\sm}H(E))^{BW_E^I}\rightarrow D_{\et}([*/H(E)]\times(\Div_X^1)^I;k),
\end{equation}
where the source is identified with $D_{\et}([*/H(E)]\times[*/W_E^I];k)$ by \cite[Theorem V.1.1]{FS}. It is exact, colimit-preserving, and fully faithful by \cite[Corollary IV.7.2]{FS}.

Let $V\in\Rep_k((\wh G\rtimes Q)^I)$. We apply Proposition \ref{prop: realized Smith localization} to the Hecke correspondence defining \eqref{eq: shtuka Hecke identification} for $W=\Nm\iellt V$, with both bundles trivial, keeping the framing of the target bundle and dividing the framing of the source by $K$. The $\Sigma$-fixed points of the Hecke correspondence for $G$ at infinite level agrees with the Hecke correspondence for $H$, since a $\Sigma$-fixed modification of trivial $G$-bundles is the same as a modification of trivial $H$-bundles (as in the proof of Proposition \ref{prop: fixed points of sht}). Then Lemma \ref{lem: coprime level fixed descent} implies that the analogous statement holds even at a finite level $K$ of prime-to-$\ell$ pro-order. 
Hence we obtain
\begin{align}
F(A_I(\cS_V))&\cong\Bigl(\bigl(j_G^*T_{\Nm\iellt V}(j_{G!}\cInd_K^{G(E)}k)\bigr)\big|_{[*/H(E)]\times(\Div_X^1)^I}\Bigr)^{t\Sigma}\label{eq: F of A_I localization}\\
&\cong J\Bigl(\bigl((\rR f^G_{K!}\cS_{\Nm\iellt V})|_{H(E)}\bigr)^{t\Sigma}\Bigr),\label{eq: F of A_I descent}
\end{align}
where \eqref{eq: F of A_I localization} is Proposition \ref{prop: realized Smith localization} together with Lemma \ref{lem: bounded Tate commutation} for $\rR h_{2!}$, and \eqref{eq: F of A_I descent} is Lemma \ref{lem: bounded Tate commutation} for $J$, with the transported $W_E^I$-action on the right. In particular, $F(A_I(\cS_V))$ lies in the essential image of $J$.

On the other hand, by Proposition \ref{prop: Brauer perverse coefficients} every ${}^p\rH^nA_I(\cS_V)$ is isomorphic to the Satake object $\wcbr^I(\cS_V)$, so $F({}^p\rH^nA_I(\cS_V))$ also lies in the image of $J$. Lemma \ref{lem: filtered descent} therefore descends the objects $F({}^p\tau^{\leq b}A_I(\cS_V))$, with their transition maps, to objects $F_b(V)$ of $D(\Rep_k^{\sm}H(E))^{BW_E^I}$. They form a $W_E$-equivariant filtration
\begin{equation}\label{eq: descended perverse filtration}
F_b(V)\rightarrow\bigl((\rR f^G_{K!}\cS_{\Nm\iellt V})|_{H(E)}\bigr)^{t\Sigma},\qquad b\in\Z,
\end{equation}
and hence, through \eqref{eq: global realized localization}, a filtration of $\rR f^H_{\iota^*K!}(\pi^H_{\iota^*K})^*A_I(\cS_V)$. By Theorem \ref{thm: relative functoriality}(1), its graded pieces are
\begin{equation}\label{eq: excursion graded coefficient}
\operatorname{gr}_bF_\bullet(V)\cong\rR f^H_{\iota^*K!}\cS_{\ld\psi^*V}[-b],
\end{equation}
with their native $W_E^I$-actions. By naturality of localization and t-exactness of pullback in the base, the filtration is functorial in $V$ and compatible with base change and with coalescence of legs. This completes the proof of Proposition \ref{prop: filtered Weil compatibility}. \qed

\subsection{Functoriality for excursion operators}\label{ssec: functoriality for excursion operators}

Let $K$ be $\Sigma$-stable of prime-to-$\ell$ pro-order. Taking $I=\{0\}$ and $W=\bbm1$ in \eqref{eq: global realized localization} and passing to degree-zero cohomology gives the first isomorphism in
\begin{equation}\label{eq: equiv local trivial coeff}
\rT^0(\Sht_{(G,1_G,\{0\}),K};\bbm1)
\cong\rT^0(\Sht_{(H,1_H,\{0\}),\iota^*K};\bbm1)
\cong\cInd_{\iota^*K}^{H(E)}k,
\end{equation}
and the second isomorphism is \eqref{eq: comparing cohomology and Tate cohomology} combined with Example \ref{ex: cohomology of trivial shtuka}. 

\begin{thm}\label{thm: localize excursion operator} Let $K < G(E)$ be a compact open subgroup with prime-to-$\ell$ pro-order, stable under $\Sigma$. Then the isomorphism \eqref{eq: equiv local trivial coeff} is equivariant for the action of $ \Exc_k(W_E, \wh{G}) $, acting on the LHS via the normed excursion action (\S \ref{sssec: normed excursion action}) and on the RHS through the homomorphism $\ld \psi^* \co \Exc_k(W_E, \wh{G}) \rightarrow \Exc_k(W_E, \wh{H})$ from \eqref{eq: excursion algebra functoriality} followed by the native excursion action (\S \ref{sssec: native excursion action}).
\end{thm}

\begin{proof}
It suffices to treat an excursion operator of the form $S_{\cD}$, with $\cD=(I,V,\alpha,\beta,(\gamma_i)_i)$, since these generate the excursion algebra. Let $\Delta\co\Spd\brE\rightarrow\prod_{i\in I}\Spd\brE$ be the diagonal. By Lemma \ref{lem: norm extension}, the normed action of $S_{\cD}$ is $\rT^0$ of the composite of $\Sigma$-equivariant maps
\begin{equation}\label{eq: normed excursion composite}
\begin{tikzcd}[ampersand replacement=\&, column sep=2.8em]
\rR f^G_{K!}\cS_{\bbm1} \ar[r,"\Nm\iellt\alpha"] \& \Delta^*\rR f^G_{K!}\cS_{\Nm\iellt V} \ar[r,"(\gamma_i)_i"] \& \Delta^*\rR f^G_{K!}\cS_{\Nm\iellt V} \ar[r,"\Nm\iellt\beta"] \& \rR f^G_{K!}\cS_{\bbm1},
\end{tikzcd}
\end{equation}
where the outer maps are the creation and annihilation maps, induced by $\Nm\iellt\alpha$ and $\Nm\iellt\beta$ through the fusion isomorphism \eqref{eq: fusion}, and the middle map is the $W_E^I$-action.

Apply the Tate construction to \eqref{eq: normed excursion composite}, restricting the $G(E)$-action to $H(E)$. The equivariant localization isomorphism \eqref{eq: global realized localization} gives a commutative diagram
\begin{equation}\label{eq: localized excursion composite}
\begin{tikzcd}[ampersand replacement=\&, column sep=3em]
\bigl((\rR f^G_{K!}\cS_{\bbm1})|_{H(E)}\bigr)^{t\Sigma} \ar[r,"\sim"] \ar[d,"\Nm\iellt\alpha"'] \& \rR f^H_{\iota^*K!}(\pi^H_{\iota^*K})^*\bigl((\iota^*\cS_{\bbm1})^{t\Sigma}\bigr) \ar[d, "(\iota^*\Nm\iellt\alpha)^{t\Sigma}"] \\
\Delta^*\bigl((\rR f^G_{K!}\cS_{\Nm\iellt V})|_{H(E)}\bigr)^{t\Sigma} \ar[r,"\sim"] \ar[d,"(\gamma_i)_i"'] \& \Delta^*\rR f^H_{\iota^*K!}(\pi^H_{\iota^*K})^*A_I(\cS_V) \ar[d,"(\gamma_i)_i"] \\
\Delta^*\bigl((\rR f^G_{K!}\cS_{\Nm\iellt V})|_{H(E)}\bigr)^{t\Sigma} \ar[r,"\sim"] \ar[d,"\Nm\iellt\beta"'] \& \Delta^*\rR f^H_{\iota^*K!}(\pi^H_{\iota^*K})^*A_I(\cS_V) \ar[d, "(\iota^*\Nm\iellt\beta)^{t\Sigma}"] \\
\bigl((\rR f^G_{K!}\cS_{\bbm1})|_{H(E)}\bigr)^{t\Sigma} \ar[r,"\sim"] \& \rR f^H_{\iota^*K!}(\pi^H_{\iota^*K})^*\bigl((\iota^*\cS_{\bbm1})^{t\Sigma}\bigr)
\end{tikzcd}
\end{equation}
The left column is the Tate construction of \eqref{eq: normed excursion composite}, so its composite induces the normed action of $S_{\cD}$ on $\rH^0$. 

Next we filter the right column by the perverse filtration \eqref{eq: descended perverse filtration}. By Proposition \ref{prop: filtered Weil compatibility}, the middle map preserves the filtration. The outer maps also preserve it: they are induced by $\alpha$ and $\beta$ after pullback along $\Delta$, and the norm, the Tate construction and perverse truncation commute with this pullback (the last by \cite[Proposition VI.7.4]{FS}). Passing to the associated graded, we obtain for each $b$ a commutative diagram
\begin{equation}\label{eq: graded excursion composite}
\begin{tikzcd}[ampersand replacement=\&, column sep=3em]
F_b(\bbm1) \ar[r] \ar[d] \& \rR f^H_{\iota^*K!}\cS_{\bbm1}[-b] \ar[d,"\ld\psi^*\alpha"] \\
\Delta^*F_b(V) \ar[r] \ar[d,"(\gamma_i)_i"'] \& \Delta^*\rR f^H_{\iota^*K!}\cS_{\ld\psi^*V}[-b] \ar[d,"(\gamma_i)_i"] \\
\Delta^*F_b(V) \ar[r] \ar[d] \& \Delta^*\rR f^H_{\iota^*K!}\cS_{\ld\psi^*V}[-b] \ar[d,"\ld\psi^*\beta"] \\
F_b(\bbm1) \ar[r] \& \rR f^H_{\iota^*K!}\cS_{\bbm1}[-b]
\end{tikzcd}
\end{equation}
The left column is the filtered piece of the right column of \eqref{eq: localized excursion composite}, and the horizontal maps are the projections to the graded pieces, identified with the cohomology of Satake sheaves for $H$ by \eqref{eq: excursion graded coefficient}. The right column is computed as follows: the outer maps are $\ld\psi^*\alpha$ and $\ld\psi^*\beta$ by the naturality of \eqref{eq: Brauer perverse coefficients} and the compatibility of $\wcbr^I$ with fusion (Theorem \ref{thm: wtcbr tannakian functor}), and the middle map is the native action of $W_E^I$ by Proposition \ref{prop: filtered Weil compatibility}. Thus the right column is the native excursion operator of $\ld\psi^*\cD$, shifted by $[-b]$.

It remains to pass from the graded pieces to $\rT^0$. By Proposition \ref{prop: equiv localization for shtukas} with $W=\bbm1$, Example \ref{ex: cohomology of trivial shtuka}, \eqref{eq: realized trivial coefficients} and the projection formula, we have an isomorphism of (derived) smooth $H(E)$-representations
\begin{equation}\label{eq: unit endpoint Tate}
\bigl((\cInd_K^{G(E)}k)|_{H(E)}\bigr)^{t\Sigma}\cong \cInd_{\iota^*K}^{H(E)}k\dotimes_k k^{t\Sigma}.
\end{equation}
 Since $A_{\{0\}}(\cS_{\bbm1})$ is supported on the unit section, which has relative dimension $0$, perverse truncation there is ordinary truncation \cite[Definition/Proposition VI.7.1]{FS}, which gives the isomorphism of (derived) smooth $H(E)$-representations
\begin{equation}\label{eq: unit endpoint filtration}
F_0(\bbm1)\cong \cInd_{\iota^*K}^{H(E)}k\dotimes_k\tau^{\leq0}k^{t\Sigma}.
\end{equation}
The natural maps
\begin{equation}\label{eq: unit graded edge}
\begin{tikzcd}[ampersand replacement=\&]
\rH^0(\cInd_{\iota^*K}^{H(E)}k\dotimes_k k^{t\Sigma}) \& \rH^0F_0(\bbm1) \ar[l,"\sim"'] \ar[r,"\sim"] \& \rH^0\operatorname{gr}_0F_\bullet(\bbm1)=\cInd_{\iota^*K}^{H(E)}k
\end{tikzcd}
\end{equation}
are isomorphisms: the first because $\tau^{\leq0}k^{t\Sigma}\rightarrow k^{t\Sigma}$ is an isomorphism on $\rH^0$, and the second because $\rH^0$ and $\rH^1$ of $F_{-1}(\bbm1)=\cInd_{\iota^*K}^{H(E)}k\dotimes_k\tau^{\leq-1}k^{t\Sigma}$ vanish. Both maps intertwine the endomorphisms induced by the composites in \eqref{eq: localized excursion composite} and \eqref{eq: graded excursion composite} for $b=0$. 
\begin{itemize}
\item On the left, this endomorphism is the normed action of $S_{\cD}$ on $\rT^0(\Sht_{(G,1_G,\{0\}),K};\bbm1)=\rH^0(\cInd_{\iota^*K}^{H(E)}k\dotimes_k k^{t\Sigma})$.
\item On the right, it is $\rH^0$ of \eqref{eq: graded excursion composite} for $b=0$, which is the native action of $\ld\psi^*S_{\cD}$ by Lemma \ref{lem: excursion pullback} and Remark \ref{rem: excursion continuous Weil}. 
\end{itemize} 
The resulting identification of $\rH^0(\cInd_{\iota^*K}^{H(E)}k\dotimes_k k^{t\Sigma})$ with $\cInd_{\iota^*K}^{H(E)}k$ is \eqref{eq: equiv local trivial coeff}. This completes the proof.
\end{proof}

\section{Derived Treumann--Venkatesh Conjecture}\label{sec: TV conjecture}

In this section we will reap the applications of the preceding material, especially the functoriality relations for excursion operators from Theorem \ref{thm: localize excursion operator}. In \S \ref{ssec: derived TV} we formulate the ``derived'' version of Treumann-Venkatesh's Conjecture \ref{conj: TV}(2), and prove it in \S \ref{ssec: TV conjectures}. In \S \ref{sssec: double coset fixed point} we will construct the Treumann-Venkatesh homomorphism alluded to in \eqref{eq: intro TV homomorphism}, and in \S \ref{ssec: BC homomorphism for BC} we establish the commutative square \eqref{eq: intro bernstein center functoriality}. Finally, in \S \ref{ssec: functorial lift} we prove the existence of functorial lifts along $\sigma$-dual homomorphisms, including Theorem \ref{thm: intro functorial lift}.

\subsection{Formulation}\label{ssec: derived TV} By the realization of the Bernstein center $\mf{Z}(G)$ as the ring of endomorphisms of the identity functor on $\Rep^{\sm}_k G(E)$, there is a tautological action of $\mf{Z}(G)$ on any $\Pi \in D^b(\Rep^{\sm}_k G(E))$. Composing this with the Fargues-Scholze map $\Exc_k(W_E, \wh{G}) \xrightarrow{\FS_G} \mf{Z}(G;k)$ discussed in \S \ref{ssec: review FS}, we obtain an action of $\Exc_k(W_E, \wh{G}) $ on any such $\Pi$. In particular, to each $\Pi$ we may attach a subset 
\[
\supp_{\Exc_k(W_E, \wh{G})}(\rH^*(\Pi)) \subset \Spec \Exc_k(W_E, \wh{G}),
\]
which can be interpreted as the ``set of semisimple $L$-parameters attached to $\Pi$''. For example, if $\Pi \in \Rep^{\sm}_k G(E)$ is irreducible admissible, then $\supp_{\Exc_k(W_E, \wh{G})}(\Pi)$ is necessarily a single closed point, corresponding to the \emph{Fargues-Scholze parameter} $\rho_{\Pi} \in H^1(W_E; \wh{G}(k))$. This defines the map \eqref{eq: LLC}. 

Let $\Pi \in D^b(\Rep^{\sm}_k (G(E) \rtimes \Sigma)) =  D^b(\Rep^{\sm}_k G(E))^{B \Sigma}$, the equivalence being as in Lemma \ref{lem:equivariant-derived-category-trivial-action}. Then we may form the Tate cohomology $\rT^j(\Pi)$, which is a smooth representation of $H(E)$ by \S \ref{sssec: tate other settings}.

Using the square root fixed in \S \ref{ssec: shtukas}, we transport the $\sigma$-dual homomorphism to $\ld\psi:=\alpha_G^{-1}\circ\lc\psi\circ\alpha_H\co\ld H\rightarrow\ld G$.

\begin{thm}\label{thm: TV} 
Let $F=F_G\co\Exc_k(W_E,\wh G)\rightarrow\Exc_k(W_E,\wh G)$ be the linearized Frobenius of \eqref{eq: transported excursion Frobenius}. This induces a functor $F_*$ on $\Exc_k(W_E, \wh{G})$-modules. 

For any $j \in \Z/2\Z$, $\supp_{\Exc_k(W_E, \wh{H})}  (\rT^j(\Pi))$ lies over $F_* \supp_{\Exc_k(W_E, \wh{G})}( \rH^*( \Pi))$ with respect to the map $ \Spec \Exc_k(W_E, \wh{H}) \xrightarrow{\ld \psi_*} \Spec \Exc_k(W_E, \wh{G}) $ induced by the $\sigma$-dual homomorphism $\ld \psi \co \ld H \rightarrow \ld G$. In other words, there is a commutative diagram
\[
\begin{tikzcd}
\supp_{\Exc_k(W_E, \wh{H})}  (\rT^j(\Pi)) \ar[d] \ar[r, hook] & \Spec \Exc_k(W_E, \wh{H})  \ar[d, "\ld \psi_*"] \\
F_* \supp_{\Exc_k(W_E, \wh{G})}( \rH^* (\Pi))  \ar[r, hook] & \Spec \Exc_k(W_E, \wh{G}) 
\end{tikzcd}
\]
\end{thm}

\begin{example}[Treumann-Venkatesh functoriality conjecture]\label{ex: TV conj}
If $\Pi$ is an irreducible smooth representation of $G(E)$ (hence admissible, by \cite{DHKM}), then $\Exc_k(W_E, \wh{G})$ acts on $\Pi$ through a character $\chi_{\Pi}$. By \eqref{eq: coefficient Frobenius character}, the composition $\chi_{\Pi} \circ F$ is the character of $\Exc_k(W_E, \wh{G})$ associated to the Frobenius twist $\Pi^{(\ell)} := \Pi \otimes_{k, \Frob_\ell} k$, so we have 
\[
F_* \supp_{\Exc_k(W_E, \wh{G})} ( \Pi) =  \supp_{\Exc_k(W_E, \wh{G})} (F_*  \Pi) = \{ \chi_{\Pi^{(\ell)}} \}.
\]
Suppose furthermore that $\Pi \cong {}^\sigma \Pi \in \Rep^{\sm}_k G(E)$. Then by \cite[Proposition 6.1]{TV}, the $G(E)$-action on $\Pi$ extends uniquely to a $G(E) \rtimes \Sigma$-action. Then Theorem \ref{thm: TV} and Lemma \ref{lem: excursion pullback} imply that for each $j\in\Z/2\Z$ and every irreducible
$H(E)$-subquotient $\pi$ of $\rT^j(\Pi)$, we have 
\[
\rho_{\Pi^{(\ell)}}\sim (\ld\psi\circ\rho_\pi)^{\mrm{ss}}.
\]
This proves Theorem \ref{thm: intro TV conj}. 
\end{example}


\begin{example}\label{ex: inner has equivariant structure}
Suppose $\sigma$ is inner, induced by conjugation by an element $s\in G(E)$ of order $\ell$. Then there is an isomorphism
\[
G(E)\rtimes\Sigma\xrightarrow{\sim}G(E)\times\Sigma,
\qquad
(g,\sigma^i)\longmapsto(gs^i,\sigma^i).
\]
Hence we get a functor 
\[
D^b(\Rep^{\sm}_k G(E)) \rightarrow D^b(\Rep^{\sm}_k ( G(E) \times \Sigma))  \xrightarrow{\sim} D^b(\Rep^{\sm}_k (G(E) \rtimes \Sigma)). 
\]
This equips every representation $\Pi \in D^b(\Rep^{\sm}_k G(E))$ with a canonical $\Sigma$-equivariant structure, where $\sigma$ acts by $\Pi(s)$. 
\end{example}

\subsection{The Treumann-Venkatesh homomorphism}\label{sssec: double coset fixed point} We next establish some technical lemmas which aid in studying the properties of the Brauer homomorphism. 




\begin{prop}\label{prop: inj}
Let $K \subset G(E)$ be a $\Sigma$-stable compact open subgroup with prime-to-$\ell$ pro-order. Then the natural map 
\[
\iota^* K \bs H(E)  / \iota^* K \rightarrow K \bs G(E)  / K
\]
is injective. 
\end{prop}

\begin{proof}
Let $a,b \in \iota^* K \bs H(E)$ be two elements whose images in $K \bs G(E)$ lie in the same orbit for the right translation of $K$. In other words, $a = b\kappa$ for some $\kappa \in K$. Applying $\sigma$ to this equation, and using that $a,b$ are fixed by $\sigma$, we obtain $a = b \sigma(\kappa)$, so $\sigma(\kappa) \kappa^{-1} \in \Stab_{K}(b)$. Since $\Sigma$ is of order $\ell$ while $K$ is prime-to-$\ell$, we have $\rH^1(\Sigma; \Stab_{K}(b)) = 0$. Hence there exists $y \in \Stab_{K}(b)$ such that $\sigma(\kappa) \kappa^{-1} = \sigma(y) y^{-1}$. Then $y^{-1} \kappa$ is fixed by $\sigma$, so $y^{-1} \kappa \in H(E) \cap K = \iota^* K$. But then 
\[
 a = b \kappa = (by^{-1}) \kappa  = b (y^{-1} \kappa),
 \]
which shows that $a$ and $b$ lie in the same orbit for the right translation of $\iota^* K$ on $\iota^* K \bs H(E)$. 
\end{proof}

Recall that we defined a compact open subgroup $K \subset G(E)$ to be \emph{plain} if the natural map $H(E)/\iota^*K \rightarrow [G(E)/K]^\sigma$
is a bijection. By a similar argument involving the vanishing of non-abelian cohomology (cf. \cite[Lemma 6.6]{Fe23}), any compact open subgroup $K \subset G(E)$ with prime-to-$\ell$ pro-order (henceforth abbreviated \emph{prime-to-$\ell$ subgroup}) is plain, so the (un-normalized) Brauer homomorphism (cf. \S \ref{sssec: un-normalized Brauer}) $\Br \co \sH(G, K)^{\sigma} \rightarrow \sH(H, \iota^* K)$ is defined for any such $K$. 

\begin{cor}\label{cor: Brauer center}
Let $K \subset G(E)$ be a $\Sigma$-stable prime-to-$\ell$ compact open subgroup. Then $\Br$ induces a map of centers,
\begin{equation}\label{eq: Z(Br)}
\Br \co Z(\sH(G, K)^\sigma) \rightarrow Z(\sH(H, \iota^* K)).
\end{equation}
\end{cor}

\begin{proof}
Proposition \ref{prop: inj} implies that $\Br \co \sH(G, K)^\sigma \rightarrow \sH(H, \iota^* K)$ is a surjective $k$-algebra homomorphism, so it maps the center to the center. 
\end{proof}

It is evident from the definition that the map \eqref{eq: Z(Br)} factors over the quotient $Z(\sH(G, K)^{\sigma})  / N \cdot \mf{Z}(G, K)$, which induces a map 
\begin{equation}\label{eq: tate Brauer}
\rT^0 \mf{Z}(G, K)  = \frac{\mf{Z}(G, K)^{\sigma} }{ N \cdot \mf{Z}(G, K)} \inj \frac{Z(\sH(G, K)^{\sigma})  }{ N \cdot \mf{Z}(G, K)}  \rightarrow \mf{Z}(H, \iota^* K).
\end{equation}
Note that we have a natural $\F_\ell$-structure on $\mf{Z}(G, K;k)$ given by 
\begin{equation}\label{eq: Hecke F_ell structure}
\mf{Z}(G, K; k) = \mf{Z}(G, K; \F_\ell) \otimes_{\F_\ell} k.
\end{equation}
The following is a ``normalized'' (cf. \S \ref{sssec: normalized Brauer}) version of \eqref{eq: Z(Br)}, in the sense of the normalized Brauer homomorphism \eqref{eq: normalized Brauer}.

\begin{defn}\label{def: TV homomorphism}
Let $K \subset G(E)$ be a $\Sigma$-stable prime-to-$\ell$ compact open subgroup. We define the \emph{Treumann-Venkatesh homomorphism} 
\[
\mf{Z}_{\TV,K}  \co \mf{Z}(G, K) \rightarrow \mf{Z}(H, \iota^* K)
\]
to be the composition of \eqref{eq: tate Brauer} with the map $\Nm \iellt \co \mf{Z}(G, K) \rightarrow \rT^0(\mf{Z}(G,K))$ from Definition \ref{def: tate diagonal}, with respect to the $\F_\ell$-structure \eqref{eq: Hecke F_ell structure}. 
\end{defn}

\subsection{Modular functoriality for Bernstein centers}\label{ssec: BC homomorphism for BC}

For an inclusion $K \subset K' \subset G(E)$ of $\Sigma$-stable prime-to-$\ell$ compact open subgroups, we have the map $e_G^{K \rightarrow K'} \co \mf{Z}(G, K) \rightarrow \mf{Z}(G, K')$ given by convolution with the idempotent $[K':K]^{-1}\bbm{1}_{K'}$ (the index is prime to $\ell$). Similarly, we have $e_H^{K \rightarrow K'} \co \mf{Z}(H, \iota^* K) \rightarrow \mf{Z}(H, \iota^* K')$ given by convolution with $[\iota^*K':\iota^*K]^{-1}\bbm{1}_{\iota^* K'}$. The diagram 
\[
\begin{tikzcd}
\mf{Z}(G, K) \ar[r, "\mf{Z}_{\TV,K}"]  \ar[d, "e_G^{K \rightarrow K'}"] & \mf{Z}(H, \iota^*K) \ar[d, "e_H^{K \rightarrow K'}"]  \\
\mf{Z}(G, K') \ar[r, "\mf{Z}_{\TV,K'}"] & \mf{Z}(H, \iota^* K') 
\end{tikzcd}
\]
commutes.

\begin{defn}[Base change homomorphism for Bernstein centers]\label{def: BC for BC} We define the \emph{Treumann-Venkatesh homomorphism} (for Bernstein centers) $\mf{Z}_{\TV} \co \mf{Z}(G) \rightarrow \mf{Z}(H)$ as
\[
 \varprojlim_{K} \mf{Z}_{\TV,K} \co \varprojlim_{K} \mf{Z}(G, K) \rightarrow \varprojlim_{K} \mf{Z}(H, \iota^* K) 
\]
where the limit is taken over $\Sigma$-stable prime-to-$\ell$ compact open subgroups $K \subset G(E)$. 
\end{defn}

Recall that in Theorem \ref{thm: 1-leg brauer functor} we have constructed the $\sigma$-dual homomorphism $\lc \psi \co \lc H \rightarrow \lc G$ over $k$. Using the fixed choice of $\sqrt{q} \in k$, we may transport it to $\ld\psi\co\ld H\rightarrow\ld G$. 

\begin{thm}\label{thm: FS TV} The diagram \eqref{eq: base change diagram} commutes.
\begin{equation}\label{eq: base change diagram}
\begin{tikzcd}
\Exc_k(W_E, \wh{G})  \ar[r, "\ld \psi^*"] \ar[d, "\FS_G"']  & \Exc_k(W_E, \wh{H})  \ar[d, "\FS_H"]\\
\mf{Z}(G)\ar[r, "\mf{Z}_{\TV}"] & \mf{Z}(H)
\end{tikzcd}
\end{equation}
\end{thm}

Note that if $K$ is a plain subgroup of $G(E)$, then we have 
\begin{equation}\label{eq: tate of cInd}
\rT^0 (\cInd_K^{G(E)} k) \cong \cInd_{\iota^*K}^{H(E)} k \in \Rep_k^{\mrm{sm}} H(E).
\end{equation}
(This follows, for example, from Lemma \ref{lem: tate of compact induction}.)

In preparation for the proof of Theorem \ref{thm: FS TV}, we record the following interpretation of the Brauer homomorphism. 

\begin{lemma}\label{lem: Br}
Under the identifications $\sH(G, K) \cong \End_{G(E)}(\cInd_K^{G(E)} k)$ and $\sH(H, \iota^*K) \cong\End_{H(E)} (\cInd_{\iota^*K}^{H(E)} k)$, the map $\Br \co \rT^0\sH(G, K) \rightarrow \sH(H, \iota^* K)$ induced by the Brauer homomorphism is identified with the map 
\[
\rT^0 \End_{G(E)}(\cInd_K^{G(E)} k) \rightarrow \End_{H(E)}(\rT^0 (\cInd_K^{G(E)} k) ) \stackrel{\eqref{eq: tate of cInd}}\cong \End_{H(E)}(\cInd_{\iota^*K}^{H(E)} k)
\]
sending a $\Sigma$-equivariant endomorphism of $\cInd_K^{G(E)} k$ to the induced endomorphism on its Tate cohomology. 
\end{lemma}

\begin{proof}
Apply \cite[Lemma 6.7]{Fe23} (which comes from \cite[\S 6.2]{TV}) with $\Pi := \cInd_K^{G(E)} k$. 
\end{proof}

\begin{proof}[Proof of Theorem \ref{thm: FS TV}] For any prime-to-$\ell$, $\Sigma$-stable compact open subgroup $K \subset G(E)$, we also denote by
\[
\Exc_k(W_E, \wh{G}) \xrightarrow{\FS_{G}}  \mf{Z}(G, K) 
\]
the composition of $\FS_G$ with the projection $\mf{Z}(G) \rightarrow \mf{Z}(G,K)$. Applying Tate cohomology, this induces
\begin{equation}\label{eq: tate of FS}
\rT^0\Exc_k(W_E, \wh{G})  \xrightarrow{\FS_{G}} \rT^0 \mf{Z}(G, K).
\end{equation}

We have identifications
\begin{equation}\label{eq: Bernstein fixed endpoint comparison}
\begin{tikzcd}[ampersand replacement=\&]
\rT^0(\cInd_K^{G(E)} k) \ar[r, "\sim"', "\text{Ex. \ref{ex: tate cohomology of trivial shtuka}}"] \&  \rT^0(\Sht_{(G, 1_G, \{0\}),K}; \bbm{1})   \ar[r, "\sim"', "\eqref{eq: equiv local trivial coeff}"] \&  \rT^0(\Sht_{(H,1_H,\{0\}),\iota^*K};\bbm1) \ar[d, "\sim", "\text{Ex. \ref{ex: tate cohomology of trivial shtuka}}"'] \\
 \&  \cInd_{\iota^* K}^{H(E)} k \& \ar[l, "\sim"', "\text{Ex. \ref{ex: tate coh of trivial}}"]  \rT^0(\cInd_{\iota^* K}^{H(E)} k) 
\end{tikzcd}
\end{equation}
Here the vertical identification is the analogue of Example \ref{ex: tate cohomology of trivial shtuka} for $H$. Now, for any excursion datum $\cD$ for $\wh{G}$, Theorem \ref{thm: localize excursion operator} implies that the identifications in \eqref{eq: Bernstein fixed endpoint comparison} intertwine
\begin{equation}\label{eq: excursion action localize}
\left(\begin{array}{@{}c@{}}
\text{the action of }  \text{$S_{\Nm\iellt \cD}$} \\  \text{on $\rT^0(\cInd_K^{G(E)} k)$}
\end{array} \right)
 =
\left( \begin{array}{@{}c@{}} \text{the action of }\ld\psi^*(S_{\cD})\\ \text{on $\cInd_{\iota^*K}^{H(E)} k$}
 \end{array}\right).
\end{equation}
 
The diagram 
\begin{equation}\label{eq: Nm Exc and Z}
\begin{tikzcd}
\Exc_k(W_E, \wh{G}) \ar[r, "\Nm \iellt"] \ar[d, "\FS_G"]  & \rT^0 \Exc_k(W_E, \wh{G}) \ar[d, "\eqref{eq: tate of FS}"]\\
\mf{Z}(G, K) \ar[r, "\Nm \iellt"]  & \rT^0 \mf{Z}(G, K)
\end{tikzcd}
\end{equation}
commutes by the definition of $\Nm$, and the fact that $\FS_G$ respects the transported $\F_\ell$-structure of \S \ref{sssec: excursion Frobenius} on the source and \eqref{eq: Hecke F_ell structure} on the target. Therefore, unraveling the definition of $\mf{Z}_{\TV}$ and using Lemma \ref{lem: norm extension}, it suffices to show that 
\begin{equation}\label{eq: br FS identity}
\Br(\FS_G(S_{\Nm \iellt \cD}))  = \FS_H(\ld\psi^*(S_{\cD}))
\end{equation}
for all excursion data $\cD$ for $\wh{G}$. By Lemma \ref{lem: Br}, $\Br(\FS_G(S_{\Nm \iellt \cD}))$ is the endomorphism of $\rT^0(\cInd_K^{G(E)} k)$, identified via \eqref{eq: tate of cInd} with an endomorphism of $\cInd_{\iota^* K}^{H(E)} k$, obtained by taking $\rT^0$ of the action of $S_{\Nm \iellt \cD}$ on $\cInd_K^{G(E)} k$, which according to \eqref{eq: excursion action localize} is the endomorphism given by $\FS_H(\ld\psi^*(S_{\cD}))$.
\end{proof}

\begin{cor}\label{cor: functoriality character} For any irreducible $H(E)$-representation $\pi$, the character 
\[
\chi_{\pi} \circ \mf{Z}_{\TV} \co \mf{Z}(G) \xrightarrow{\mf{Z}_{\TV}} \mf{Z}(H) \xrightarrow{\chi_{\pi}} k
\]
has the property that for any irreducible $G(E)$-representation $\Pi$ on which $\mf{Z}(G)$ acts through $\chi_{\pi} \circ \mf{Z}_{\TV}$, we have $\rho_\Pi\sim(\ld\psi\circ\rho_\pi)^{\mrm{ss}}$.
\end{cor} 


\subsection{Proof of Theorem \ref{thm: TV}}\label{ssec: TV conjectures}

\begin{lemma}\label{lem: normed action via TV}
Let $\Pi\in D^b(\Rep_k^{\sm}G(E))^{B\Sigma}$ and $j\in\Z/2\Z$. Then the normed excursion action of $\Exc_k(W_E,\wh G)$ on $\rT^j(\Pi)$ agrees with the action obtained from
\[
\Exc_k(W_E,\wh G)\xrightarrow{\ld\psi^*}
\Exc_k(W_E,\wh H)\xrightarrow{\FS_H}\mf Z(H).
\]
Moreover, $\ld\psi^*\circ\sigma=\ld\psi^*$ on $\Exc_k(W_E,\wh G)$, and hence $\ld\psi^*(N\cdot\Exc_k(W_E,\wh G))=0$.
\end{lemma}

\begin{proof}
Proposition \ref{prop: br compatible with fiber functor} gives a tensor isomorphism $\wcbr\circ{}^\sigma(-)\cong\wcbr$, so Tannakian reconstruction implies that $\sigma\circ\ld\psi \sim \ld\psi$. This implies that $\ld\psi^*(\sigma S)=\ld\psi^*(S)$ for all excursion operators $S \in \Exc_k(W_E,\wh G)$, hence 
\[
\ld\psi^*(N\cdot S)=\sum_{i=0}^{\ell-1}\ld\psi^*(\sigma^i S)
=\ell\,\ld\psi^*(S)=0.
\]

It remains to compare the two actions. It suffices to do so on a generator $S_{\cD}$. Put
\[
\xi_G(\cD) :=\FS_G(S_{\Nm\iellt\cD})\in\mf Z(G)^\sigma,
\qquad
\xi_H(\cD):=\FS_H(\ld\psi^*(S_{\cD}))\in\mf Z(H).
\]
By Lemma \ref{lem: norm extension}, the normed action of $S_{\cD}$ is via $\xi_G(\cD)$. 

Let $K$ range over a cofinal system of $\Sigma$-stable compact open subgroups of prime-to-$\ell$ pro-order. By smoothness, $\Pi=\colim_K\Pi^K$, and Tate cohomology commutes with this filtered colimit, so every vector in $\rT^j(\Pi)$ is the image of some $u\in\rT^j(\Pi^K)$. The normed action of $S_{\cD}$ on $u$ is through $\Br(\FS_G(S_{\Nm \iellt \cD}))$, which by \eqref{eq: br FS identity} identifies with $\FS_H(\ld\psi^*(S_{\cD}))$, as desired. 
\end{proof}

\begin{proof}[Proof of Theorem \ref{thm: TV}]
By Lemma \ref{lem: normed action via TV}, restriction along $\ld\psi^*$ of the native $\Exc_k(W_E,\wh H)$-action on $\rT^j(\Pi)$ is the normed $\Exc_k(W_E,\wh G)$-action. Hence $\ld\psi_*$ carries
\begin{equation}\label{eq: supp 1}
\supp_{\Exc_k(W_E, \wh{H})}^{\mrm{native}} \rT^j(\Pi)  \xrightarrow{\ld \psi_* } \supp_{\Exc_k(W_E, \wh{G})}^{\mrm{normed}} \rT^j(\Pi),
\end{equation}
where the native excursion action is defined in \S \ref{sssec: native excursion action}.

 We will compare the support of the normed and native actions of $\Exc_k(W_E, \wh{G})^{\sigma}$ on $\rT^j(\Pi)$. By the discussion in \S \ref{sssec: normed excursion action}, we have for any $j \in \Z/2\Z$ an equality
\begin{equation}\label{eq: supp 1.3}
\supp_{\Exc_k(W_E, \wh{G})^\sigma}^{\mrm{normed}} (\rT^j (\Pi))  =  F_* \supp^{\mrm{native}}_{\Exc_k(W_E, \wh{G})^\sigma} (\rT^j (\Pi)) \subset \Spec \Exc_k(W_E, \wh{G})^\sigma.
\end{equation}
By the functorial filtration of Remark \ref{rem: tate filtration}, each $\rT^j(\Pi)$ has a finite filtration each of whose graded pieces is a subquotient of $\rH^*(\Pi)$, hence we have
\begin{equation}\label{eq: supp 1.5}
 F_* \supp^{\mrm{native}}_{\Exc_k(W_E, \wh{G})^\sigma} (\rT^j (\Pi))  \subset F_* \supp^{\mrm{native}}_{
\Exc_k(W_E, \wh{G})^\sigma} (\rH^*(\Pi))  = \supp^{\mrm{native}}_{\Exc_k(W_E, \wh{G})^\sigma} (F_* \rH^*(\Pi)).
\end{equation}
Let $q \co \Spec \Exc_k(W_E, \wh{G}) \rightarrow \Spec \Exc_k(W_E, \wh{G})^\sigma$ be the map of spectra induced by the obvious inclusion. Combining \eqref{eq: supp 1.3} and \eqref{eq: supp 1.5} with \eqref{eq: supp 1} yields
\begin{equation}\label{eq: supp 2}
q \circ \ld \psi_*  \left( \supp_{\Exc_k(W_E, \wh{H})}^{\mrm{native}} \rT^j(\Pi) \right) \subset F_* \supp^{\mrm{native}}_{\Exc_k(W_E, \wh{G})^\sigma} (\rH^* (\Pi)).
\end{equation}


The surjection $\Exc_k(W_E, \wh{G})^\sigma \surj \rT^0 \Exc_k(W_E, \wh{G})$ induces a closed embedding on spectra, as in the diagram 
\[
\begin{tikzcd}  
& \Spec \Exc_k(W_E, \wh{G}) \ar[d, "q"] \\
\Spec \rT^0 \Exc_k(W_E, \wh{G})  \ar[r, hook] & \Spec \Exc_k(W_E, \wh{G})^\sigma
\end{tikzcd}
\]
Now, \cite[Lemma 5.15]{Fe23} says that any character from $\rT^0 \Exc_k(W_E, \wh{G})$ to any perfect field has a unique extension to $\Exc_k(W_E, \wh{G})$, which implies that the map $\Spec \Exc_k(W_E, \wh{G}) \rightarrow \Spec \Exc_k(W_E, \wh{G})^\sigma$ is one-to-one (and surjective) over the closed subscheme $\Spec \rT^0 \Exc_k(W_E, \wh{G}) \subset \Spec \Exc_k(W_E, \wh{G})^\sigma$. Since $\ld\psi^*$ annihilates $N\cdot\Exc_k(W_E,\wh G)$ by Lemma \ref{lem: normed action via TV}, every prime in the image of $\ld\psi_*$ contains this ideal. Therefore $q \circ \ld \psi_*  \left( \supp_{\Exc_k(W_E, \wh{H})}^{\mrm{native}} \rT^j(\Pi) \right)$ lands inside this subspace, so we may lift \eqref{eq: supp 2} to the containment 
\begin{equation}
\ld \psi_*  \left( \supp_{\Exc_k(W_E, \wh{H})}^{\mrm{native}} \rT^j(\Pi) \right)  \subset F_* \supp_{\Exc_k(W_E, \wh{G})}^{\mrm{native}} (\rH^*(\Pi)),
\end{equation}
as desired. 
\end{proof}

\subsection{Existence of functorial lifts}\label{ssec: functorial lift} We generalize the proof of \cite[Theorem 6.26]{Fe23} to prove the existence of functorial liftings along any $\sigma$-dual homomorphism. 

\begin{thm}\label{thm: mod ell lifting}  Let $\pi$ be an irreducible representation of $H(E)$ over $k$, with Fargues-Scholze parameter $\rho_{\pi} \in \rH^1(W_E; \wh{H}(k))$. Then there is an irreducible representation $\Pi$ of $G(E)$ over $k$ such that $\rho_{\Pi}\sim(\ld\psi\circ\rho_\pi)^{\mrm{ss}}\in\rH^1(W_E;\wh G(k))$. 
\end{thm}

\begin{proof} Choose a $\Sigma$-stable pro-$p$ compact open subgroup $K < G(E)$ such that $\pi^{\iota^* K} \neq 0$. From Theorem \ref{thm: FS TV} we have a commutative diagram 
\begin{equation}\label{eq: FS functorial diagram}
\begin{tikzcd}[column sep = huge]
\Spec  \mf{Z}(H, \iota^* K) \ar[r, "\Spec \mf{Z}_{\TV}"]  \ar[d, "\Spec \FS_H"]  & \Spec \mf{Z}(G, K) \ar[d, "\Spec \FS_G"]  \\
\Spec \Exc_k(W_E, \wh{H}) \ar[r, "\ld \psi_*"] & \Spec \Exc_k(W_E, \wh{G})  
\end{tikzcd}
\end{equation}
The representation $\pi$ gives a $k$-point of $\Spec  \mf{Z}(H, \iota^* K) $ lying over $\rho_\pi$ viewed as a $k$-point of $\Spec \Exc_k(W_E, \wh{H})$. The commutativity of the diagram then tells us that there is a $k$-point of $\Spec \mf{Z}(G,K)$, say with maximal ideal $\mf{m}$, lying over the $k$-point of $\Spec \Exc_k(W_E, \wh{G})  $ corresponding to $(\ld\psi\circ\rho_\pi)^{\mrm{ss}}$.

Recall that the functor $\Pi \mapsto \Pi^{K}$ induces a bijection between irreducible admissible $G(E)$-representations with non-zero $K$-invariants and irreducible $\sH(G, K)$-modules (cf. \cite[\S I.4]{Vig96}). It therefore suffices to construct an irreducible representation of $\sH(G,K)$ on which the $\mf{Z}(G,K)$-action factors over $\mf{m}$. 

By a result of Dat-Helm-Kurinczuk-Moss \cite[Theorem 1.1]{DHKM} $\sH(G, K)$ is a finite $\mf{Z}(G, K)$-module. The localization of $\sH(G,K)$ at $\mf{m}$ is non-zero, since $\mf{Z}(G,K) \inj \sH(G,K)$, so Nakayama's Lemma implies that the left $\sH(G, K)$-module quotient $\sH(G, K) /\sH(G, K) \mf{m}$ is finite-dimensional and non-zero. By design, it is supported over $\mf{Z}(G,K)/\mf{m}$, so it has an irreducible $\sH(G, K)$-subquotient on which the $\mf{Z}(G,K)$-action factors over $\mf{m}$, as was to be shown.

\end{proof}

\section{Fargues-Scholze parameters of toral supercuspidals}\label{sec: toral supercuspidals}
In this section we will demonstrate how the preceding theory may be used to calculate the Fargues-Scholze parameters attached to explicitly constructed representations $\Pi \in D^b(\Rep^{\mrm{sm}}_k G(E))$. We will apply Theorem \ref{thm: TV} with the automorphism $\sigma$ being conjugation by a strongly regular $\ell$-torsion element $s \in G(E)$. Then $H$ is a \emph{torus} $T<G$, so the Fargues-Scholze correspondence is completely understood for $H$. Using Theorem \ref{thm: TV}, we can describe $\rho_\Pi \in \rH^1(W_E; \wh{G}(k))$ in terms of $\rT^j(\Pi)$ and the $\sigma$-dual homomorphism $\ld \psi$.

We will focus our attention on $\Pi \in D^b(\Rep^{\mrm{sm}}_k G(E))$ related to the ``(Howe-unramified) toral supercuspidals'' studied by Chan-Oi \cite{CO21}. (The method has broader scope but we regard this as a sufficiently interesting proof-of-concept for now.) Then $T$ must be taken to be an unramified elliptic torus, and we compute the $\sigma$-dual embedding in \S \ref{ssec: example sigma-dual homomorphism}: it turns out to be the canonical $L$-embedding corresponding to an unramified maximal torus, composed with Frobenius. 

The relevant $\Pi$ arise as compact inductions of representations cut out from the cohomology of so-called ``deep level Deligne-Lusztig varieties'' studied by Chan-Ivanov \cite{CI21}. In \S \ref{ssec: Tate of deep level DL} we use equivariant localization to calculate the Tate cohomology of these deep level Deligne-Lusztig varieties. We feed the answer into \S \ref{ssec: toral representations} in order to compute the Tate cohomology of the compact inductions, and then identify an explicit point in their Fargues--Scholze parameter support in Theorem \ref{thm: derived parameter}.

\subsection{Review of tori}

Let $G$ be a connected reductive group over $E$.

\subsubsection{The abstract Cartan} To $G$ we can associate a canonical $E$-torus $\bT$, which we call the ``abstract Cartan'' of $G$. (We make no claim that $\bT$ admits an $E$-rational embedding into $G$.) If $G$ is quasi-split, then $\bT$ is defined as the colimit over $E$-rational Borel subgroups $B < G$ of $B/U$, where $U$ is the unipotent radical of $B$. In general, we pass to an extension of $E$ where it becomes quasi-split, and then descend this construction.

Given an extension $E'/E$ and a Borus $(B,T)$ in $G_{E'}$, the composition 
\[
T \inj B \surj B/U \rightarrow \bT_{E'}
\]
defines an isomorphism $i_B \co T \xrightarrow{\sim} \bT_{E'}$.  

We denote by $\bW$ the Weyl group of $\bT$, defined as the colimit of $N_{G_{E^s}}(T)/T$ over the category of Boruses $(B,T)$ in $G_{E^s}$.

\subsubsection{The cocycle of a torus}\label{sssec: torus cocycle}

Let $T \subset G$ be a maximal torus over $E$. Over $E^s$, we can find a Borus $(B, T_{E^s})$ inside $G_{E^s}$, which gives an identification $i_B \co T_{E^s} \xrightarrow{\sim} \bT_{E^s}$. 

For $\gamma \in \Gal(E^s/E)$, $\gamma B$ is another Borel subgroup of $G$ containing $T_{E^s}$, so we have another identification $i_{\gamma B} \co T_{E^s} \xrightarrow{\sim} \bT_{E^s}$. 

For $\gamma \in \Gal(E^s/E)$, we denote by $\gamma_{\bT}$ (resp. $\gamma_T$) the endomorphism of $\bT_{E^s}$ (resp. $T_{E^s}$) induced by $\gamma$. From the definition of the abstract Cartan, we see that $\gamma_{\bT}$ is the composition 
\[
 \bT_{E^s} \xrightarrow{i_B^{-1}}  T_{E^s} \xrightarrow{\gamma_T} T_{E^s} \xrightarrow{i_{\gamma B}}  \bT_{E^s}.
\]

For each $\gamma \in \Gal(E^s/E)$, the automorphism $i_B  i_{\gamma B}^{-1}$ of $\bT_{E^s}$ is given by an element of $\bW$. Hence the function $z_{T,B} \co \gamma  \mapsto (i_B i_{\gamma B}^{-1})$ defines a cocycle in $Z^1(\Gal(E^s/E); \bW)$. Choosing a different $B$ alters $z_{T,B}$ by a coboundary, hence the cohomology class $h_T := [z_{T,B}] \in \rH^1(\Gal(E^s/E); \bW)$ is independent of $B$. Then we have 
\[
i_B \gamma_T i_B^{-1} = i_B (i_{\gamma B}^{-1} \gamma_{\bT} i_B) i_B^{-1} = (i_B i_{\gamma B}^{-1}) \gamma_{\bT}.
\]
In other words, the cocycle $z_{T,B}$ expresses the ``difference'' between the $\Gal(E^s/E)$-action on $T$ and on $\bT$. 

\subsubsection{The canonical $L$-embedding}\label{sssec: canonical L-embedding} Recall from the work of Langlands-Shelstad \cite{LS87} that given a maximal torus $T \subset G$ together with a choice of ``$\chi$-data'', there is a $\wh{G}$-conjugacy class of admissible dual embeddings $\ld j \co \ld T \inj \ld G$. In general there is no distinguished choice of $\ld j$, but if $T$ is \emph{unramified} then it has a distinguished choice of $\chi$-data, which gives a \emph{canonical} conjugacy class of embeddings $\ld j \co \ld T \rightarrow \ld G$. We will describe it more explicitly. 

Using $\wh{G}$-conjugation, we can arrange that $\wh{j}$ sends $\wh{T}$ isomorphically to $\wh{\bT}$. Then $\ld j$ is specified by a cocycle with values in $N_{\wh{G}}(\wh{\bT})(k)$ lifting $h_T\in\rH^1(W_E;\bW)$. Since $T$ is unramified, the actions factor through $\mrm{val}\co W_E\surj\tw{\varphi}\cong\Z$, and we choose a lift inflated from $\tw{\varphi}$.

Fix the representative $z_{T,B}$ of $h_T$ from \S \ref{sssec: torus cocycle}, and choose $n\in N_{\wh{G}}(\wh{\bT})(k)$ lifting $z_{T,B}(\varphi)$. Since $\tw{\varphi}\cong\Z$, the assignment $\varphi\mapsto n$ defines a cocycle lifting $z_{T,B}$. The lifts of this representative, modulo $\wh{\bT}(k)$-conjugacy, are parametrized by $\rH^1(\tw{\varphi};\wh{T}(k))$.
We will explicate the correct lift only in the case that $T$ is \emph{elliptic}, i.e., anisotropic mod center. 

First suppose $G$ is semisimple: then $T$ is anisotropic, so $t\mapsto t\,\varphi(t)^{-1}$ is an isogeny of $\wh{T}$, which is surjective on $k$-points because $k$ is algebraically closed. Hence 
\[
\rH^1(\tw{\varphi}; \wh{T}(k)) = \wh{T}(k)/\varphi\mrm{-conj} = \{1\},
\]
hence $\ld j$ is uniquely determined in this case by the condition of lifting $h_T$. 

In general, consider the adjoint quotient $G \surj G_{\ad}$. Let $\ol{T} := T/Z(G)$, so we have a pullback square
\[
\begin{tikzcd}
T \ar[r] \ar[d] & \ol{T} \ar[d] \\
G \ar[r] & G_{\ad}
\end{tikzcd}
\]
On the dual side we have $\wh{G_{\ad}}\rightarrow\wh G$ and $\wh{\ol T}\rightarrow\wh T$, forming a pushout square. Let $\ld j_{\ad}\co\ld\ol T\rightarrow\ld G_{\ad}$ be the $L$-embedding specified in the preceding paragraph. Then we have a commutative square
\[
\begin{tikzcd}[ampersand replacement=\&]
\ld\ol T\ar[d,"\ld j_{\ad}"]\ar[r] \& \ld T\ar[d,"\ld j"]\\
\ld G_{\ad}\ar[r] \& \ld G
\end{tikzcd}
\]
This discussion is phrased for the usual $L$-group. In \S \ref{ssec: sigma-dual homomorphism}, the $\sigma$-dual homomorphism is constructed first for the $c$-group and is then transported to the usual $L$-group after choosing a square root of the residue cardinality; the comparison below uses that same choice.

\subsection{Calculation of the $\sigma$-dual homomorphism}\label{ssec: example sigma-dual homomorphism}

Let $T \subset G$ be an unramified elliptic maximal torus defined over $E$. We use the square root fixed in \S \ref{ssec: shtukas} to identify $\ld G\cong\lc G$.

\begin{prop}\label{prop: sigma-dual embedding unram elliptic torus} Suppose $T(E)$ contains an element $s$ of order $\ell$ which maps to a strongly regular element of $G_{\ad}(E)$. Let $\sigma = \mrm{conj}_s$ be the conjugation action of $s$ on $G$, so $G^\sigma = T$.  Then the $\sigma$-dual homomorphism $\ld \psi \co \ld T \rightarrow \ld G$ satisfies $\ld \psi \sim \ld j \circ \Fr_\ell$, where $\Fr_\ell$ is induced by the Frobenius endomorphism $\wh{T} \rightarrow \wh{T}$ (corresponding to multiplication by $\ell$ on character groups). 
\end{prop}

%
%

We begin with some preliminaries before turning to the proof. Any choice of $B \supset T_{E^s}$ gives a commutative diagram
\begin{equation}\label{eq: dual embedding calc 1}
\begin{tikzcd}
T_{E^s} \ar[r, equals] \ar[d, "i_B"] & T_{E^s} \ar[d] \ar[r, equals] & T_{E^s} \ar[d] \\
\bT_{E^s} & B \ar[l] \ar[r]  & G_{E^s}
\end{tikzcd}
\end{equation}
By Lemma \ref{lem: CT br} and Corollary \ref{cor: dual torus}, diagram \eqref{eq: dual embedding calc 1} induces an isomorphism 
\begin{equation}\label{eq: dual embedding calc 2}
\wcbr\cong(\Fr_\ell)^*i_B^*\CT_B[\deg_G]\co
\Rep_k(\wh G)\longrightarrow\Rep_k(\wh T);
\end{equation}
here we make no assertion about $W_E$-actions yet. Thus the choice of $B$ gives an identification $\wh{i}_B^{-1} \co \wh{T} \xrightarrow{\sim} \wh{\bT}$ in such a way that the map $\chpsi$, the restriction of $\ld \psi$ to the identity component, factors as in the diagram below 
\begin{equation}\label{eq: dual embedding calc 3}
\begin{tikzcd} 
\wh{T} \ar[rr, "\chpsi", bend left] \ar[r, "{\chphi_B}"'] & \wh{\bT} \ar[r, hook] &  \chG
\end{tikzcd}
\end{equation}
where $\chphi_B$ is $\wh{i}_B^{-1}$ composed with the Frobenius endomorphism $\Fr_\ell$ of $\wh{T}$. As discussed in \S \ref{sssec: torus cocycle}, the map $\wh{i}_B$ carries the $W_E$-action on $\wh{T}$ to the $W_E$-action on $\wh{\bT}$ twisted by the cocycle $z_{T,B}$, so for all $\gamma \in W_E$ we have 
\begin{equation}\label{eq: twisted action}
\chpsi \circ \gamma_{\wh{T}}  = z_{T,B}(\gamma) \cdot ( \gamma_{\wh{\bT}} \circ \chphi_B).
\end{equation}


This has the following consequence. By \eqref{eq: twisted action}, the homomorphism $\ld \psi$ carries $W_E \subset \ld T$ into $N_{\wh{G}}(\wh{\bT}) \rtimes W_E$, so its projection to the first component defines a class $\pr_1(\ld \psi) \in \rH^1(W_E;  N_{\wh{G}} (\wh{\bT}))$. Denote by $[\pr_1(\ld \psi)] \in \rH^1(W_E; \bW)$ the projection of $\pr_1(\ld \psi)$ along $N_{\wh{G}}(\wh{\bT}) \surj \bW$. From \eqref{eq: twisted action} we conclude: 

\begin{lemma}\label{lem: cocycle in W}
We have $h_T = [\pr_1(\ld \psi)] \in \rH^1(W_E;\bW)$.
\end{lemma}

\begin{lemma}\label{lem:inflated}
The cohomology class $\pr_1(\ld \psi) \in \rH^1(W_E;N_{\wh{G}}(\wh{\bT}))$ is inflated from $\rH^1(\Z; N_{\wh{G}}(\wh{\bT}))$.
\end{lemma}

\begin{proof}
We need to see that the cocycle can be trivialized on $I_E \subset W_E$. Since $T$ is unramified, it can be split after passing to a finite unramified extension. Choose a finite unramified extension $E'/E$ splitting $T$. Compatibility with unramified base change (Lemma \ref{lem:psi-bc} and Remark \ref{rem:ldpsi-base-change}) identifies the restrictions of the classes $\pr_1(\ld\psi)$ and $\pr_1(\ld\psi_{E'})$ to $I_E=I_{E'}$. Since $T_{E'}$ is a split maximal torus, $G_{E'}$ is split; choosing a Borel subgroup over $E'$ makes \eqref{eq: dual embedding calc 2} $W_{E'}$-equivariant. Thus $\pr_1(\ld\psi_{E'})$ is cohomologically trivial on $I_{E'}$, and so is $\pr_1(\ld\psi)$ on $I_E$. Note that the warning in Remark \ref{rem:ldpsi-base-change} causes no difficulty here: inertia maps to $0$ under the valuation $W_E \rightarrow \Z$, hence lies in the subgroup on which the chosen square-root identifications of $L$-groups and $c$-groups are compatible.
\end{proof}

\begin{proof}[Proof of Proposition \ref{prop: sigma-dual embedding unram elliptic torus}]

First consider the semisimple case. Then $T$ is anisotropic in this case, so $\rH^1(\tw{\varphi};\wh{T}(k))=\{1\}$ as explained in \S \ref{sssec: canonical L-embedding}, and the canonical unramified $L$-embedding $\ld j$ is uniquely characterized by the requirement that its $\bW$-valued cocycle be $h_T$. Hence Lemmas \ref{lem: cocycle in W} and \ref{lem:inflated} already imply that Proposition \ref{prop: sigma-dual embedding unram elliptic torus} holds in this case. 

If $G$ is a torus, the assertion follows from
Proposition \ref{prop: sigma-dual tori}. Otherwise, put
$\ol T=T/Z(G)$. The image of $s$ in $G_{\ad}$ has order $\ell$
and centralizer $\ol T$, which is unramified elliptic. The discussion of the semisimple case applies to $(G_{\ad},\ol T)$. Compatibility with the
central quotient (Lemma \ref{lem:psi-central-isogeny})
gives a commutative diagram
\[
\begin{tikzcd}
\ld \ol{T} \ar[d, "\ld \psi_{\ad}"] \ar[r] & \ld T \ar[d, "\ld \psi"] \\
\ld G_{\ad} \ar[r] & \ld G
\end{tikzcd}
\]
Corollary \ref{cor: dual torus} implies that it factors as 
\[
\begin{tikzcd}
\ld \ol{T} \ar[dd, "\ld \psi_{\ad}"', bend right] \ar[r] \ar[d, "\Fr_\ell"]  & \ld T \ar[dd, "\ld \psi", bend left] \ar[d, "\Fr_\ell"']  \\
\ld \ol{T} \ar[d, "\zeta_{\ad}"] \ar[r] & \ld T \ar[d, "\zeta"'] \\
\ld G_{\ad} \ar[r] & \ld G
\end{tikzcd}
\]
Choose compatible representatives of the $L$-homomorphisms with $\zeta|_{\wh T}=\wh j$ and $\zeta_{\ad}|_{\wh{\ol T}}=\wh j_{\ad}$, using \eqref{eq: dual embedding calc 3}. Since $\zeta_{\ad}\sim\ld j_{\ad}$, choose $g\in\wh{G_{\ad}}(k)$ conjugating $\zeta_{\ad}$ to $\ld j_{\ad}$. Their restrictions to $\wh{\ol T}$ agree, so $g$ centralizes the maximal torus $\wh j_{\ad}(\wh{\ol T})$ and hence lies in it. Its image in $\wh G(k)$ therefore centralizes $\wh j(\wh T)$. Conjugating $\zeta$ by this image preserves its restriction to $\wh T$ and makes it agree with $\ld j$ on the image of $\ld\ol T$. These two subgroups generate $\ld T$, since $\ld\ol T\rightarrow\ld T$ induces the identity on $W_E$. Thus $\zeta\sim\ld j$.
\end{proof}

\subsection{Tate cohomology of deep level Deligne-Lusztig varieties}\label{ssec: Tate of deep level DL} We briefly recall the generalized Deligne-Lusztig representations appearing in \cite{CI21}. 

\subsubsection{Group-theoretic setup} Let $T \inj G$ be an unramified maximal torus and $x \in \cB(G/E)$ be a point of the Bruhat-Tits building of $G$ that lies in the apartment of $T$. Corresponding to $x$ we have by Bruhat-Tits theory a parahoric group scheme $\cG/\cO_E$, whose generic fiber is $G/E$. 

Recall that $\F_q$ is the residue field of $E$. By assumption, $T$ splits over $\brE$. Choose a $\brE$-rational Borel subgroup of $G_{\brE}$ containing $T_{\brE}$, and let $U$ be its unipotent radical.

For $r \in \Z_{\geq 0}$, we have group schemes $\GG_r, \TT_r$ over $\F_q$ as in \cite[\S 6.1]{CO21}\footnote{The indexing seems to differ from that of \cite[\S 2.5, 2.6]{CI21} by $1$.} corresponding to subquotients of the Moy-Prasad filtration at $x$, such that 
\[
\GG_r(\F_q) = G_{x, 0:r+} := G_{x,0}/G_{x,r+} \quad \text{and} \quad \TT_r(\F_q) = T_{x,0:r+} := T_{x, 0}/T_{x,r+}.
\]
We also have a group scheme $\UU_r \subset (\GG_r)_{\ol{\F}_q}$ corresponding to $U$.

\subsubsection{Deep level Deligne-Lusztig varieties} Let $\Fr_q$ be the $q$-power Frobenius for schemes over $\F_q$. We recall certain schemes constructed in \cite[\S 4]{CI21}: the ``deep level Deligne-Lusztig varieties''
\[
S_{\TT_r,\UU_r} := \{ g \in (\GG_r)_{\ol{\F}_q} \co g^{-1} \Fr_q(g) \in \UU_r\}.
\]
(The variety $S_{\TT_r,\UU_r}$ is called $X_r$ in \cite{CO21}.) It is a separated, smooth, finite type scheme over $\ol{\F}_q$, with an action of $\GG_r(\F_q) \times \TT_r(\F_q)$ by multiplication on the left and right.
The action of $\TT_r(\F_q)$ is free, and we define 
\[
X_{\TT_r, \UU_r} :=   S_{\TT_r, \UU_r} / \TT_r(\F_q).
\]


\begin{example} When $r=0$, quotienting $S_{\TT_0,\UU_0}$ by $\UU_0\cap\Fr_q^{-1}(\UU_0)$ and $\TT_0(\F_q)$ gives the classical Deligne--Lusztig variety in the sense of \cite{DL76} (see \cite[Remark 3.2]{CI21}). 
\end{example}

\begin{defn}[Deep level Deligne-Lusztig induction]\label{def: DL induction} Let $\theta \co \TT_r(\F_q) \rightarrow k^{\times}$ be a character. Let $\cL_{\theta}$ be the corresponding local system on $X_{\TT_r, \UU_r}$. We define
\[
R^{\GG_r}_{\TT_r,\UU_r }(\theta)  :=
\RGamma_c(X_{\TT_r,\UU_r}; \cL_\theta) \in D^b(\Rep_k \GG_r(\F_q)). 
\]
\end{defn}

\subsubsection{Calculation of Tate cohomology} Following \cite[\S 2.8]{CI21}, we define $W_x(T)$ to be the subgroup of $W(T,G)$ generated by the reflections corresponding to the vector parts of affine roots $\psi$ of $G$ for which $\psi(x) = 0$. 


Following \cite[Definition 5.1]{CI21}, we say that $s \in G_x(\cO_{\brE})$ is \emph{unramified very regular} (with respect to $x$) if $s$ is regular semisimple in $G_{\brE}$, its connected centralizer $Z^\circ_G(s)$ is an $\brE$-split maximal torus of $G_{\brE}$, and $\alpha(s) \not\equiv 1 \pmod{\varpi_E}$ for all roots $\alpha$ of $G_{\brE}$ relative to $Z^\circ_G(s)$. We say that $s \in \GG_r(\ol{\F}_q)$ is unramified very regular if it is the image of an unramified very regular element of $G_x(\cO_{\brE})$. 



\begin{prop}\label{prop: tate of deep DL} Let $s\in\TT_r(\F_q)$ be unramified very regular in $\GG_r(\ol{\F}_q)$ and of order $\ell$. Let $\Sigma$ act on $\GG_r$ by conjugation by $s$. Then for each $j\in\Z/2\Z$,
\begin{equation}\label{eq: tate of DL induction}
\rT^j (R^{\GG_r}_{\TT_r,\UU_r }(\theta))  \cong \bigoplus_{{w \in W_x(T)^{\Fr_q}}} \theta^w \in  \Rep_k \TT_r(\F_q).
\end{equation}
where $\theta^w = \theta \circ w^{-1}$ is the translate of $\theta$ by the action of $w$. 
\end{prop}

\begin{proof} In the canonical $\sigma$-equivariant structure on $R^{\GG_r}_{\TT_r,\UU_r}(\theta)$, the action of $\sigma$ is induced by left multiplication by $s$ on $X_{\TT_r,\UU_r}$. 

 By \cite[Proposition 5.6]{CI21}, we have
\begin{equation}
\Fix(\sigma, X_{\TT_r, \UU_r}) = \bigcup_{w \in W_x(T)^{\Fr_q}} [\dot{w}]
\end{equation}
where $[\dot{w}]$ is the coset $\dot{w}\TT_r(\F_q)$ of any Frobenius-fixed lift $\dot{w}\in\GG_r(\F_q)$ of $w$ (which exists by Lang's theorem for the connected group $\TT_r$). 

By the equivariant localization theory for schemes in Remark \ref{rem: localization other settings}, we have
\[
\rT^j (R^{\GG_r}_{\TT_r,\UU_r }(\theta))  = \rT^j (\RGamma_c(X_{\TT_r,\UU_r}; \cL_\theta)) \cong \rT^j ( \bigcup_{w \in W_x(T)^{\Fr_q}} [\dot{w}] ; \cL_\theta ) \cong  \bigoplus_{w \in W_x(T)^{\Fr_q}} \cL_\theta|_{[\dot{w}]}.
\]
Here $\cL|_{[\dot{w}]}$ is a $\TT_r(\F_q)$-equivariant sheaf on a point, i.e., a representation of $\TT_r(\F_q)$. Note that the $\Sigma$-action on the rank-one stalk of $\cL_\theta$ has $\ell$-power order and is therefore trivial in characteristic $\ell$. Writing $t\dot{w} = \dot{w} (\dot{w}^{-1} t \dot{w})$, we see that $\cL_\theta|_{[\dot{w}]} \cong \theta^w \in \Rep_k \TT_r(\F_q)$. 

\end{proof}


\subsection{Toral compact inductions}\label{ssec: toral representations}

We may regard $R^{\GG_r}_{\TT_r,\UU_r }(\theta)$ as a (derived) smooth representation of $G_{x,0}$ by inflation. 
Choose some extension $\wt{\theta}$ of $\theta$ to a character of $T(E)$, which we use to regard $R^{\GG_r}_{\TT_r,\UU_r }(\theta)$ as a (derived) smooth representation of $T(E)G_{x,0}$. Then we define 
\[
\pi_{T, U, \wt{\theta}} := \cInd_{T(E)G_{x,0}}^{G(E)}
R^{\GG_r}_{\TT_r,\UU_r }(\theta) \in D^b(\Rep^{\sm}_k G(E)).
\]

Note that since $\sigma$ is inner, Example \ref{ex: inner has equivariant structure} applies to equip any $\Pi \in D^b(\Rep^{\mrm{sm}}_k G(E))$ with a canonical $\Sigma$-equivariant structure. In particular, we use this to view $\pi_{T, U, \wt{\theta}} \in D^b(\Rep^{\mrm{sm}}_k G(E))^{B\Sigma}$. 

Let $K\subset G(E)$ be a $\Sigma$-stable open subgroup and $M$ a bounded complex of smooth $k$-representations of $K\rtimes\Sigma$. We view $\cInd_K^{G(E)}M$ as the complex of locally constant functions $f\co G(E)\rightarrow M$ with $f(gk)=k^{-1}f(g)$ and compact support modulo $K$. There is a natural action of $\Sigma$ on $\cInd_K^{G(E)}M$ by $(\sigma f)(g)=\sigma(f(\sigma^{-1}(g)))$.

\begin{lemma}[Tate cohomology of compact inductions]\label{lem: tate of compact induction}
Let $K\subset G(E)$ be a $\Sigma$-stable open subgroup and $M$ a bounded complex of smooth $k$-representations of $K\rtimes\Sigma$. Define
\[
\ker^1(\Sigma;K,G) :=\ker\bigl(\rH^1(\Sigma;K)\rightarrow\rH^1(\Sigma;G(E))\bigr).
\]
For each $\xi\in\ker^1(\Sigma;K,G)$ choose $g_\xi\in G(E)$ such that $c_\xi=g_\xi^{-1}\sigma(g_\xi)$ lies in $K$ and represents $\xi$. Let $H_\xi=H(E)\cap g_\xi Kg_\xi^{-1}$, and let $M_\xi$ be $M$ with $H_\xi$ acting through $h\mapsto g_\xi^{-1}hg_\xi$ and with $\sigma$ acting by $\sigma_\xi(m)=c_\xi\cdot\sigma(m)$. Then $\sigma_\xi$ commutes with $H_\xi$, and there is a natural isomorphism
\begin{equation}\label{eq: tate of compact induction}
\rT^n\bigl(\cInd_K^{G(E)}M\bigr)\cong\bigoplus_{\xi\in\ker^1(\Sigma;K,G)}\cInd_{H_\xi}^{H(E)}\rT^n(M_\xi).
\end{equation}
The summand of the trivial class is $\cInd_{K^\sigma}^{H(E)}\rT^n(M)$. In particular, if $\rH^1(\Sigma;K)=1$ then
\begin{equation}\label{eq: tate of cInd general}
\rT^n\bigl(\cInd_K^{G(E)}M\bigr)\cong\cInd_{K^\sigma}^{H(E)}\rT^n(M).
\end{equation}
\end{lemma}

\begin{proof}
Note that $\cInd_K^{G(E)}M$ is the direct sum of the functions supported on the union of the free $\Sigma$-orbits in $G(E)/K$ and those supported on $(G(E)/K)^\sigma$, and both summands are stable under $H(E)$ and $\Sigma$. The first summand is a direct sum of induced $k[\Sigma]$-complexes, one for each free orbit, all concentrated in the same degrees as $M$, so its Tate cohomology vanishes by \S \ref{sssec: tate other settings}.

A coset $gK$ is $\sigma$-fixed if and only if $g^{-1}\sigma(g)\in K$. Replacing $g$ by $gk$ changes this cocycle to $k^{-1}g^{-1}\sigma(g)\sigma(k)$, and two elements with the same cocycle differ by left multiplication by $G(E)^\sigma=H(E)$. Hence the map $g\mapsto[g^{-1}\sigma(g)]$ identifies $H(E)\backslash(G(E)/K)^\sigma$ with $\ker^1(\Sigma;K,G)$, and the orbit of $\xi$ is $H(E)g_\xi K/K\cong H(E)/H_\xi$. Functions supported on this orbit are identified as $H(E)$-representations with $\cInd_{H_\xi}^{H(E)}M_\xi$ via the map sending such a function $f$ to $\varphi(h)=f(hg_\xi)$. Since $\sigma^{-1}(g_\xi)=g_\xi\sigma^{-1}(c_\xi)^{-1}$ we have
\[
(\sigma\varphi)(h)=\sigma\bigl(f(hg_\xi\sigma^{-1}(c_\xi)^{-1})\bigr)=c_\xi\cdot\sigma(\varphi(h)).
\]
The cocycle condition gives $\sigma_\xi^\ell=1$, and $\sigma_\xi$ commutes with $H_\xi$ because $c_\xi\sigma(g_\xi^{-1}hg_\xi)c_\xi^{-1}=g_\xi^{-1}hg_\xi$ for $h\in H_\xi$. Now $\Sigma$ acts on $\cInd_{H_\xi}^{H(E)}M_\xi$ pointwise, and compact induction from an open subgroup is exact and commutes with direct sums, so computing with the periodic resolution gives $\rT^n(\cInd_{H_\xi}^{H(E)}M_\xi)\cong\cInd_{H_\xi}^{H(E)}\rT^n(M_\xi)$. This proves \eqref{eq: tate of compact induction}.

\end{proof}



\begin{thm}\label{thm: derived parameter}
Let $T < G$ be an elliptic unramified maximal torus.  Assume that $T(E)$ contains an element $s$ of order $\ell$, which is unramified very regular with respect to $x \in \cB(G/E)$ and strongly regular in $G_{\ad}(E)$. Then for any $r \in \Z_{\geq 0}$, any $\theta \co \TT_r(\F_q) \rightarrow k^\times$, and any extension $\wt{\theta} \co T(E) \rightarrow k^\times$ of $\theta$, $\supp_{\Exc_k(W_E, \wh{G})} (\rH^*(\pi_{T,U,\wt{\theta}}))$ contains the point of $\Spec \Exc_k(W_E, \wh{G})$ corresponding to the semisimple $L$-parameter 
\begin{equation}\label{eq: toral L-param}
W_E \xrightarrow{\ld \wt{\theta}} \ld T(k) \xrightarrow{ \ld j} \ld G(k)
\end{equation}
where $\ld \wt{\theta}$ is the $L$-parameter given by class field theory, and $\ld j$ is the canonical $L$-embedding of $T$ (\S \ref{sssec: canonical L-embedding}). 
\end{thm}

\begin{proof}
Let $\sigma \in \Aut(G)$ be the conjugation by $s$, so $G^\sigma = T$. Note that the hypotheses on $s$ imply that Proposition \ref{prop: sigma-dual embedding unram elliptic torus} holds in this situation. 

Let $K = T(E)G_{x,0}$. By the unramifiedness and ellipticity of $T$, we have $T(E)G_{x,0}=Z_G(E)G_{x,0}$, where $Z_G$ denotes the center of $G$; this is recorded in \cite[p. 1508]{CO21}, and follows from the standard equality $T(E)=Z_G(E)T_{x,0}$ for elliptic tori $T$. From the long exact sequence in non-abelian cohomology
\[
1 \rightarrow K^\sigma \rightarrow G(E)^\sigma \rightarrow  (G(E)/K)^\sigma  \rightarrow \rH^1(\Sigma; K) \rightarrow \rH^1(\Sigma; G(E))
\]
and the torsor-shifting discussion in \cite[\S I.5.3--I.5.4]{Se94}, we see that 
\[
(G(E)/K)^\sigma = \bigcup_{\xi \in \ker [ \rH^1(\Sigma; K) \rightarrow \rH^1(\Sigma; G(E))]} G(E)^\sigma/({}_\xi K)^\sigma
\]
where ${}_\xi K$ is the twist of $K$ corresponding to the torsor $\xi$. Henceforth we abbreviate 
\[
\ker^1(\Sigma; K,G) := \ker [ \rH^1(\Sigma; K) \rightarrow \rH^1(\Sigma; G(E)) ].
\]
We claim that each quotient $G(E)^\sigma/({}_\xi K)^\sigma$ is a point, which we will denote by $[\xi]$. Indeed, a coset $gK \in G(E)/K$ is fixed by $\sigma$ if and only if $s\in gKg^{-1}$, hence $s$ fixes $g\cdot x \in \cB(G/E)$. But by tame descent \cite[Ch.~12]{KP}, the fixed-point set of $s$ on the reduced building $\cB(G/E)$ is the image of $\cB(Z_G(s)^\circ/E)=\cB(T/E)$, which is the singleton $x$ because $T$ is elliptic. We deduce that $g\cdot x=x$, so $g$ normalizes $G_{x,0}$, and then 
\[
gKg^{-1}=Z_G(E)\,gG_{x,0}g^{-1}=Z_G(E)G_{x,0}=K\supset T(E).
\]
Thus the stabilizer of $gK$ in $G(E)^\sigma=T(E)$ is all of $T(E)$, proving the claim. Lemma \ref{lem: tate of compact induction}, applied to $M=R^{\GG_r}_{\TT_r,\UU_r}(\theta)$ extended to $K$ by $\wt\theta$ and with $\sigma$ acting through $s$, gives an isomorphism of $T(E)$-representations
\[
\rT^j(\pi_{T,U,\wt\theta})\cong\bigoplus_{\xi\in\ker^1(\Sigma;K,G)}\rT^j\bigl(R^{\GG_r}_{\TT_r,\UU_r}(\theta)_\xi\bigr),
\]
since $H_\xi=T(E)$ for every $\xi$. For the trivial class, $\sigma$ acts through $s$ as before, so Proposition \ref{prop: tate of deep DL} identifies this summand with a direct sum of Weyl conjugates of $\wt\theta$.

Taking $w=1$ in \eqref{eq: tate of DL induction} shows that $\rT^j(\pi_{T,U,\wt\theta})$ contains $\wt\theta$. Since Tate cohomology commutes with coefficient twists, Theorem \ref{thm: TV}, applied to $\pi_{T,U,\wt\theta}\iellt$, shows that $F_{G*}\supp_{\Exc_k(W_E,\wh G)}\rH^*(\pi_{T,U,\wt\theta}\iellt)$ contains the image under $\ld\psi$ of the parameter of $\wt\theta\iellt$. For tori, this is the parameter given by class field theory \cite[Proposition IX.6.5]{FS},

\begin{equation}\label{eq: L-param Frob twist}
W_E\xrightarrow{\ld(\wt\theta\iellt)}\ld T(k).
\end{equation}

By Proposition \ref{prop: sigma-dual embedding unram elliptic torus}, this image is
\[
\ld\psi\circ\ld(\wt\theta\iellt)
\sim\ld j\circ\Fr_\ell\circ\ld(\wt\theta\iellt)
=\ld j\circ\ld\wt\theta.
\]
The support identity in the proof of Lemma \ref{lem:frobenius-twist-FS-parameter} gives
\[
F_{G*}\supp_{\Exc_k(W_E,\wh G)}\rH^*(\pi_{T,U,\wt\theta}\iellt)
=\supp_{\Exc_k(W_E,\wh G)}\rH^*(\pi_{T,U,\wt\theta}),
\]
which proves the assertion. 
\end{proof}


\begin{cor}\label{cor: cohomology constituent L-parameter}
There is a non-zero irreducible $G(E)$-subquotient $\Pi$ of $\rH^i(\pi_{T, U, \wt{\theta}})$, for some $i$, such that $\rho_{\Pi}$ is given by \eqref{eq: toral L-param}. In particular, if $\pi_{T, U, \wt{\theta}}$ is concentrated in a single cohomological degree and is irreducible, then its Fargues-Scholze parameter is given by \eqref{eq: toral L-param}.
\end{cor}

\begin{proof}
Let $\mf m\subset \Exc_k(W_E,\wh G)$ be the maximal ideal corresponding to \eqref{eq: toral L-param}. By Theorem \ref{thm: derived parameter}, $\mf m$ lies in the support of the $\Exc_k(W_E,\wh G)$-module $M := \rH^i(\pi_{T,U,\wt\theta})$ for some $i$. It then suffices to show that there exists an actual irreducible subquotient of $M$ on which $\Exc_k(W_E,\wh G)$ acts with kernel $\mf{m}$. The representation $M$ is compactly induced from a finite-dimensional representation, hence is finitely generated. Then \cite[Corollary 3.5(1)]{DHKM} says that $M$ is admissible over $\Exc_k(W_E,\wh G)$, which implies that it has a simple subquotient with support $\mf m$, as desired.
\end{proof}

\begin{remark}
Note that Theorem \ref{thm: derived parameter} imposes no regularity conditions on $\theta$. We expect that if $\theta$ is sufficiently regular, and $\ell$ is not too small, then $\pi_{T, U, \wt{\theta}}$ should be concentrated in a single cohomological degree and irreducible, and cuspidal. There are interesting classes of examples where this is known; a notable one pertains to the \emph{depth-zero supercuspidals} studied in \cite{KV06, DR09}, which correspond to the case where: 
\begin{itemize}
\item $x \in \cB(G/E)$ is a vertex (so $G_x$ is a maximal parahoric), 
\item $r=0$, so $X_{\TT_0, \UU_0}$ is a usual Deligne-Lusztig variety up to an affine fibration,
\item $\TT_0$ is not contained in any proper $\F_q$-rational parabolic subgroup of $\GG_0$, so the Deligne-Lusztig induction is cuspidal \cite[\S 8]{DL76}, and $\theta$ is non-singular (meaning that it is not orthogonal to any coroot, cf. \cite[Definition 5.15(1)]{DL76}). 
\end{itemize}
Letting $\cL_\theta$ be the character sheaf on $X_{\TT_0, \UU_0}$ associated with the non-singular character $\theta \co \TT_0(\F_q) \rightarrow k^\times$, it follows from \cite[Lemma 9.14]{DL76} and \cite[Lemma 3.5]{Br90} that in this situation $\rH^*_c(X_{\TT_0, \UU_0}; \cL_\theta)$ concentrates in a single degree; and if the prime-to-$\ell$ component of $\theta$ is in general position, then $\rH^*_c(X_{\TT_0, \UU_0}; \cL_\theta)$ is irreducible by \cite[Corollary 3.6]{Br90}. 

\end{remark}

\bibliographystyle{amsalpha}
\bibliography{Bibliography,BibliographyV2-additions}




\end{document}